\documentclass[10pt,a4paper,leqno]{amsart}
\usepackage{amsmath}
\usepackage{amssymb}
\usepackage{amsfonts}
\usepackage{amsthm}
\usepackage{mathrsfs}
\usepackage{dsfont}
\usepackage{mathtools}
\usepackage{bm}

\usepackage[backref=page]{hyperref}
\usepackage{xcolor}

\definecolor{darkblue}{rgb}{0,0,1}
\hypersetup{colorlinks,breaklinks,
linkcolor=darkblue,urlcolor=darkblue,
anchorcolor=darkblue,citecolor=darkblue}

\newcommand*{\mailto}[1]{\href{mailto:#1}{\nolinkurl{#1}}}

\theoremstyle{plain}
\newtheorem{theorem}{Theorem}
\newtheorem{lemma}[theorem]{Lemma}

\theoremstyle{definition}
\newtheorem{definition}[theorem]{Definition}
\theoremstyle{remark}
\newtheorem{remark}[theorem]{Remark}

\numberwithin{theorem}{section}
\numberwithin{equation}{section}

\DeclareMathOperator*{\esssup}{\rm ess\,sup}

\def\mff{{\mathfrak f}}
\def\mx{{\bf x}}
\def\my{{\bf y}}

\def\Div{{\rm div}}

\def\pa{\partial}

\let\mib=\boldsymbol

\def\R{\Bbb{R}}

\def\N{\Bbb{N}}
\def\S{\Bbb{S}}

\def\eps{\varepsilon}
\def\mxi{{\mib \xi}}
\def\mx{{\mib x}}
\def\my{{\mib y}}
\def\mz{{\mib z}}
\def\mw{{\mib w}}
\def\mh{{\mib h}}

\def\malpha{{\mib \alpha}}
\def\mbeta{{\mib \beta}}

\def\mkappa{{\mib \kappa}}
\def\mlambda{{\mib \lambda}}
\def\mdelta{{\mib \delta}}
\def\momega{{\mib \omega}}

\def\mnu{{\mib \nu}}

\def\meta{{\mib \eta}}

\newcommand{\cO}{\mathcal{O}}

\newcommand{\cDp}{\mathcal{D}^\prime}

\newcommand{\cS}{\mathcal{S}}
\newcommand{\cA}{\mathcal{A}}

\newcommand{\cF}{\mathcal{F}}

\newcommand{\norm}[1]{\left \lVert#1 \right\rVert}
\newcommand{\cM}{\mathcal{M}}
\newcommand{\cL}{\mathcal{L}}

\newcommand{\abs}[1]{\left\lvert#1\right\rvert }
\newcommand{\bS}{\Bbb{S}}

\newcommand{\innb}[1]{\bigl\langle#1\bigr\rangle}

\newcommand{\loc}{\operatorname{loc}}
\newcommand{\supp}{\operatorname{supp}\,}

\newcommand{\seq}[1]{\left\{#1\right\}}

\newcommand{\sign}{\operatorname{sign}}
\newcommand{\meas}{\operatorname{meas}}
\newcommand{\En}{\mathbf{1}}

\makeatletter
\@namedef{subjclassname@2020}{%
  \textup{2020} Mathematics Subject Classification}
\makeatother

\begin{document}

\title[Heterogeneous velocity averaging]
{Regularity of velocity averages in 
kinetic equations with heterogeneity}

\author[Erceg]{M. Erceg}
\address[Marko Erceg]
{Department of Mathematics, Faculty of Science, 
University of Zagreb, 
Bijeni\v cka cesta 30, 
10000 Zagreb, Croatia}
\email{\mailto{maerceg@math.hr}}

\author[Karlsen]{K. H. Karlsen}
\address[Kenneth H. Karlsen]
{Department of Mathematics, 
University of Oslo, 
NO-0316 Oslo, Norway}
\email{\mailto{kennethk@math.uio.no}}

\author[Mitrovi\'{c}]{D. Mitrovi\'{c}}
\address[Darko Mitrovi\'{c}]
{University of Montenegro, 
Faculty of Mathematics, 
Cetinjski put bb,
81000 Podgorica, 
Montenegro}
\email{\mailto{darkom@ucg.ac.me}}

% 35B65 - Smoothness and regularity of solutions to PDEs
% 35L65	- Hyperbolic conservation laws
% 35Q83 - Vlasov equations
% 42B37 - Harmonic analysis and PDEs


\begin{abstract}
This study investigates the regularity of 
kinetic equations with spatial heterogeneity.
Recent progress has shown that velocity averages of 
weak solutions $h$ in $L^p$ ($p>1$) are strongly 
$L^1_{\text{loc}}$ compact under 
the natural non-degeneracy condition. 
We establish regularity estimates for equations 
with an $\boldsymbol{x}$-dependent drift vector 
$\mathfrak{f} = \mathfrak{f}(\boldsymbol{x}, \boldsymbol{\lambda})$,
which satisfies a quantitative version of the 
non-degeneracy condition. We prove that $(t,\boldsymbol{x})
\mapsto \int  h(t,\boldsymbol{x},\boldsymbol{\lambda}) 
\rho(\boldsymbol{\lambda}) 
\, d\boldsymbol{\lambda}$, for any sufficiently 
regular $\rho(\cdot)$, belongs to the fractional Sobolev space 
$W_{\text{loc}}^{\beta,r}$, for some regularity $\beta\in (0,1)$ 
and integrability $r \geq 1$ exponents. 
While such estimates have long been known for 
$\boldsymbol{x}$-independent drift vectors 
$\mathfrak{f}=\mathfrak{f}(\boldsymbol{\lambda})$, 
this is the first quantitative regularity estimate 
in a general heterogeneous $L^p$ setting, 
valid for all $p>1$ and allowing measure-valued 
source terms with derivatives of any order in the velocity variable. 
As applications, we obtain (i) new regularity 
estimates for heterogeneous 
conservation laws with nonlinear 
discontinuous flux and bounded 
initial data, and (ii) regularity of the density 
variable in isentropic gas dynamics under a 
regularity assumption on the velocity field.
\end{abstract}

\date{\today}

\maketitle

\tableofcontents
\clearpage

\section{Introduction}\label{sec:intro}
Kinetic equations describe the statistical 
behavior of a large number of particles or agents, 
capturing the evolution of their distribution functions 
over time and space. These equations, including the Boltzmann 
and Vlasov equations, are pivotal in modeling particle 
interactions through collisions or external forces 
\cite{Cercignani:1988aa,Vasseur:2008aa,Villani:2002aa}. 
In the context of hyperbolic conservation laws, kinetic formulations 
\cite{Lions:1994qy,Perthame:2002qy} offer a robust 
framework for understanding the propagation 
of shock waves, elucidating the macroscopic behavior 
of hyperbolic systems from a microscopic perspective.

So-called velocity averaging lemmas are fundamental tools 
in the analysis of kinetic equations, providing insights into the regularity 
and compactness properties of macroscopic observables. 
They assert that, under certain conditions, the averages 
of solutions with respect to the velocity variable exhibit 
improved regularity---quantified in terms of fractional Sobolev or Besov 
spaces---compared to the original solutions. Velocity averaging lemmas 
have a long and rich history, dating back to the 
foundational work of Agoshkov \cite{Agoshkov:84} 
and Golse, Lions, Perthame, and Sentis \cite{Golse-etal:88} 
in the $L^2$ setting. They were later extended 
to general $L^p$ spaces ($p>1$) in \cite{Bezard:94} 
and \cite{DiPerna-etal:91}. This was followed by 
the optimal Besov space regularity proved in \cite{DeVorePetrova:01}. 
The study of averaging lemmas under various structural and regularity 
conditions has been pursued in numerous works 
\cite{Arsenio:2011aa,ArsenioMasmoudi:19,ArsenioSaint-Raymond:11, 
GolseSaint-Raymond:02, Jabin:2022aa, JabinVega:03, JabinVega:04, 
PerthameSouganidis:98, Tadmor:2006vn, Westdickenberg:02} 
(this list is by no means exhaustive), employing a wide range of 
analytical techniques, including wavelet decompositions, 
real-space approaches in time and space-time, Radon 
and $X$-transforms, duality-based dispersion 
estimates, and energy methods.

Most research in this area has traditionally focused on 
the homogeneous case, where the coefficients of the equations 
are independent of the spatial variable $\mx$. However, the 
heterogeneous case, where the drift vector  
$\mff=\mff(\mx,\mlambda)$ depends on both the space 
($\mx$) and velocity ($\mlambda$) variables, is equally 
important due to its relevance in more complex applications, 
where environmental variations significantly impact particle dynamics.

When the coefficients depend on the spatial location $\mx$, 
the primary method for deriving velocity averaging results is through 
the use of microlocal defect measures that account for oscillations. 
These measures were independently introduced by 
G\'erard \cite{Gerard:91} and Tartar
\cite{Tartar:1990mq}.  
In this context, the velocity averaging results manifest as 
strong compactness (convergence) statements about 
sequences of weak solutions. However, they have not yet 
provided the quantitative compactness (regularity) estimates seen in the 
original velocity averaging lemmas. Nonetheless, these results 
are highly robust with respect to the $\mx$-regularity of the 
coefficients, accommodating even discontinuities.  
For some representative works, see \cite{Panov:1994gf, Panov:2010aa} 
for $L^\infty$ solutions (related to hyperbolic conservation laws) 
and \cite{ Erceg:2023aa,Lazar:2012aa,LazarMitrovic:16} for $L^p$ 
solutions under the essential condition that $p \ge 2$. 
The case where $1<p < 2$, which combines low integrability of the 
solutions with heterogeneities, remained an open problem since 
\cite[page 1764]{Gerard:91}, until it was addressed 
recently in \cite{Erceg:2023ab}.

Despite the progress since \cite{Gerard:91}, the 
problem of obtaining regularity estimates for velocity 
averages has remained open in the heterogeneous 
setting at the level of generality considered here. 
To the best of our knowledge, the only available results 
in this direction are those of \cite{Gerard1990,Gerard:1992aa}. 
However, when reduced to the first-order transport equation, 
their approach differs substantially from ours and does not 
extend to the framework of hyperbolic conservation 
laws---a point we will comment on later in this introduction.

We investigate the regularity 
question for kinetic equations of the form
\begin{equation}\label{eq-1}
	\pa_{t}h+\Div_\mx \bigl(\mff(\mx,\mlambda) h\bigr)
	= \pa_{\mlambda}^\mkappa \gamma, 
	\quad 
	t\in \R, \, \, \mx \in \R^d, 
	\,\, \mlambda \in \R^m,
\end{equation} 
where $h=h(t,\mx,\mlambda)$, $\mkappa \in \N_0^m$, and $m\in \N$.
Since initial data
hold no significance in our analysis, 
we have assumed that the variable $t$ takes on all real values, 
rather than the more conventional range of $\R_+:=(0,\infty)$.

By setting $\mx_0 = (x_0, \mx)
=(t,\mx)\in \R^D$ with $D = 1 + d \ge 2$, 
we can write this equation more compactly as:
\begin{equation}\label{eq-1-new}
	\Div_{\mx_0} \bigl( (1,\mff(\mx,\mlambda))h\bigr)
	= \pa_{\mlambda}^\mkappa \gamma, 
	\quad 
	\mx_0\in \R^D, 
	\,\,  \mlambda \in \R^m,
\end{equation} 
where $h=h(\mx_0,\mlambda)$. We will
make the following standing assumptions about the 
solution and ``data" of the equation \eqref{eq-1}:
\begin{itemize}
	\item[({\bf H1})] $h \in L^p_{\loc}(\R^{D+m})$, $p>1$;

	\item[({\bf H2})] (i) For almost every 
	$\mx \in \R^{D-1}$, the mapping 
	$\mlambda \mapsto \mff(\mx, \mlambda)$ belongs to 
	$C^{\mkappa}(\R^m; \R^{D-1})$. Moreover, this mapping is 
	uniformly locally bounded with respect to $\mx$. That is, for every 
	$\malpha \in \N_0^m$ with $\malpha \le \mkappa$, 
	the partial derivative $\partial_{\mlambda}^{\malpha} \mff$ exists 
	and remains bounded on compact subsets of $\R^{D-1} \times \R^m$.
	
	\medskip
	
	\noindent (ii) There exist exponents 
	$\bar s \in (0,1)$ and $\bar p > p'$ 
	(where $1 = 1/p + 1/p'$ and $p$ is the 
	exponent appearing in ({\bf H1})) 
	such that, for any open and bounded subsets 
	$K \subset \R^{D-1}$ and $L \subset \R^m$, 
	$$
	\sup_{\mlambda \in L}  
	\norm{\mff(\cdot, \mlambda)}_{W^{\bar s, \bar p}(K)}
	<\infty,
	$$
	where $W^{\bar s, \bar p}(K)$ denotes the 
	fractional Sobolev space characterized 
	by its regularity exponent $\bar s$ and integrability 
	exponent $\bar p$ (see Section \ref{sec:prelim}).

	\item[({\bf H3})] $\gamma=\gamma(\mx_0,\mlambda)$ 
	is a Radon measure, $\gamma \in \cM_{\loc}(\R^D\times \R^m)$.
\end{itemize}

Throughout the paper, we work under the following 
quantitative non-degeneracy condition on 
the drift $\mff=(\mff_1,\ldots,\mff_d)$: 
there exists an exponent $\alpha>0$ such that
\begin{align}\label{non-deg}
	\esssup\limits_{\mx\in K}
	\sup\limits_{(\xi_0,\xi_1,\ldots,\xi_d)\in \bS^d} 
	\meas\seq{\mlambda\in L: 
	\abs{\xi_0+\sum\limits_{k=1}^d
	\xi_k \mff_k(\mx,\mlambda)}<\nu}
	\mathrel{\lesssim} \nu^{\alpha},
\end{align} 
for any $\nu>0$, $K \subset\subset \R^{D-1}$, and 
$L\subset\subset \R^m$.
The implicit constant
may depend on $K$ and $L$, but is independent of $\nu$.

Our main result concerns the regularity of the 
$\mlambda$-averages 
\begin{equation*}
	\innb{h,\rho}(\mx_0)
	:=\int_{\R^m} h(\mx_0,\mlambda) \rho(\mlambda)
	\, d\mlambda, 
	\quad 
	\rho\in C^{\abs{\mkappa}}_c(\R^m),
\end{equation*}
where $h$ is a weak solution to \eqref{eq-1-new}.
More precisely, we will prove that 
\begin{equation}\label{aver-quant}
	\innb{h,\rho}
	\in W^{\beta,r}_{\loc}(\R^D), 
\end{equation} 
for some $\beta\in(0,1)$ and $r\in\bigl(1,\min\{p,2\}\bigr)$. 
The two exponents $\beta$ and $r$ are determined explicitly 
when $p\ge 2$, while for $p<2$ we provide 
an explicit lower bound for $\beta$. 
This complexity arises because their determination 
requires solving a system of inequalities involving 
several free parameters intrinsic to our method of proof. 
In addition, the resulting exponents 
depend on the parameters appearing in 
assumptions ({\bf H1})--({\bf H3}) and in 
\eqref{non-deg} (see Subsection \ref{subsec:final} and 
Remark \ref{rem:reg-exp} for details). 

While we make no claim of optimality---indeed, our 
exponents are considerably lower than those predicted by the 
homogeneous theory---our analysis establishes 
a significant new regularity result for velocity averages, 
substantially extending the scope of the available theory. 
In particular, it resolves an open problem that 
goes well beyond the original question 
posed in \cite{Gerard:91} concerning the 
$L^p$ compactness ($p \neq 2$) of velocity averages. 
Possible refinements of the exponents 
are left for future work.

We continue the historical overview with 
a more detailed comparison between 
our result and the few existing contributions that 
incorporate some heterogeneity. The theory developed 
in \cite{Gerard1990,Gerard:1992aa} considers 
pseudodifferential operators and formulates the regularity 
microlocally, that is, by analyzing specific 
directions $\mxi_0$ at each point $\mx_0$. However, when 
restricted to regularity theory for the 
specific equation \eqref{eq-1} in the present 
work, our approach yields several substantial 
improvements. To our knowledge, each of these advances 
poses significant obstacles to being achieved with 
the techniques of \cite{Gerard1990,Gerard:1992aa}. 
First, our analysis is carried out in the full 
$L^p$ framework for arbitrary $p>1$, in sharp contrast 
to the $L^2$ setting adopted 
in \cite{Gerard1990,Gerard:1992aa}.  
Second, our theory allows for 
irregular source terms, including Radon measures 
and $\mlambda$-derivatives of Radon measures. 
This constitutes a key generalization beyond the setting of 
\cite{Gerard1990,Gerard:1992aa}, where the source term 
is confined to a negative Sobolev space $H^{-\sigma}$ 
for some $\sigma < 1$. Since Radon measures 
are not contained in $H^{-\sigma}$ for $\sigma < 1$ in 
several dimensions, our approach enables the 
analysis of physically relevant models---such as 
scalar conservation laws with
heterogeneity, isentropic 
gas dynamics, and heterogeneous Boltzmann-type equations 
on manifolds—that remain inaccessible within 
the earlier theory. Third, we allow for 
spatial discontinuities in the drift vector, 
whereas the analysis in \cite{Gerard1990,Gerard:1992aa} 
requires smooth coefficients, as is typically 
the case in pseudodifferential settings. 
Finally, with regard to the non-degeneracy condition, 
our assumption \eqref{non-deg} coincides with condition 
(H$_\delta$) from \cite{Gerard1990}. 
By contrast, the analysis in \cite{Gerard:1992aa} employs a stronger 
transversality condition which, although yielding a sharper result 
than \cite{Gerard1990}, is considerably more restrictive, 
as it requires the $\mlambda$-derivative of the symbol 
to be nonvanishing at every point of the characteristic manifold.

\begin{remark}\label{rem:H2}
A straightforward case in which assumption ({\bf H2}) 
is satisfied arises when 
$\mff \in C^{1,\abs{\mkappa}}(\R^{D-1} \times \R^m; \R^{D-1})$, 
that is, when $\mff$ is a $(D-1)$-dimensional vector field 
that is $C^1$ in $\mx$ and $C^{\abs{\mkappa}}$ in the 
velocity variable $\mlambda$. 
In this smooth setting, it is easy to verify that one may take 
$\bar{s}=1$ and $\bar{p}=\infty$, which simplifies the analysis 
in the final part of the proof of \eqref{aver-quant} 
(see Subsection \ref{subsec:final}). 

More importantly, however, assumption ({\bf H2}) 
also accommodates drift vectors $\mff=\mff(\mx,\mlambda)$ 
that are discontinuous with respect to $\mx$. 
In particular, drift vectors with $BV$ (bounded variation)
regularity in $\mx$ are admissible (see 
Lemma \ref{lem:BV_H2}).
\end{remark}

\begin{remark}
We require that condition \eqref{non-deg} holds 
for every compact subset $L \subset\subset \R^m$. 
This choice ensures that, in~\eqref{aver-quant}, 
the test function $\rho$ may be taken from 
$C^{\abs{\mkappa}}_c(\R^m)$. 
Conversely, if one wishes to consider only 
test functions $\rho \in C^{\abs{\mkappa}}_c(L)$ 
for a fixed $L \subset\R^m$, 
it is sufficient to assume that~\eqref{non-deg} 
holds for that particular set~$L$. 
An analogous remark applies to assumption~({\bf H2}). 
This observation becomes relevant when dealing with 
bounded entropy solutions of heterogeneous 
conservation laws (see Section \ref{sec:appl}).
\end{remark}

Let us briefly outline some of the key ideas behind 
the proof of \eqref{aver-quant}. When the drift term $\mff = \mff(\mlambda)$ 
is independent of $\mx$, the Fourier transform is classically 
applied to the equation \eqref{eq-1} to separate the solution $h$ from 
the coefficients (see, e.g., \cite[Section 5.3]{Perthame:2002qy}). 
Specifically, we obtain
\begin{equation}\label{eq:hom-drift}
	2\pi i \, \widehat{h}(\mxi_0, \mlambda) 
	= \frac{\pa_{\mlambda}^\mkappa \widehat{\gamma}
	(\mxi_0, \mlambda)}{\abs{\mxi_0}\left(\xi'_0 + \mff \cdot \mxi'\right)}, 
	\qquad \mxi_0' \in \bS^{D-1},
\end{equation}
where $\widehat{u} = \cF(u)$ denotes the Fourier transform 
of a function $u: \R^D \to \R$, and $\cF^{-1}$ denotes the inverse 
transform. Occasionally, we use subscripts 
like $\cF_{\mx_0}$ and $\cF_{\xi_0}^{-1}$ 
to indicate the specific independent variable of the function 
being transformed. Moreover, for a given 
$\mxi_0=(\xi_0,\mxi)\in \R^D$, we 
denote by
\begin{equation}\label{eq:mxi-prime}
	\mxi_0'= \frac{\mxi_0}{\abs{\mxi_0}}, 
	\qquad \mxi_0'=(\xi_0',\mxi'),
\end{equation} 
the corresponding projection on the unit 
sphere $\bS^{D-1}\subset \R^D$.

Given \eqref{eq:hom-drift}, one can separate the regions where 
the symbol $\cL = \xi'_0 + \mff \cdot \mxi'$ is near zero 
from those where it is not. One then
analyzes the behavior of
the solution $h$ in these distinct regions separately. 
In the region where the symbol is away from zero, 
the term $\frac{1}{\abs{\mxi_0}}$ 
on the right-hand side of \eqref{eq:hom-drift} 
ensures regularity.  Conversely, in the region where the 
symbol approaches zero, the non-degeneracy 
condition \eqref{non-deg} is 
used to obtain the required estimates 
(see, e.g., \cite{Lions:1994qy, Tadmor:2006vn}).

In contrast, when dealing with a heterogeneous 
drift $\mff=\mff(\mx, \mlambda)$, the direct application 
of the Fourier transform to \eqref{eq-1-new} yields:
\begin{equation*}
	2\pi i \cF_{\mx_0}\Bigl(\bigl(1, \mff(\cdot_\mx, \mlambda)\bigr) 
	h(\cdot_{\mx_0}, \mlambda)\Bigr)(\mxi_0) \cdot \mxi_0 
	= \pa_\mlambda^\mkappa \cF_{\mx_0}
	\Bigl(\gamma\bigl(\cdot_{\mx_0}, \mlambda\bigr)\Bigr).
\end{equation*}
A natural idea is to employ a commutation lemma to shift 
the $\mx$-dependent vector $1/\mff$ from inside to 
outside the Fourier transformation and 
to get an expression of the form:
\begin{align*}
	&\Biggl\|\cF^{-1}_{\mxi_0} \Biggl[
	\int_{\R^m} \Biggl(\frac{1}
	{\bigl(1, \mff(\cdot_\mx, \mlambda)\bigr) \cdot \mxi_0'}
	\cF_{\mx_0}\Bigl[\rho(\mlambda)
	\bigl(1, \mff(\cdot_\mx, \mlambda)\bigr) 
	\cdot \mxi_0' \, h(\cdot_{\mx_0},\mlambda)\Bigr]
	\\ & \quad\qquad \quad\quad
	- \cF_{\mx_0}\Bigl[\rho(\mlambda)
	\, h(\cdot_{\mx_0},\mlambda)\Bigr] 
	\Biggr)\, d\mlambda \Biggr]
	\Biggr\|_{W^{s,q}(\R^D)}
	\\ & \;
	=\Biggl\|\cF^{-1}_{\mxi_0} \Biggl[
	\int_{\R^m} \Biggl(\frac{1}
	{\bigl(1, \mff(\cdot_\mx, \mlambda)\bigr) \cdot \mxi_0'}
	\cF_{\mx_0}\Bigl[\rho(\mlambda)
	\bigl(1, \mff(\cdot_\mx, \mlambda)\bigr)
	\cdot \mxi_0' \, h(\cdot_{\mx_0},\mlambda)\Bigr]
	\\ & \quad\qquad 
	-\cF_{\mx_0}\left[\frac{1}
	{\bigl(1, \mff(\cdot_\mx, \mlambda)\bigr) \cdot \mxi_0'} 
	\rho(\mlambda)
	\bigl(1, \mff(\cdot_\mx, \mlambda)\bigr) \cdot \mxi_0'
	\, h(\cdot_{\mx_0},\mlambda)\right] \Biggr)
	\, d\mlambda \Biggr] \Biggr\|_{W^{s,q}(\R^D)} \lesssim 1.
\end{align*}
This approach, which seeks to decouple 
the coefficients of the equation from the 
Fourier transform of the solution, parallels the techniques 
employed in the homogeneous setting. 
A related idea appears in \cite{Gerard1990}, 
where Proposition 3.3---concerned with the 
commutation between the symbol 
and the Fourier-Bros-Iagolnitzer (FBI) transform---constitutes 
a central step in the proof.  To the best of our 
knowledge, however, extending this methodology 
to the level of generality addressed in the present work 
poses insurmountable technical difficulties.

We mention here that the notation $A \lesssim B$ signifies 
that $A \leq C B$ for some constant $C > 0$, which remains 
independent of the primary parameters. For greater clarity, when 
necessary, we will use $\lesssim_g$ to indicate that 
the implicit constant depends on $g$. 
We write $A \sim B$ to denote that 
$A \lesssim B$ and $B \lesssim A$.

A key idea of this paper is to adjust 
the aforementioned strategy by 
localizing the kinetic equation
around a fixed point $\my\in \R^{D-1}$. 
Subsequently, we integrate 
the derived estimates over $\my$. 
This method---which may be regarded as a ``non-asymptotic" 
blow-up approach---overcomes the previously mentioned 
difficulty and avoids the necessity of intricate 
estimates involving
pseudodifferential operators.
Let us elaborate on the method. First, simple 
``Littlewood-Paley" considerations (see 
Section \ref{sec:prelim})
reveal that, to establish 
\eqref{aver-quant}, it is sufficient to show the 
existence of exponents $r> 1$ and
$\beta_0 > 0$ ($\beta_0>\beta$) 
such that
\begin{equation}\label{eq:intro-goal}
	\norm{\cA_{\phi_j} 
	\left(\innb{h, \rho}\right)}_{L^r(\R^D)} 
	\mathrel{\lesssim} 2^{-j \beta_0}, \quad j \in \N_0,
\end{equation}
where $\cA_{\phi_j}(\cdot)$ 
is the Fourier multiplier operator with 
the symbol (multiplier) 
$\phi_j(\abs{\mxi_0})$,
and $\phi_j(\cdot)$ is a smooth $\R$-valued 
function supported in an annulus of size $2^j$.

Throughout this paper, we will assume, without loss 
of generality, that the solution $h$ and 
the source $\gamma$ are compactly supported 
on $K_0\times L$, where
\begin{equation}\label{eq:K0-def}
	K_0 = [0, T] \times K \subset \R^D, 
	\quad 
	T > 0, 
	\quad 
	K \subset\subset \R^{D-1}, 
	\quad L \subset\subset \R^m.
\end{equation}
For practical reasons, we will often implicitly extend
compactly supported functions by zero to the whole space.
We will deduce \eqref{eq:intro-goal} via duality by proving that, for 
any $\varphi = \varphi(\mx) \in C^1_c(K)$ and 
$v = v(\mx_0) \in C^\infty_c(\R^D)$,
$$
\abs{I_j(v; \varphi)} \lesssim 2^{-\beta_0 j} 
 \norm{v}_{L^{r'}(\R^D)}, 
\quad \text{with } \frac{1}{r} + \frac{1}{r'} = 1,
$$
where
$$
I_j(v; \varphi) := \int_{\R^D} 
\cA_{\phi_j}
\left(\int_{\R^m} \rho(\mlambda) \varphi(\cdot_{\mx}) 
h(\cdot_{\mx_0}, \mlambda) 
\, d\mlambda \right)\!(\mx_0) v(\mx_0) \, d\mx_0.
$$
Unfortunately, this is difficult to achieve directly from the 
heterogeneous equation
\eqref{eq-1-new} for reasons outlined above. 
As a key technical innovation (in what we 
might call a non-asymptotic blow-up 
style of argument), we instead consider the term
\begin{align*}
	I_{j,\eps}(v; \varphi) 
	& := \int_{K} \int_{\R^D} \cA_{\phi_j} 
	\left(\int_{\R^m} \rho(\mlambda)
	\chi^{j, \eps}_{\my}(\cdot_{\mx_0})
	h(\cdot_{\mx_0}, \mlambda) \, d\mlambda \right)\!(\mx_0)
	\\ & \qquad \qquad \qquad \qquad \qquad \qquad 
	\qquad \quad \times 
	v(\mx_0) \varphi(\my) \, d\mx_0 d\my,
\end{align*}
where $\chi^{j, \eps}_\my(\mx_0)=\chi^{j, \eps}_\my(x_0,\mx)$ 
behaves like a Gaussian mollifier in $\mx\in \R^{D-1}$ with 
scaling parameter $2^{-j \eps}$, $\eps > 0$. The error between the two 
quantities is bounded by 
$$
\abs{I_j(v; \varphi) - I_{j,\eps}(v; \varphi)} 
\lesssim 2^{-j\eps} \norm{\varphi}_{C^1(K)} 
\norm{h}_{L^r(K_0 \times L)} \norm{v}_{L^{r'}(\R^D)},
$$
where the crucial asymptotic decay in $j$ becomes evident.

Therefore, the main part of the proof involves estimating the 
term $I_{j,\eps}(v; \varphi)$. This estimation can now be derived using 
the kinetic equation \eqref{eq-1-new}, where 
the non-asymptotic blow-up method allows 
us to view the equation (in $\mx$) as having 
a homogeneous drift 
$\mff_{\my}(\mlambda) = \mff(\my, \mlambda)$, 
where $\my$ can be considered as ``frozen" 
in crucial parts of the proof. 
While this approach addresses the previously 
mentioned issue, it introduces several 
delicate terms that require careful handling. Moreover, we must 
account for the potentially limited integrability of 
the solution, as assumed in ({\bf H1}). 
The interpolation argument, a standard tool in regularity theory (cf.~\cite{Erceg:2025aa,Gess:2018ab,Tadmor:2006vn}), 
is not required in our approach.

The final estimate for $I_{j,\eps}(v; \varphi)$ involves 
five adjustable parameters: the previously introduced $r$ and $\eps$, 
along with three additional ones: $\epsilon$, $\zeta$, and $\sigma$. 
The parameter $\epsilon$ controls a specific
regularization of key
symbols in Fourier multiplier operators, 
$\zeta$ determines the magnitude 
of the symbol of the kinetic equation, and $\sigma$ governs the 
minimal integrability of $h$, where $h \in L^p$ for some $p$ 
potentially close to 1. 
These five parameters must be delicately linked and internally 
balanced to conclude that \eqref{aver-quant} 
holds for appropriate values of $\beta$ and $r$.

Once the key ideas are in place, deriving the 
estimates requires only a few tools, such as 
the Marcinkiewicz multiplier theorem and a tailored 
commutation estimate for multiplier operators.
Additionally, we rely on some results demonstrating 
that multiplier operators with a symbol of the form
$$
\frac{\delta \abs{\mxi'}^2}{\abs{\xi_0+F(\mlambda)\cdot \mxi'}^2 
+ \delta \abs{\mxi'}^2}, \quad \delta > 0,
$$
are bounded on $L^q$ spaces for $q > 1$, independently 
of $\mlambda$, $\delta$, and $F$, provided $F$ is locally integrable. 
For further details, we refer to Section \ref{sec:main}.

\medskip

It turns out that the proof techniques discussed above are 
sufficiently robust, with suitable modifications, to address the 
question of regularity of velocity averages for a class of 
stochastic kinetic equations. These equations are structurally similar 
to the deterministic ones but include 
an It{\^o} stochastic forcing term $\Phi\, dW$.  
Stochastic perturbations (and the inclusion of 
degenerate second-order operators) will be 
discussed elsewhere.
 
%%%%%%%%%%%%%%%%%%%%%%%%%%%
%%%%%%%%%%%%%%%%%%%%%%%%%%%
\section{Preliminaries and multiplier estimates}\label{sec:prelim}
In this section, we establish the necessary notations, 
outline some preliminary functional analytic 
considerations, and prove 
an essential commutator estimate for 
Fourier multiplier operators.

Recall that the fractional Sobolev 
(or Sobolev-Slobodetskii) space $W^{s,q}(\Omega)$, 
with $s\in(0,1)$, $q\in [1,\infty)$, and $\Omega\subseteq\R^D$
open, is defined as the space of $u \in L^q(\Omega)$ such that 
(cf.~\cite[Definition 2.5.16]{Hytonen:2016aa})
$$
\norm{u}_{W^{s,q}(\Omega)}^q
= \norm{u}_{L^{q}(\Omega)}^q 
+\int_{\Omega} \int_{\Omega} 
\frac{\abs{u(\mx_0) - u(\my_0)}^q}
{\abs{\mx_0 - \my_0}^{D + sq}} 
\, d\mx_0 \, d\my_0 <\infty.
$$

When $\Omega$ is bounded and sufficiently regular 
(e.g., open and convex, ensuring the 
existence of an extension operator), 
we have the continuous embedding 
$W^{s,q}(\Omega) \hookrightarrow W^{s_1,q_1}(\Omega)$ 
for any $s_1 \in (0,s)$ and $q_1 \in (1,q)$ 
(see, e.g., \cite[Theorem~3.3]{Visintin:2020aa} 
and Remark~\ref{rem:Sob_Omega_bdd} below). 
This embedding will be used in Subsection~\ref{subsec:R3} 
when treating part~(ii) of assumption~({\bf H2}). 
Note, however, that the inclusion fails for $s_1 = s$ 
\cite{Mironescu:2015aa}: for fixed $q_1 < q$, 
the constant in the norm inequality blows up as $s_1 \to s$.

For regular domains $\Omega \subseteq \R^D$, 
the fractional Sobolev space $W^{s,q}(\Omega)$ 
can be characterized as the restriction to $\Omega$ 
of functions in $W^{s,q}(\R^D)$, 
which necessitates the use of an extension operator. 
Apart from assumption~({\bf H2}), however, 
we will primarily be concerned with the case 
where the domain is the whole space, $\Omega = \R^D$. 
Accordingly, we shall restrict our attention 
to this setting for the remainder of the discussion.

As an intermediate step in our analysis, we will require 
more general Besov spaces $B^s_{q,r}(\R^D)$. 
However, we will introduce these spaces only 
after establishing some necessary notations.
The Fourier transform of a function $u:\R^D\to \R$, 
$u=u(\mx_0)$, is denoted by 
$\cF(u)(\mxi_0) = \widehat{u}(\mxi_0) 
= \int_{\R^D} e^{-2\pi i \mx_0 \cdot \mxi_0} 
u(\mx_0) \, d\mx_0$, and $\cF^{-1}$ represents the inverse 
Fourier transform. In this paper, we
utilize only basic
properties of the Fourier transform, as gathered, 
for example, in \cite[Theorem 2.2.14]{Grafakos:2014aa}.

Let $\cA_{\psi}$ denote the Fourier multiplier operator 
with symbol $\psi=\psi(\mxi_0)$, i.e.,
\begin{equation}\label{eq:A-psi-def}
	\cF\left(\cA_{\psi}(u)\right)\!(\mxi_0) 
	= \psi(\mxi_0)\cF(u)(\mxi_0).
\end{equation}
For important properties of these operators, we refer to
\cite[Section 2.5]{Grafakos:2014aa}.

Let us fix an inhomogeneous Littlewood-Paley sequence
$\seq{\phi_j}_{j \in \N_0}$, which forms a dyadic partition 
of unity on $\R$, and, via radial extension, on $\R^D$, 
with the usual properties. Namely, each $\phi_j=\phi_j(\xi)$
is compactly supported, smooth, and, for $j\geq 1$,
localized in
dyadic annuli of the form $2^{j-1} \leq \abs{\xi} \leq 2^{j+1}$.
More precisely, we will use functions 
that have the following properties:
\begin{equation}\label{eq:phi-j-def}
	\begin{split}
		& 0\leq \phi \in C_c^\infty(\R),  \quad 
		\supp(\phi) \subseteq \left(2^{-1},2\right), \quad 
		\sqrt{\phi} \in C_c^\infty(\R),
		\\ & 
		0\leq  \phi_0\in C_c^\infty(\R), \quad 
		\supp(\phi_0)\subseteq (-2,2),
		\\ &
		\phi_j(\xi):=\phi\left(2^{-j}\xi\right), 
		\quad j\in \N, \quad
		\sum_{j\ge 0}\phi_j(\xi)=1, 
		\quad \xi\in [0,\infty).
	\end{split}
\end{equation} 
The existence of $\phi$ and $\phi_0$ can 
be established by following the approach outlined 
in \cite[Exercise 6.1.1]{Grafakos:2014aa} (see also 
\cite[Section 14.2.c]{Hytonen:2023aa}). It suffices to consider 
the restriction $\xi \in [0,\infty)$, as we will later 
work with radial symbols of the form 
$\phi_j(\abs{\mxi_0})$.

The functions $\phi_j$ often act as symbols for Fourier 
multiplier operators. Therefore, it is useful 
to revisit the essential continuity properties 
of these operators. Since $\mxi_0\mapsto\phi(\abs{\mxi_0})$ 
is smooth and compactly supported,
$\cA_{\phi(\abs{\cdot})}$ is continuous from $L^q(\R^D)$ 
to itself for any $q\in [1,\infty]$ 
(see \cite[Lemma 14.2.11]{Hytonen:2023aa} 
and \cite[Section 2.5.4]{Grafakos:2014aa}).
Moreover, the same holds for 
$\cA_{\phi_j(\abs{\cdot})}$, $j\in\N$, while by 
\cite[Proposition 2.5.14]{Grafakos:2014aa} the 
corresponding norms are independent of $j$. 
On the other hand, at first glance, it may appear that 
the function $\mxi_0\mapsto\phi_0(\abs{\mxi_0})$ has 
a singularity at the origin. 
However, by \eqref{eq:phi-j-def}, $\phi_0\equiv 1$ on $[0,1]$. 
Hence, it is smooth and compactly supported. 
Thus, the corresponding multiplier operator 
exhibits the same continuity properties as 
those mentioned for the other functions $\phi_j$.
In order to simplify the notation, we will use 
$\cA_{\phi_j}$ to denote the multiplier operator 
with the symbol $\mxi_0\mapsto \phi_j(\abs{\mxi_0})$. 

For $s\in \R$ and $q,r\in [1,\infty)$, the Besov space 
$B^s_{q,r}(\R^D)$ is the collection 
of all tempered distributions $u$ on $\R^D$ such that 
for any $j\geq 0$ we have
$\cA_{\phi_j}(u)\in L^q(\R^{D})$
and the quantity 
\begin{equation*}
	\norm{u}_{B^s_{q,r}(\R^D)} 
	:= \biggl(\, \sum_{j\geq 0} 2^{jsr}
	\norm{\cA_{\phi_j}(u)}_{L^q(\R^D)}^r\biggr)^{1/r}
\end{equation*}
is finite (see, e.g., \cite[Section 14.4]{Hytonen:2023aa}).

For any function space defined on 
an open set $\Omega\subset\R^D$, its local counterpart 
consists of all tempered distributions $u$ 
such that $\varphi u$ lies in the original 
space for every $\varphi\in C^\infty_c(\Omega)$.

If we can establish the estimate \eqref{eq:intro-goal} 
for $u$ (with $u$ in place of $\innb{h,\rho}$)
it follows that $u$ belongs to the space 
$B^\beta_{r,1}(\R^D)$ for every $\beta < \beta_0$. 
Using the embeddings
\begin{equation}\label{eq:besov_emb}
	B^\beta_{r,1}(\R^D) \hookrightarrow 
	B^\beta_{r,r}(\R^D) 
	= W^{\beta,r}(\R^D),
\end{equation}
for which we refer to 
\cite[Corollary 14.4.25 and (14.22)]{Hytonen:2023aa}, 
we reach the desired conclusion that $u\in W^{\beta,r}(\R^D)$ 
for any $\beta\in (0,\beta_0)$. 
Moreover, assuming $u \in L^r(\R^D)$, the $L^r$-norm of 
each $\cA_{\phi_j}(u)$ is evidently finite, as 
discussed earlier concerning $\cA_{\phi_j}$. 
Consequently, our attention can be directed toward 
establishing \eqref{eq:intro-goal} for sufficiently large $j$. 
Specifically, for a fixed $j_0 \in \N$, we restrict 
our consideration to $j \geq j_0$.

Summarizing, if for some $r\in [1,\infty)$ 
we have $u \in L^r(\R^D)$ 
and for some $\beta_0>0$ and $j_0\in\N$ the estimate
\begin{equation}\label{eq:besov_sp_u}
	\norm{\cA_{\phi_j}(u)}_{L^r(\R^D)}
	\mathrel{\lesssim} 2^{-j\beta_0}, \quad j\geq j_0,
\end{equation}
holds, then $u\in W^{\beta,r}(\R^D)$ 
for any $\beta\in (0,\beta_0)$. 
In view of this, to prove the main theorem 
below (Theorem \ref{thm-1}), it is sufficient to 
verify \eqref{eq:intro-goal} for large values of $j$.
Moreover, if $u$ is compactly supported 
(which will be the case in our considerations 
since the functions will be localized; see \eqref{eq:K0-def}), 
then it is sufficient to establish the estimate above 
on a compact set containing the support of $u$. 
Indeed, this can be justified in the spirit of 
the Kato-Ponce inequality \cite[(1.1)]{KatoPonce1988}, 
although we provide a more direct 
argument in Remark \ref{rem:est_compact} below.

\begin{remark}\label{rem:Sob_Omega_bdd}
Let us briefly discuss the embedding
$W^{s,q}(\Omega)\hookrightarrow
W^{s_1,q_1}(\Omega)$ for a regular bounded domain
$\Omega\subseteq\R^D$, where $0<s_1<s$ and $1<q_1<q$.
A suitable extension
operator gives the standard interpolation identities and Besov
embeddings on $\Omega$ from those on $\R^D$.
Since $\Omega$ is bounded, H\"older's inequality also gives
$L^q(\Omega)\hookrightarrow L^{q_1}(\Omega)$ and
$W^{1,q}(\Omega)\hookrightarrow W^{1,q_1}(\Omega)$.
For fixed $q$ and $\Omega$, the norms of these inclusions are
bounded uniformly for $1<q_1<q$.
Hence,
\begin{align*}
	W^{s,q}(\Omega)=B^s_{q,q}(\Omega) 
	& =\left(L^q(\Omega),W^{1,q}(\Omega)\right)_{s,q} 
	\\ & 
	\hookrightarrow 
	\left(L^{q_1}(\Omega),W^{1,q_1}(\Omega)\right)_{s,q}
	=B^s_{q_1,q}(\Omega)
	\\ & \qquad\qquad\qquad\qquad\qquad\quad 
	\hookrightarrow_{s_1,q_1} B^{s_1}_{q_1,q_1}(\Omega)
	=W^{s_1,q_1}(\Omega),
\end{align*}
where $(\cdot,\cdot)_{s,q}$ denotes 
real interpolation (cf.~\cite[C.3]{Hytonen:2016aa}), 
and for each of the intermediate results 
we refer to \cite[Subsection~14.4]{Hytonen:2023aa}.
We emphasize that the nonuniform dependence
on $s_1$ and $q_1$ 
arises in the second embedding, 
where the associated norm bound 
involves a constant of the form 
$\bigl(\sum_{k\geq 0} 2^{-k(s-s_1)\tilde q}
\bigr)^{1/\tilde q}$, with 
$1/q_1 = 1/q + 1/\tilde q$.
\end{remark}

\begin{remark}\label{rem:est_compact}
Let us briefly demonstrate that it is 
sufficient to consider \eqref{eq:besov_sp_u} 
on a compact set, provided that 
$u$ is compactly supported. 
Let $\supp u\subseteq K_0$ and let us define
$\widetilde K_0 = \bigl\{ \mx_0\in\R^D : 
\operatorname{dist}(\mx_0,K_0)<1\bigr\}$.
If $u\in L^r(\R^D)$,
then for any $M>0$ we have
$$
\norm{\cA_{\phi_j}(u)}_{L^r(\R^D\setminus\widetilde{K}_0)}
\lesssim_M 2^{-j M}\norm{u}_{L^r(\R^D)},
\quad j\geq 1.
$$
Indeed, since $\cF^{-1}(\phi_j(\abs{\cdot}))$ 
is a Schwartz function, and 
$\phi_j(\abs{\mxi_0})=\phi(2^{-j}\abs{\mxi_0})$, 
for any $M>0$ and any
$\abs{\mx_0}\geq 1$ we have
$$
\abs{\cF^{-1}(\phi_j(\abs{\cdot}))(\mx_0)}
\lesssim_M 2^{-j(M-D)}\abs{\mx_0}^{-M}.
$$
Then for any $M>D$ we get
\begin{align*}
	\norm{\cA_{\phi_j}(u)}_{L^r(\R^D\setminus\widetilde{K}_0)}
	& \lesssim_M 2^{-j(M-D)}
	\norm{\abs{\cdot}^{-M} \En_{B(0,1)^c}
	\star \abs{u}}_{L^r(\R^D)}
	\\ &\lesssim_M 2^{-j(M-D)}
	\norm{u}_{L^r(\R^D)},
\end{align*}
where we have used that 
$\abs{\mx_0-\my_0}\geq 1$ for $\mx_0\in K_0$ 
and $\my_0\in\R^D\setminus\widetilde{K}_0$,
and that $\abs{\cdot}^{-M} \En_{B(0,1)^c}$ 
is integrable for $M>D$. Therefore, once we 
establish \eqref{eq:besov_sp_u} 
on $\widetilde{K}_0$, the estimate for
the entire $\R^D$ follows immediately. 
\end{remark}

For $\my \in \R^{D-1}$, $\eps > 0$, and $j \in \N$, 
introduce the functions
\begin{equation}\label{eq:chi-jeps-y-def}
	\begin{split}
		& \tilde \chi^{j,\eps}_{\my}(\mx_0)
		=Ce^{-x_0^2/2}\, 2^{\eps j (D-1)/2} 
		\prod\limits_{k=1}^{D-1}
		e^{-2^{2\eps j}\left(x_k-y_k\right)^2/2},
		\\ 
		& \chi^{j,\eps}_{\my}(\mx_0) 
		=\left(\tilde \chi^{j,\eps}_{\my}(\mx_0)\right)^2,
		\quad \mx_0=(x_0,\mx)\in \R^D,
	\end{split}
\end{equation}
where the unimportant constant
$C$ is chosen such that
\begin{equation}\label{eq:unit-mass}
	\int_{\R^D} \chi^{j,\eps}_{\my}(\mx_0)\,d\mx_0=1,
	\qquad 
	\int_{\R^{D-1}}\chi^{j,\eps}_{\my}(\mx_0)\,d\my \leq1.
\end{equation} 
Given the identity
\begin{equation}\label{eq:gaussian-q-moment}
	\int_{\R} \abs{z}^q e^{-a\abs{z}^2}\,dz 
	=a^{-\frac{1+q}{2}} \Gamma 
	\left(\frac{q+1}{2}\right), \quad a>0,\,\, q\ge 0, 
\end{equation}
and noting that $\Gamma(1/2)=\pi^{1/2}$, 
we can specify $C^2=\pi^{-D/2}$. 
The function $\chi^{j,\eps}_{\my}(\mx_0)
=\chi^{j,\eps}_{\my}(x_0,\mx)$ can be 
regarded as a (Gaussian) mollifier in $\mx$ with 
a scaling parameter of $2^{-\eps j}$.

For later use, note that
\begin{align}
	\label{eq:chi-jeps-y-prop1}
	\int_{\R^{D-1}} 
	\abs{\nabla_{\!\mx}^k\chi^{j,\eps}_{\my}(\mx_0)}^\ell
	\, d\my
	\lesssim 2^{\eps j
	\left(\left(D-1\right)\left(\ell-1\right)
	+k\ell\right)},
	\quad k=0,1,
\end{align}
for any finite $\ell \ge 1$. 
This estimate can be derived using 
\eqref{eq:gaussian-q-moment}.

Recall that for constants $a>0$ and $y\in \R$,
$$
\cF\left(e^{-a(x-y)^2/2}\right)
=\sqrt{2\pi}a^{-1/2}e^{-2\pi i y \xi} 
e^{-2\pi^2a^{-1}\xi^2}, 
\quad x,\xi\in \R.
$$ 
Using this we can compute the Fourier transform of 
$\tilde \chi^{j,\eps}_{\my}(\cdot)$:
\begin{equation}\label{exact}
	\begin{split}
		\cF\left(\tilde \chi^{j,\eps}_{\my}\right)\!(\mxi_0)
		& =C(2\pi)^{D/2}\,
		2^{-\eps j (D-1)/2}\,
		e^{-2\pi^2\xi_0^2}
		\\ & \qquad\quad \times 
		\prod\limits_{k=1}^{D-1}
		\exp(-2\pi i y_k \xi_k) 
		\, \exp\left(-2\pi^2 \, 2^{-2\eps j}\xi_k^2\right).
	\end{split}
\end{equation}

The following commutation estimate will 
play an important role in proving
our main result (see Subsection \ref{subsec:R1}).

\begin{lemma}\label{lem:commutation}
Let $u \in L^2(\R^D)$, $\psi \in C^1(\S^{D-1})$, 
$j\in\N$, and $\beta\in \R$.
Then for any $\my\in\R^{D-1}$ we have
\begin{equation}\label{commutation}
	\begin{split}
		& \norm{\abs{\cdot}^\beta \phi_j(\abs{\cdot})
		\Bigl( \psi\, \cF\left(\tilde \chi^{j,\eps}_{\my} u\right)\!(\cdot)
		-\cF\left(\tilde\chi^{j,\eps}_{\my}\,
		\cA_{\psi}(u)\right)\Bigr)(\cdot)}_{L^2(\R^D)} 
		\\ & \qquad 
		\lesssim  \norm{\nabla \psi}_{L^\infty(\S^{D-1})} 
		2^{-j\left(1-\eps(D+1)/2-\beta\right)}
		\norm{u}_{L^2(\R^D)},
	\end{split}
\end{equation}
where $\psi$ is short-hand for 
$\psi\left(\mxi_0/\abs{\mxi_0}\right)$ with 
$\mxi_0\in \R^D$, $\phi_j$ is defined in \eqref{eq:phi-j-def}, 
and $\tilde \chi^{j,\eps}_{\my}$ is defined in \eqref{eq:chi-jeps-y-def}.
\end{lemma}

\begin{proof}
Set $\phi_j^{(\beta)}:=\abs{\mxi_0}^\beta \phi_j(\abs{\mxi_0})$ 
and note that $\abs{\mxi_0}^\beta \sim 2^{\beta j}$ for
$\abs{\mxi_0}\in \supp\phi_j$, so that 
$0\le \phi_j^{(\beta)}\lesssim 2^{\beta j}$. By the 
convolution theorem for the Fourier transform and 
the definition \eqref{eq:A-psi-def} of $\cA_\psi$,
\begin{align*}
	&\norm{\phi_j^{(\beta)}
	\Bigl( \psi \, \cF\left(\tilde \chi^{j,\eps}_{\my} u\right)
	-\cF\left(\tilde\chi^{j,\eps}_{\my}\,
	\cA_{\psi}(u)\right)\Bigr)}_{L^2(\R^D)}^2
	\\& \quad 
	=\norm{ \phi_j^{(\beta)}
	\Bigl( \psi \,
	\bigl( \widehat{\tilde \chi^{j,\eps}_{\my}}\star \widehat{u}\bigr)
	-\widehat{\tilde \chi^{j,\eps}_{\my}} \star 
	\left(\psi\, \widehat{u}\right) \Bigr)}^2_{L^2(\R^D)} 
	\\ & \quad 
	=\int_{\R^D}\abs{\phi_j^{(\beta)}\int_{\R^D} \left[
	\psi\left(\frac{\mxi_0}{\abs{\mxi_0}}\right)
	-\psi\left(\frac{\meta_0}{\abs{\meta_0}}\right)\right] 
	\widehat{u}(\meta_0)\,
	\widehat{\tilde \chi^{j,\eps}_{\my}}(\mxi_0-\meta_0)
	\,d\meta_0}^2 \, d\mxi_0
	\\ & \quad 
	\leq \norm{\nabla \psi}^2_{\infty}
	\int_{\R^D} \abs{\phi_j^{(\beta)}\int_{\R^D} 
	\abs{\frac{\mxi_0}{\abs{\mxi_0}}
	-\frac{\meta_0}{\abs{\meta_0}}} 
	\, \abs{\widehat{u}(\meta_0)} 
	\, \abs{\widehat{\tilde \chi^{j,\eps}_{\my}}
	(\mxi_0-\meta_0)} 
	\, d\meta_0}^2 \, d\mxi_0
	\\ & \quad
	\lesssim 2^{2\beta j} \norm{\nabla \psi}^2_{\infty} 
	\int_{\abs{\mxi_0}\sim 2^j} 
	\Biggl | \int_{\R^D}
	\abs{\frac{\left(\mxi_0-\meta_0\right)\abs{\meta_0}
	+\meta_0 \left(\abs{\meta_0}-\abs{\mxi_0}\right)}
	{\abs{\meta_0}\abs{\mxi_0}}}
	\\ & \quad\qquad\qquad\qquad\qquad\qquad
	\qquad\qquad\qquad\times 
	\abs{\widehat{u}(\meta_0)}
	\, \abs{\widehat{\tilde \chi^{j,\eps}_{\my}}
	\left(\mxi_0-\meta_0\right)} 
	\, d\meta_0  \Biggr|^2 \, d\mxi_0
	\\ & \quad 
	\lesssim 2^{2(\beta-1) j}\norm{\nabla \psi}^2_{\infty} 
	\int_{\R^D}\abs{\int_{\R^D}  \abs{\widehat{u}(\meta_0)}\, 
	\abs{\mxi_0-\meta_0} \, 
	\abs{\widehat{\tilde \chi^{j,\eps}_{\my}}
	\left(\mxi_0-\meta_0\right)}
	\, d\meta_0}^2 \, d\mxi_0 
	\\ & \quad 
	\leq 2^{2(\beta-1) j} \norm{\nabla \psi}^2_{\infty} 
	\norm{\widehat{u}}_{L^2(\R^D)}^2 \norm{\abs{\cdot}
	\widehat{\tilde \chi^{j,\eps}_{\my}}(\cdot)}_{L^1(\R^D)}^2,
\end{align*} 
where in the last step we applied Young's 
convolution inequality.  

Finally, considering the exact form \eqref{exact} 
of the Fourier transform 
$\widehat{\tilde{\chi}^{j,\eps}_{\my}}(\cdot)$, 
we can estimate the $L^1$ norm 
of $\abs{\cdot}\widehat{\tilde{\chi}^{j,\eps}_{\my}}(\cdot)$
as follows:
$$
\norm{\abs{\cdot}
\widehat{\tilde \chi^{j,\eps}_{\my}}(\cdot)}_{L^1(\R^D)}
\sim 2^{j(D+1)\eps/2}, 
\quad \forall \my \in \R^{D-1}.
$$ 
Here, we have used the identity \eqref{eq:gaussian-q-moment} 
with $q=1$ and an appropriate choice of $a\sim 2^{-2\eps j}$, 
and then extended it to higher dimensions by exploiting the 
multiplicative structure in \eqref{exact}. 
\end{proof}

%%%%%%%%%%%%%%%%%%%%%%%%%%%
%%%%%%%%%%%%%%%%%%%%%%%%%%%
\section{The regularity estimate}\label{sec:main}

In this section, we prove the fractional Sobolev 
estimate \eqref{aver-quant}. 

\begin{theorem}\label{thm-1}
Let $h$ be a distributional solution 
of the kinetic equation \eqref{eq-1}. 
Assume that the hypotheses $\mathrm{({\bf H1})}$, 
$\mathrm{({\bf H2})}$, 
and $\mathrm{({\bf H3})}$ are satisfied, 
for some $p>1$, $\bar{p}>p'$, and $\bar{s}\in (0,1)$. 
Under the non-degeneracy condition \eqref{non-deg}, 
there exist a number
$r_0\in 
\bigl(1, 
\min\bigl\{
\frac{D}{D-1},
\frac{D-1}{D-1-\bar{s}}, 
\frac{p\bar{p}}{p+\bar{p}}
\bigr\}\bigr)$ 
and a strictly positive, decreasing function 
$\beta_0:(1,r_0)\to\R^+$ such that for 
every $r\in (1,r_0)$ and 
any $\beta<\beta_0(r)$, we have
\begin{equation}\label{aver-quant-mod}
	\R^{D}\ni \mx_0\mapsto 
	\int_{\R^m} h(\mx_0,\mlambda) \rho(\mlambda)
	\, d\mlambda \in W^{\beta,r}_{\loc}(\R^D),
	\quad \rho\in C^{\abs{\mkappa}}_c(\R^m). 
\end{equation}
More precisely, let
$U_0\Subset U_1\subset\R^D$ and $V\subset\R^m$ be bounded open
sets, and let $\rho\in C_c^{\abs{\mkappa}}(V)$.
For $0<\beta<\beta_0(r)$,
\begin{align*}
	\norm{\innb{h,\rho}}_{W^{\beta,r}(U_0)}
	&\leq C\left(\norm{h}_{L^p(U_1\times V)}
	+\norm{\gamma}_{\cM(U_1\times V)}\right).
\end{align*}
The constant depends on the fixed sets and cutoffs, $\rho$,
the exponents, dimensions, derivative order, and local bounds
in ({\bf H2}) and \eqref{non-deg}, but is independent of
$h$ and $\gamma$.
\end{theorem}

In the case $\bar{p}=\infty$, the 
term $\frac{p\bar{p}}{p+\bar{p}}$
appearing in the statement is 
to be interpreted as $p$.
For the smooth case in
Remark~\ref{rem:H2}, when $D=2$ and $\bar s=1$, we interpret
$\frac{D-1}{D-1-\bar s}=\infty$ and omit this term
from the minimum.
For information about the exponent $\beta_0$, we direct the reader 
to Subsection \ref{subsec:final} 
and to Remark \ref{rem:reg-exp}.

The proof will be developed 
in a series of subsections.

\subsection{Preparations}\label{subsec:prep}
Let $r$ be an exponent such that
\begin{equation}\label{eq:r-p-ass}
	1<r< p,
\end{equation}
where the precise admissible upper 
bound $r_0$ on $r$ will be determined 
in the final step of the proof.
The exponent $p$ appears first in ({\bf H1}). 
Going forward, we will assume that
\begin{equation}\label{eq:r-p-ass_2}
	p<2,
\end{equation}
which entails no loss of generality 
since our focus is on localized spaces. 
However, the other case, $p \geq 2$, leads to a 
slightly more favorable situation and will be 
briefly addressed in the final step of the proof.

Furthermore, we shall assume that $h$ and $\gamma$ 
in \eqref{eq-1-new} are compactly supported in $K_0\times L$ as 
defined in \eqref{eq:K0-def}. Such an assumption does not 
lead to any loss of generality. Indeed, owing to the 
linearity of the kinetic equation \eqref{eq-1-new}, it is possible 
to multiply it by a cutoff function, leading to a modified 
equation of essentially the same type. To be more precise, 
let
$\chi\in C_c^\infty(U_1\times V)$ be a fixed
cutoff function equal to one near
$\overline{U_0}\times\supp\rho$, with the sets as in the theorem.
Define 
$$
\widetilde{h} := h \chi, 
\quad 
\widetilde{\gamma} := \gamma \chi.
$$  
One can check that the kinetic equation satisfied by the 
compactly supported function
$\tilde{h}$ and measure $\tilde{\gamma}$
may be rewritten in a form similar to the original 
equation \eqref{eq-1-new}. 
Clarifying further, differentiation by parts gives
$$
(\partial_\mlambda^\mkappa \gamma) \chi
= \partial_\mlambda^\mkappa \widetilde \gamma
+ \!\!\!
\sum_{\substack{\mkappa_1 + \mkappa_2 = \mkappa
\\ \mkappa_2 \neq 0}} \binom{\mkappa}{\mkappa_1}
(-1)^{\abs{\mkappa_2}}
\partial_\mlambda^{\mkappa_1} 
\left(\gamma 
\, \partial_\mlambda^{\mkappa_2} \chi\right).
$$
The main source term $\partial_\mlambda^\mkappa \widetilde \gamma$ 
maintains the same form as before, with the distinction that 
it is now compactly supported across all variables. The new 
additional terms on the right-hand side involve $\mlambda$-derivatives 
of compactly supported measures, but the order of 
differentiation for these terms is lower than that 
in the main term $\partial_\mlambda^\mkappa \widetilde \gamma$.
Thus, they are estimated term by term
as in Subsection~\ref{subsec:R5}, with no worse decay.
A similar argument applies
to the spatial drift term \eqref{eq-1-new}.
Specifically, we can write 
$$
\Div_{\mx_0} \bigl( (1,\mff(\mx,\mlambda))h\bigr)\chi
=\Div_{\mx_0} \bigl( (1,\mff(\mx,\mlambda))\widetilde h\bigr)
- h\,(1,\mff(\mx,\mlambda))\cdot \nabla_{\mx_0}\chi ,
$$
where the second term on the right-hand side 
can be considered a zero-order source 
term in $L^1_{\mx_0,\mlambda}$. 
Write the localized source as
$\sum_{\mnu\leq\mkappa}\partial_\mlambda^\mnu\mu_\mnu$,
including the cutoff drift contribution in $\mu_0$.
The product rule above and
({\bf H2}) give
\begin{align*}
	\norm{\tilde h}_{L^p}
	+\sum_{\mnu\leq\mkappa}\norm{\mu_\mnu}_{\cM}
	&\leq C\left(\norm{h}_{L^p(U_1\times V)}
	+\norm{\gamma}_{\cM(U_1\times V)}\right),
\end{align*}
where $C$ depends on the fixed cutoffs and coefficient bounds.
The estimate in Subsection~\ref{subsec:R5} applies to each source,
with at least the decay obtained for order $\mkappa$.
Thus the bounds below hold with $\norm{\gamma}_{\cM}$ replaced by
$\sum_{\mnu\leq\mkappa}\norm{\mu_\mnu}_{\cM}$. 

In summary, from this point forward, we will focus on the 
original equation \eqref{eq-1-new}, assuming 
that $h$ and $\gamma$ have compact support 
in all variables. 

We will frequently use the notation
\begin{equation}\label{eq:fy-def}
	\mff_{\my}(\mlambda)  :=\mff(\my,\mlambda), 
\end{equation}
where $\my\in \R^{D-1}$ is (often) 
viewed as a fixed point.

For later use, let us define
the truncation functions
\begin{equation}\label{eq:Tl-def}
	\begin{split}
		& T_l(z)=
		\begin{cases}
			z, & \text{if $\abs{z}<l$} \\
			0,& \text{otherwise},
		\end{cases}
		\qquad 
		T^l(z)=z-T_{l}(z)=
		\begin{cases}
			0, & \text{if $\abs{z}<l$} \\
			z,& \text{otherwise},
		\end{cases}
	\end{split}
\end{equation} 
where $l>0$ is the truncation level. 
Under \eqref{eq:r-p-ass} and \eqref{eq:r-p-ass_2},
the following basic estimates hold: 
\begin{equation}\label{eq:Tl-est}
	\begin{split}
		&\norm{T_l(h)}^2_{L^2(K_0\times L)} 
		\leq l^{2-p} \norm{h}^p_{L^p(K_0\times L)}, 
		\\ & 
		\norm{T^l(h)}^r_{L^r(K_0\times L)} 
		\leq l^{r-p}\norm{h}^p_{L^p(K_0\times L)}.
	\end{split}
\end{equation}

The goal is to establish  \eqref{eq:intro-goal}.  
To this end, let us test the kinetic equation \eqref{eq-1-new} 
using the function
$$
(\mx_0,\mlambda)\mapsto 
e^{-2\pi i \mxi_0\cdot \mx_0} \rho(\mlambda) 
\phi_j(\abs{\mxi_0}) \chi^{j,\eps}_{\my}(\mx_0)
\frac{\xi_0+\mff_{\my}(\mlambda) \cdot \mxi}
{\abs{\xi_0+\mff_{\my}(\mlambda) \cdot \mxi}^2
+2^{-\epsilon j}\abs{\mxi}^2},
$$
for an $\epsilon>0$ to be determined later, $j\in\N$, 
and $\mx_0=(x_0,\mx)\in \R^D$.
Here, $\mxi_0=(\xi_0,\mxi)\in \R^D$ and $\my\in \R^{D-1}$ are 
considered as parameters. 
While they may look alike, the parameter $\epsilon$ is 
distinct from the $\eps$ employed in \eqref{eq:chi-jeps-y-def} 
for defining $\chi^{j,\eps}_{\my}(\cdot)$. 
From now on, we will assume that the set $L$ 
from \eqref{eq:K0-def} is defined as $L=\supp(\rho)$, 
where $\rho$ appears in \eqref{aver-quant-mod}.

The result is the following expression:
{\small \begin{equation}\label{new-eq-1}
	\begin{split}
		& 2\pi i \int_{\R^{D+m}}e^{-2\pi i \mxi_0\cdot \mx_0}
		\rho(\mlambda) \phi_j(\abs{\mxi_0}) \chi^{j,\eps}_{\my}(\mx_0)
		\cL_{j,\epsilon}^{(1)}
		h(\mx_0,\mlambda) \, d\mx_0\, d\mlambda
		\\ & \quad
		+2\pi i \int_{\R^{D+m}}e^{-2\pi i \mxi_0\cdot \mx_0} 
		\rho(\mlambda) \phi_j(\abs{\mxi_0}) \chi^{j,\eps}_{\my}(\mx_0) 
		\\ & \qquad \qquad \qquad\qquad \times
		\cL_{j,\epsilon}^{(2)}
		\cdot\bigl(\mff(\mx,\mlambda)-\mff_{\my}(\mlambda)\bigr)
		h(\mx_0,\mlambda) \,d\mx_0\, d\mlambda
		\\ & \quad
		-\int_{\R^{D+m}}e^{-2\pi i \mxi_0\cdot \mx_0}
		\rho(\mlambda) \phi_j(\abs{\mxi_0})
		\cL_{j,\epsilon}^{(3)}
		\bigl((1,\mff(\mx,\mlambda))\cdot \nabla_{\!\mx_0} 
		\chi^{j,\eps}_{\my}(\mx_0)\bigr)
		h(\mx_0,\mlambda) \,d\mx_0 \,d\mlambda
		\\ & \quad\qquad 
		=(-1)^{\abs{\mkappa}} \int_{\R^{D+m}}
		e^{-2\pi i \mxi_0\cdot \mx_0} 
		\phi_j(\abs{\mxi_0}) \, \chi^{j,\eps}_{\my}(\mx_0) 
		\cS_{j,\epsilon}
		\, d\gamma(\mx_0,\mlambda),
	\end{split}
\end{equation}}
where
{\small
\begin{align*}
	&\cL_{j,\epsilon}^{(1)}
	=\frac{\abs{\xi_0+\mff_{\my}(\mlambda) \cdot \mxi}^2 }
	{\abs{\xi_0+\mff_{\my}(\mlambda) \cdot \mxi}^2
	+2^{-\epsilon j}\abs{\mxi}^2},
	\quad
	\cL_{j,\epsilon}^{(2)}
	=\frac{\bigl(\xi_0+\mff_{\my}(\mlambda)\cdot \mxi\bigr)
	\mxi}{\abs{\xi_0+\mff_{\my}(\mlambda) \cdot \mxi}^2
	+2^{-\epsilon j}\abs{\mxi}^2},
	\\ & 
	\cL_{j,\epsilon}^{(3)}
	=\frac{\xi_0+\mff_{\my}(\mlambda)\cdot \mxi}
	{\abs{\xi_0+\mff_{\my}(\mlambda) \cdot \mxi}^2
	+2^{-\epsilon j}\abs{\mxi}^2},
	\quad
	\cS_{j,\epsilon}
	=\pa_{\mlambda}^\mkappa
	\left(\rho(\mlambda)
	\frac{\xi_0+\mff_{\my}(\mlambda) \cdot \mxi}
	{\abs{\xi_0+\mff_{\my}(\mlambda) \cdot \mxi}^2
	+2^{-\epsilon j}\abs{\mxi}^2}\right).
\end{align*}}
Using that $\cL_{j,\epsilon}^{(1)}
=1-\widetilde{\cL}_{j,\epsilon}^{(1)}$,
where
\begin{equation}\label{eq:tL1-def}
	\widetilde{\cL}_{j,\epsilon}^{(1)}
	:=\frac{2^{-\epsilon j}\abs{\mxi}^2}
	{\abs{\xi_0+\mff_{\my}(\mlambda) \cdot \mxi}^2
	+2^{-\epsilon j}\abs{\mxi}^2},
\end{equation}
we may reformulate the first term in \eqref{new-eq-1} 
to achieve the following form: 
{\small \begin{equation}\label{new-eq-2}
	\begin{split}
		& 2\pi i \int_{\R^{D+m}}e^{-2\pi i \mxi_0\cdot \mx_0}
		\rho(\mlambda) \phi_j(\abs{\mxi_0}) \chi^{j,\eps}_{\my}(\mx_0)
		h(\mx_0,\mlambda) \, d\mx_0\, d\mlambda
		\\ & 
		=2\pi i \int_{\R^{D+m}}e^{-2\pi i \mxi_0\cdot \mx_0}
		\rho(\mlambda) \phi_j(\abs{\mxi_0}) \chi^{j,\eps}_{\my}(\mx_0)
		\widetilde{\cL}_{j,\epsilon}^{(1)}
		h(\mx_0,\mlambda) \, d\mx_0\, d\mlambda
		\\ & \quad 
		-2\pi i \int_{\R^{D+m}}e^{-2\pi i \mxi_0\cdot \mx_0} 
		\rho(\mlambda) \phi_j(\abs{\mxi_0}) \chi^{j,\eps}_{\my}(\mx_0) 
		\\ & \qquad\qquad \qquad \qquad\qquad \times
		\cL_{j,\epsilon}^{(2)}
		\cdot\bigl(\mff(\mx,\mlambda)-\mff_{\my}(\mlambda)\bigr)
		h(\mx_0,\mlambda) \,d\mx_0\, d\mlambda
		\\ & \quad
		+\int_{\R^{D+m}}e^{-2\pi i \mxi_0\cdot \mx_0}
		\rho(\mlambda) \phi_j(\abs{\mxi_0})
		\cL_{j,\epsilon}^{(3)}
		\bigl((1,\mff(\mx,\mlambda))\cdot \nabla_{\!\mx_0} 
		\chi^{j,\eps}_{\my}(\mx_0)\bigr)
		h(\mx_0,\mlambda) \,d\mx_0 \, d\mlambda
		\\ & \quad 
		+(-1)^{\abs{\mkappa}} \int_{\R^{D+m}}
		e^{-2\pi i \mxi_0 \cdot \mx_0} 
		\phi_j(\abs{\mxi_0}) \chi^{j,\eps}_{\my}(\mx_0) 
		\cS_{j,\epsilon}\, d\gamma(\mx_0,\mlambda).
	\end{split}
\end{equation}}

Recalling \eqref{eq:mxi-prime}, it is evident that
\begin{equation}\label{eq:tL-tS-def}
	\begin{split}
		& \cL_{j,\epsilon}^{(2)}=
		\frac{\bigl(\xi_0'+\mff_{\my}(\mlambda)\cdot \mxi'\bigr)
		\mxi'}{\abs{\xi_0'+\mff_{\my}(\mlambda) \cdot \mxi'}^2
		+2^{-\epsilon j}\abs{\mxi'}^2},
		\\ & 
		\cL_{j,\epsilon}^{(3)}=\frac{1}{\abs{\mxi_0}}
		\frac{\xi_0'+\mff_{\my}(\mlambda)\cdot \mxi'}
		{\abs{\xi_0'+\mff_{\my}(\mlambda) 
		\cdot \mxi'}^2+2^{-\epsilon j}\abs{\mxi'}^2}
		=:\frac{1}{\abs{\mxi_0}}\widetilde{\cL}_{j,\epsilon}^{(3)},
		\\ & 
		\cS_{j,\epsilon}
		=\frac{1}{\abs{\mxi_0}}\pa_{\mlambda}^\mkappa
		\left(\rho(\mlambda)
		\frac{\xi_0'+\mff_{\my}(\mlambda) \cdot \mxi'}
		{\abs{\xi_0'+\mff_{\my}(\mlambda) \cdot \mxi'}^2
		+2^{-\epsilon j}\abs{\mxi'}^2}\right)
		=:\frac{1}{\abs{\mxi_0}}\widetilde{\cS}_{j,\epsilon}.
	\end{split}
\end{equation}

Let $v = v(\mx_0) \in C^\infty_c(\R^D)$ and 
$\varphi = \varphi(\my) \in C^\infty_c(K)$, where 
$K$ is defined in \eqref{eq:K0-def}. 
Multiply the equation \eqref{new-eq-2} by
$\frac{1}{2\pi i} 
\varphi(\my) \overline{\widehat{v}(\mxi_0)}$, and 
integrate the resulting expression 
over $(\my, \mxi_0) \in K \times \R^D$. 
By the Plancherel theorem,
applied to the $L^2$ scalar product,
it follows that
\begin{equation}\label{eq:init-relation}
	\innb{I,\bar{v}}
	=\innb{R_1,\bar{v}}
	+\innb{R_2,\bar{v}}
	-\innb{R_3,\bar{v}}
	+\innb{R_4,\bar{v}}
	+\innb{R_5,\bar{v}},
\end{equation}
where
\begin{align*}
	&\innb{I,\bar{v}}
	=\int_{K\times \R^D} \cA_{\phi_j} 
	\left(\int_{\R^{m}} \rho(\mlambda) \chi^{j,\eps}_{\my}(\cdot)
	h(\cdot,\mlambda) \, d\mlambda \right)\!(\mx_0) 
	\overline{v(\mx_0)} \varphi(\my)  \, d\mx_0\, d\my,
	\\ & 
	\innb{R_1,\bar{v}}=\int_{K\times \R^D \times \R^m} 
	\rho(\mlambda) \phi_j(\abs{\mxi_0}) 
	\widetilde{\cL}_{j,\epsilon}^{(1)}
	\cF\Bigl(\chi^{j,\eps}_{\my}(\cdot) 
	T_{l_j}(h(\cdot,\mlambda))\Bigr)(\mxi_0) 
	\\ & \qquad\qquad\qquad\qquad\qquad\qquad
	\qquad \qquad\qquad\times
	\overline{\widehat{v}(\mxi_0)} \varphi(\my)
	\, d\mlambda \, d\mxi_0 \, d\my,
	\\ & 
	\innb{R_2,\bar{v}}
	=\int_{K\times {\R^D}\times\R^m}
	\cA_{\phi_j
	\widetilde{\cL}_{j,\epsilon}^{(1)}}
	\Bigl( \rho(\mlambda) \chi^{j,\eps}_{\my}(\cdot)
	T^{l_j}(h(\cdot,\mlambda))\Bigr)(\mx_0) 
	\\ & \qquad\qquad\qquad\qquad\qquad\qquad
	\qquad \qquad\qquad\times
	\overline{v(\mx_0)} \varphi(\my) 
	\, d\mlambda\, d\mx_0\, d\my,
	\\ & 
	\innb{R_3,\bar{v}}
	=\int_{K\times \R^D\times \R^m}
	\cA_{\phi_j
	\cL_{j,\epsilon}^{(2)}}
	\Bigl(\rho(\mlambda) \bigl(\mff(\cdot,\mlambda)
	-\mff_{\my}(\mlambda)\bigr) \chi^{j,\eps}_{\my}(\cdot)
	h(\cdot,\mlambda) \Bigr)(\mx_0) 
	\\ & \qquad\qquad\qquad\qquad\qquad\qquad
	\qquad \qquad\qquad\times
	\overline{v(\mx_0)} \varphi(\my) 
	\, d\mlambda\, d\mx_0\, d\my\mathpunct{,}
	\\ & 
	\langle R_4,\bar{v}\rangle
	=\frac{1}{2\pi i}\int_{K\times \R^D\times \R^m}
	\cA_{\phi_j \cL_{j,\epsilon}^{(3)}}
	\Bigl( \rho(\mlambda)\bigl(1,\mff(\cdot,\mlambda)\bigr)\cdot 
	\nabla_{\!\mx_0}\chi^{j,\eps}_{\my}(\cdot) 
	h(\cdot,\mlambda) \Bigr)(\mx_0)
	\\ & \qquad\qquad\qquad\qquad\qquad\qquad
	\qquad \qquad\qquad\times
	\overline{v(\mx_0)} \varphi(\my)
	\, d\mlambda\, d\mx_0 \, d\my,
	\\ & 
	\innb{R_5,\bar{v}}
	=\frac{(-1)^{\abs{\mkappa}}}{2\pi i}
	\int_{\R^D\times K\times \R^m\times K_0}
	e^{-2\pi i \mxi_0\cdot \mx_0}
	\frac{\phi_j(\abs{\mxi_0})}{\abs{\mxi_0}}
	\chi^{j,\eps}_{\my}(\mx_0) 
	\widetilde{\cS}_{j,\epsilon}
	\\ & \qquad\qquad\qquad\qquad\qquad\qquad
	\qquad \qquad\qquad\times 
	\overline{\widehat{v}(\mxi_0)} \varphi(\my) 
	\, d\gamma(\mx_0,\mlambda)\, d\my\, d\mxi_0.
\end{align*}
The truncation functions $T_{l_j}(\cdot)$, $T^{l_j}(\cdot)$ 
are defined in \eqref{eq:Tl-def}, and 
the truncation level $l_j$ is specified as
\begin{equation}\label{eq:lj-def}
	l_j = 2^{\sigma j}, 
	\quad \sigma > 0,
\end{equation}
with $\sigma$ to be determined later. 
Here, and in other contexts, 
we apply the Fourier transform $\cF$ and the inverse Fourier 
transform $\cF^{-1}$ exclusively with 
respect to $\mx_0 = (x_0, \mx)$ and its 
dual variable $\mxi_0 = (\xi_0, \mxi)$, 
but not to $\mlambda$. 

\medskip

In the following subsections, we will examine each 
term of \eqref{eq:init-relation} individually, 
beginning with the terms on the right-hand side.

\subsection{Estimate of $R_1$}\label{subsec:R1}

Before applying the previously 
proved commutation lemma, we begin 
by splitting $\innb{R_1,\bar{v}}$ 
into $R_{1,1} + R_{1,2}$, where
\begin{align*}
	& R_{1,1}= \int\limits_{K\times \R^D \times \R^m} 
	\rho(\mlambda) \phi_j(\abs{\mxi_0}) \,
	\cF\Bigl(\tilde \chi^{j,\eps}_{\my}(\cdot) 
	\cA_{\widetilde{\cL}_{j,\epsilon}^{(1)}}
	\left(\tilde \chi^{j,\eps}_{\my}(\cdot)
	T_{l_j}(h(\cdot,\mlambda))\right)\Bigr)(\mxi_0) 
	\\ & \qquad\qquad\qquad\qquad\qquad\qquad
	\qquad \qquad\qquad\times
	\overline{\widehat{v}(\mxi_0)} \varphi(\my)
	\, d\mlambda \, d\mxi_0 \, d\my,
	\\ & 
	\Delta^{j,\eps}_{\my}(\mxi_0,\mlambda)
	=\widetilde{\cL}_{j,\epsilon}^{(1)}
	\cF\Bigl(\chi^{j,\eps}_{\my}(\cdot) 
	T_{l_j}(h(\cdot,\mlambda))\Bigr)(\mxi_0) 
	\\ & \qquad\qquad\qquad\qquad
	-\cF\Bigl(\tilde \chi^{j,\eps}_{\my}(\cdot) 
	\cA_{\widetilde{\cL}_{j,\epsilon}^{(1)}}
	\left(\tilde \chi^{j,\eps}_{\my}(\cdot)
	T_{l_j}(h(\cdot,\mlambda))\right)\Bigr)(\mxi_0),
	\\ & 
	R_{1,2}= \int\limits_{K\times \R^D \times \R^m} 
	\rho(\mlambda) \phi_j(\abs{\mxi_0}) 
	\Delta^{j,\eps}_{\my}(\mxi_0,\mlambda)\, 
	\overline{\widehat{v}(\mxi_0)} \varphi(\my)
	\, d\mlambda \, d\mxi_0 \, d\my,
\end{align*}	
recalling that $\left(\tilde \chi^{j,\eps}_{\my}\right)^2
=\chi^{j,\eps}_{\my}$, see \eqref{eq:chi-jeps-y-def}.
We note that the analysis of $R_{1,1}$
is the only place in the entire proof where
the non-degeneracy condition \eqref{non-deg} is used. 

Now, using some basic 
properties of the Fourier transform, the definition 
\eqref{eq:A-psi-def} of
a multiplier operator,  the 
Cauchy-Schwarz inequality, and keeping 
in mind the definition \eqref{eq:tL1-def} of the symbol 
$\widetilde{\cL}_{j,\epsilon}^{(1)}$,  we 
estimate $R_{1,1}$ as follows:
\begin{align*}
	\abs{R_{1,1}} & =\Biggl |\int_{K\times \R^D \times L} 
	\rho(\mlambda)
	\cA_{\widetilde{\cL}_{j,\epsilon}^{(1)}}
	\left(\tilde \chi^{j,\eps}_{\my}(\cdot)
	T_{l_j}(h(\cdot,\mlambda))\right)\!(\mx_0)
	\varphi(\my)
	\\ &\qquad\qquad\qquad\quad\quad\times
	\tilde \chi^{j,\eps}_{\my}(\mx_0) 
	\overline{\cA_{\phi_j}(v)(\mx_0)}
	\, d\mlambda \, d\mx_0 \, d\my\Biggr |
	\\ & \le \int_{K}
	\norm{\int_{L} \rho(\mlambda)
	\widetilde{\cL}_{j,\epsilon}^{(1)}
	\cF\left(\tilde \chi^{j,\eps}_{\my}(\cdot)
	T_{l_j}(h(\cdot,\mlambda))\right)\!(\mxi_0)
	\varphi(\my)\, d\mlambda}_{L^2(\R^D)}
	\\ &\qquad\qquad\qquad\quad\quad\times
	\norm{\tilde \chi^{j,\eps}_{\my}(\cdot) 
	\overline{\cA_{\phi_j}(v)(\cdot)}}_{L^2(\R^D)}\, d\my
	\\ & \lesssim_\rho 
	\esssup_{\my\in K}\sup\limits_{\abs{\mxi'_0}=1} 
	\norm{\frac{2^{-\epsilon j}\abs{\mxi'}^2}
	{\abs{\xi'_0+\mff_{\my}(\cdot) \cdot \mxi'}^2
	+2^{-\epsilon j}\abs{\mxi'}^2}}_{L^2(L)}
	\\ & \qquad  \times
	\norm{\tilde \chi^{j,\eps}_{\cdot}(\cdot)
	T_{l_j}(h(\cdot,\cdot))
	\varphi(\cdot)}_{L^2(K\times K_0\times L)}
	\norm{\tilde \chi^{j,\eps}_{\cdot}(\cdot) 
	\overline{\cA_{\phi_j}(v)(\cdot)}}_{L^2(K\times \R^D)}
	\\ & =: R_{1,1}^{(1)} \, R_{1,1}^{(2)} \, R_{1,1}^{(3)}.
\end{align*}
We can bound $R_{1,1}^{(3)}$ in the following way:
\begin{align*}
	\bigl(R_{1,1}^{(3)}\bigr)^2
	& \overset{\eqref{eq:chi-jeps-y-def}}{=}
	C^2\int_{K\times \R^D} e^{-x_0^2}\,
	2^{\eps j(D-1)} \, 
	e^{-\abs{\frac{\mx-\my}{2^{-j\eps}}}^2} 
	\, \abs{\cA_{\phi_j}(v)(\mx_0)}^2 
	\, d\mx_0\, d\my 
	\\ & \quad 
	\le C^2\int_{\R^{D-1}\times \R^D} e^{-\abs{\mz}^2} 
	\, \abs{\cA_{\phi_j}(v)(\mx_0)}^2
	\, d\mx_0  \, d\mz 
	\\ & \quad 
	=C^2\pi^{(D-1)/2}\norm{\cA_{\phi_j}(v)}_{L^2(\R^D)}^2
	\\ & \quad 
	\lesssim \norm{v}_{L^2(\R^D)}^2,
\end{align*}
where we used $e^{-x_0^2} \le 1$, 
the change of variables:
\begin{equation}\label{eq:z-var}
	\mz=\frac{\my-\mx}{2^{-\eps j}} 
	\implies 
	d\mz=2^{\eps j(D-1)}\, d\my,
\end{equation}
and the fact that $\cA_{\phi_j}:L^2(\R^D)\to L^2(\R^D)$ 
is bounded uniformly in $j$, which is 
argued in Section \ref{sec:prelim}.

Let us now turn to $R_{1,1}^{(2)}$. By using 
$\left(\tilde \chi^{j,\eps}_{\my}\right)^2
=\chi^{j,\eps}_{\my}$ and the second 
part of \eqref{eq:unit-mass}, it follows that
\begin{align*}
	R_{1,1}^{(2)} & =
	\left( \int_{K_0\times L} \left( \int_K 
	\chi_\my^{j,\eps}(\mx_0)(\varphi(\my))^2\,d\my\right) 
	\abs{T_{l_j}(h(\mx_0,\mlambda))}^2
	\,d\mx_0 \, d\mlambda \right)^{1/2}
	\\ & \lesssim_\varphi 
	\norm{T_{l_j}(h)}_{L^2(K_0\times L)}
	\\ &
	\overset{\eqref{eq:Tl-est}}{\lesssim}
	l_j^{(2-p)/2}
	\norm{h}^{p/2}_{L^p(K_0\times L)}  
	\overset{\eqref{eq:lj-def}}{=}
	2^{\sigma j(1-p/2)}
	\norm{h}^{p/2}_{L^p(K_0\times L)},
\end{align*}
keeping in mind that $p<2$, see \eqref{eq:r-p-ass_2}, 
and the $h\in L^p$ assumption ({\bf H1}). 
In passing, we note that if $p \geq 2$, the 
truncation function $T_{l_j}$ is not required, 
and the proof simplifies.

For %
$\epsilon/2<\zeta<\epsilon$ (an additional parameter to 
be determined later), define the set 
$A_j := \seq{\mlambda \in L:
\abs{\xi_0'+\mff_\my(\mlambda) \cdot \mxi'}
< 2^{-j\left(\epsilon - \zeta\right)}}$,
and define the set $A_j^c$ similarly but with $<$ 
replaced by $\ge$, so that $L = A_j \cup A_j^c$. Then
\begin{align*}
	R_{1,1}^{(1)}
	& \le \esssup\limits_{\my\in K} \sup\limits_{\abs{\mxi_0'}=1}
	\norm{\frac{2^{-\epsilon j}\abs{\mxi'}^2}
	{\abs{\xi'_0+\mff_{\my}(\cdot) \cdot \mxi'}^2
	+2^{-\epsilon j}\abs{\mxi'}^2}}_{L^2(A_j)}
	\\ & \quad 
	+\esssup\limits_{\my\in K} \sup\limits_{\abs{\mxi_0'}=1} 
	\norm{\frac{2^{-\epsilon j}\abs{\mxi'}^2}
	{\abs{\xi'_0+\mff_{\my}(\cdot) \cdot \mxi'}^2
	+2^{-\epsilon j}\abs{\mxi'}^2}}_{L^2(A_j^c)}
	\\ & \mathrel{\lesssim_{K,L}}
	 2^{-j\left(\epsilon-\zeta\right)\alpha/2}
	 +2^{-j(2\zeta-\epsilon)},
\end{align*}
where we have used the 
non-degeneracy condition \eqref{non-deg},
see also \eqref{eq:fy-def}, to handle 
the first term, which implies that 
$$
\meas\Bigl(\seq{\mlambda\in L : 
\abs{\xi_0' + \mff_\my(\mlambda) \cdot \mxi'}<
2^{-j\left(\epsilon-\zeta\right)}}\Bigr) 
\lesssim 2^{-j\left(\epsilon-\zeta\right)\alpha}.
$$ 

For the second term, the
multiplier is at most
$2^{-\epsilon j+2j(\epsilon-\zeta)}=2^{-j(2\zeta-\epsilon)}$
on $A_j^c$, since $\abs{\mxi'}\leq1$.

Summarizing our findings for $R_{1,1}$,
\begin{equation}\label{new-eq-31}
	\begin{split}
		\abs{R_{1,1}} &\lesssim 
		R_{1,1}^{(1)}\, R_{1,1}^{(2)}\, R_{1,1}^{(3)}
		\\ & \lesssim 
		\left(2^{-j\left(\epsilon-\zeta\right)\alpha/2}
		+2^{-j(2\zeta-\epsilon)}\right)2^{\sigma j(1-p/2)}
		\norm{h}^{p/2}_{L^p(K_0\times L)}
		\norm{v}_{L^2(\R^D)}
		\\ & \lesssim 
		\left(2^{-j\left(\epsilon\frac{\alpha}{2}
		-\zeta\frac{\alpha}{2}
		-\sigma\left(1-\frac{p}{2}\right)\right)}
		+2^{-j\left(2\zeta-\epsilon-\sigma
		\left(1-\frac{p}{2}\right) \right)}\right)
		\norm{h}^{p/2}_{L^p(K_0\times L)}
		\norm{v}_{L^2(\R^D)}.
	\end{split}
\end{equation}

We will tackle the term $R_{1,2}$ by employing the 
commutation lemma (Lemma \ref{lem:commutation}). 
In this context, we specify 
$\beta=0$, $\psi=\widetilde{\cL}_{j,\epsilon}^{(1)}$, and 
$u=\tilde \chi^{j,\eps}_{\my}(\cdot)T_{l_j}(h(\cdot,\mlambda))$, 
noting the bound $\norm{\nabla_{\!\mxi_0} 
\widetilde{\cL}_{j,\epsilon}^{(1)}}_{\infty} 
\lesssim 2^{\epsilon j/2}$. To proceed, 
let us first verify this bound.

For $k \in \seq{1, 2, \dots, D-1}$, where 
the case $k=0$ will be addressed separately, 
let $\my \in K$, $\mlambda \in L$, 
and $\mxi_0 \in \S^{D-1}$. Differentiating gives
\begin{align*}
	 \partial_{\xi_k} \widetilde{\cL}_{j,\epsilon}^{(1)} 
	& = \frac{2 \cdot 2^{-\epsilon j} \xi_k}
	{\abs{\xi_0 + \mff_{\my}(\mlambda) \cdot \mxi}^2 
	+ 2^{-\epsilon j} \abs{\mxi}^2} 
	\\ &\qquad 
	- \frac{2^{-\epsilon j} \abs{\mxi}^2 
	\cdot \left(2(\xi_0+\mff_{\my}(\mlambda) \cdot \mxi) 
	\mff_{\my,k}(\mlambda) + 2\cdot 2^{-\epsilon j} \xi_k\right)}
	{\bigl(\abs{\xi_0 + \mff_{\my}(\mlambda) \cdot \mxi}^2 
	+2^{-\epsilon j} \abs{\mxi}^2\bigr)^2} 
	\\ & = 2A_k - 
	2 \widetilde{\cL}_{j,\epsilon}^{(1)} \left(A_k + B_k\right),
\end{align*}
where
$$
A_k = 
\frac{2^{-\epsilon j} \xi_k}
{\abs{\xi_0 + \mff_{\my}(\mlambda) \cdot \mxi}^2 
+ 2^{-\epsilon j} \abs{\mxi}^2},
\quad 
B_k = 
\frac{\left(\xi_0 + \mff_{\my}(\mlambda) \cdot \mxi\right) 
\mff_{\my,k}(\mlambda)}
{\abs{\xi_0 + \mff_{\my}(\mlambda) \cdot \mxi}^2 
+ 2^{-\epsilon j} \abs{\mxi}^2}.
$$
Clearly, we have $\widetilde{\cL}_{j,\epsilon}^{(1)} \leq 1$. 

We estimate $A_k$ by considering two cases: 
(a) $\abs{\mxi} > \frac{1}{M}$, and 
(b) $\abs{\mxi} \leq \frac{1}{M}$, where $M > 1$ is to 
be determined later. 

Case (a). If $\abs{\mxi} > \frac{1}{M}$, 
then $\abs{\mxi} < M \abs{\mxi}^2$. 
As a result,  
$$
\abs{A_k} \leq \frac{2^{-\epsilon j} 
M \abs{\mxi}^2}{2^{-\epsilon j} \abs{\mxi}^2} = M.
$$

Case (b). If $\abs{\mxi} \leq \frac{1}{M}$, then
$\abs{\xi_0} \geq \frac{\sqrt{M^2-1}}{M}$, 
using $\abs{\xi_0} = \sqrt{1 - \abs{\mxi}^2}$. 
Applying $|a + b|^2 \geq \frac{1}{2}a^2 - b^2$, 
which holds for all $a,b\in \R$,
\begin{equation*}
	\abs{\xi_0 + \mff_{\my}(\mlambda) \cdot \mxi}^2
	\geq \frac12 \abs{\xi_0}^2 
	- \abs{\mff_{\my}(\mlambda) \cdot \mxi}^2 
	\geq \frac{M^2-1}{2M^2}
	-\frac{F^2}{M^2}=:\tilde{M},
\end{equation*}
where
\begin{equation}\label{eq:F-def}
	F := \sup_{\my \in K, \mlambda \in L} 
	\abs{\mff_{\my}(\mlambda)}<\infty,
\end{equation}
and we choose $M$ so that
$M^2\geq2(1+2F^2)$, hence $\tilde M\geq1/4$.
Using this, we obtain
$$
\abs{A_k} \leq \frac{2^{-\epsilon j}\frac{1}{M}}{\tilde M+0}
=\frac{2^{-\epsilon j}}{M\tilde M}.
$$

Combining cases (a) and (b), we deduce that
\begin{equation*}
	\abs{A_k} \leq M 
	+\frac{2^{-\epsilon j}}{\tilde M M} \lesssim 1,
\end{equation*}
where the implicit constant in $\lesssim$ depends 
solely on (the fixed) $\mff$, $K$, and $L$.

We estimate $B_k$ by considering the same two cases: 
(a) $\abs{\mxi} > \frac{1}{M}$ 
and (b) $\abs{\mxi} \leq \frac{1}{M}$. In case (a),
\begin{align*}
	\abs{B_k} &\le 
	\frac{\sqrt{\abs{\xi_0+\mff_{\my}(\mlambda) \cdot \mxi}^2
	+2^{-\epsilon j}\abs{\mxi}^2} \, F}
	{\abs{\xi_0+\mff_{\my}(\mlambda) \cdot \mxi}^2
	+2^{-\epsilon j}\abs{\mxi}^2}
	=\frac{F}{2^{-\epsilon j/2}\sqrt{2^{\epsilon j}
	\abs{\xi_0+\mff_{\my}(\mlambda) \cdot \mxi}^2
	+\abs{\mxi}^2}}
	\\ & \le
	\frac{2^{\epsilon j/2}F}{\sqrt{\frac{1}{M^2}}}
	\lesssim 2^{\epsilon j/2}.
\end{align*}

In case (b), we proceed as follows: 
$$
\abs{\xi_0 + \mff_{\my}(\mlambda) \cdot \mxi}^2
\geq \tilde{M} > 0 \implies \abs{\xi_0
+\mff_{\my}(\mlambda) \cdot \mxi}
> \sqrt{\tilde{M}},
$$
which implies that 
$$
\abs{\xi_0 + \mff_{\my}(\mlambda) \cdot \mxi} 
\leq 
\frac{1}{\sqrt{\tilde{M}}}
\abs{\xi_0 + \mff_{\my}(\mlambda) \cdot \mxi}^2.
$$
Employing this in $B_k$ and 
recalling \eqref{eq:F-def}, we obtain
$$
\abs{B_k} \leq 
\frac{\frac{1}{\sqrt{\tilde{M}}} 
\abs{\xi_0 + \mff_{\my}(\mlambda) \cdot \mxi}^2 F}
{\abs{\xi_0 + \mff_{\my}(\mlambda) \cdot \mxi}^2} 
= \frac{F}{\sqrt{\tilde{M}}} \lesssim 1.
$$

Combining cases (a) and (b) for $B_k$,
we arrive at
$$
\abs{B_k} 
\lesssim 2^{\epsilon j/2} + 1
\mathrel{\lesssim} 2^{\epsilon j/2}.
$$

Summarizing (for $k\in \seq{1,\dots,D-1}$), we have shown that
$\abs{\partial_{\xi_k} \widetilde{\cL}_{j,\epsilon}^{(1)}} 
\lesssim 2^{\epsilon j/2}$.

For $k = 0$, we have a simpler situation:
\begin{align*}
	 \partial_{\xi_0} \widetilde{\cL}_{j,\epsilon}^{(1)} 
	 &= -\frac{2^{-\epsilon j} 
	 \abs{\mxi}^2 \cdot 2 (\xi_0 + \mff_{\my}(\mlambda) \cdot \mxi)}
	 {(\abs{\xi_0 + \mff_{\my}(\mlambda) \cdot \mxi}^2 
	 + 2^{-\epsilon j}\abs{\mxi}^2)^2} 
	\\ & 
	= -2\widetilde{\cL}_{j,\epsilon}^{(1)}
	\frac{\xi_0 + \mff_{\my}(\mlambda) \cdot \mxi}
	{\abs{\xi_0 + \mff_{\my}(\mlambda) \cdot \mxi}^2 
	+ 2^{-\epsilon j}\abs{\mxi}^2},
\end{align*}
where the last term can be analyzed similarly 
to $B_k$, yielding  
$\abs{\partial_{\xi_0} \widetilde{\cL}_{j,\epsilon}^{(1)}}
\lesssim 2^{\epsilon j/2}$.

Concluding our analysis for $k=0,1,\ldots,D-1$, 
we have demonstrated that
\begin{equation}\label{eq:tilde-L1-xi-deriv}
	\sup_{\my \in K, \mlambda \in L} \norm{
	\nabla_{\mxi_0}
	\widetilde{\cL}_{j,\epsilon}^{(1)}}_{L^\infty(\S^{D-1})} 
	\lesssim 2^{\epsilon j/2}.
\end{equation}

We now return to the estimation of $R_{1,2}$.
We will employ Lemma \ref{lem:commutation} 
with $\psi=\widetilde{\cL}_{j,\epsilon}^{(1)}$, 
see \eqref{eq:tL1-def}, and $u=\tilde \chi^{j,\eps}_{\my}(\cdot)
T_{l_j}(h(\cdot,\mlambda))$, along with the fact that  
$\chi^{j,\eps}_{\my}=\left(\tilde \chi^{j,\eps}_{\my}\right)^2$. 
Applying the Cauchy-Schwarz inequality (first 
with respect to $\mxi_0$ and then with respect to $\my$, while 
utilizing the second part of \eqref{eq:unit-mass}), we 
can derive the following:
\begin{align}
	\abs{R_{1,2}} &\lesssim  
	\int_{K\times\R^m} \abs{\rho(\mlambda)}
	\norm{\phi_j(\abs{\cdot}) 
	\Delta^{j,\eps}_{\my}(\cdot,\mlambda)}_{L^2(\R^D)}
	\norm{v}_{L^2(\R^D)}\abs{\varphi(\my)}
	\, d\mlambda\, d\my
	\notag 
	\\ & 
	\overset{\eqref{commutation}, \eqref{eq:K0-def}}{\lesssim}
	\int_{K \times L} \abs{\rho(\mlambda)} 
	\norm{\nabla_{\!\mxi_0}
	\widetilde{\cL}_{j,\epsilon}^{(1)}}_{L^\infty(\S^{D-1})}
	\norm{\tilde \chi_\my^{j,\eps}(\cdot)
	T_{l_j}(h(\cdot,\mlambda))}_{L^2(K_0)} \, 
	\notag 
	\\ &
	\qquad \qquad \qquad 
	\qquad \quad\quad
	\times \abs{\varphi(\my)} \, d\mlambda \, d\my
	\, 2^{-j\left(1-\eps(D+1)/2\right)}
	\norm{v}_{L^2(\R^D)}
	\notag
	\\ &
	\lesssim \sup_{\my\in K,\mlambda\in L}\norm{\nabla_{\!\mxi_0}
	\widetilde{\cL}_{j,\epsilon}^{(1)}}_{L^\infty(\S^{D-1})}
	\norm{\tilde \chi_\cdot^{j,\eps}(\cdot)
	T_{l_j}(h(\cdot,\cdot))}_{L^2(K\times K_0 \times L)}
	\notag 
	\\ &
	\qquad \qquad \qquad
	\qquad \quad \quad
	\times 2^{-j\left(1-\eps(D+1)/2\right)}
	\norm{v}_{L^2(\R^D)}
	\notag
	\\ &
	\overset{\eqref{eq:unit-mass}, 
	\eqref{eq:tilde-L1-xi-deriv}}{\lesssim} 
	2^{\epsilon j/2}
	\norm{T_{l_j}(h)}_{L^2(K_0\times L)}
	 2^{-j\left(1-\eps(D+1)/2\right)}\norm{v}_{L^2(\R^D)}
	\notag \\ & 
	\overset{\eqref{eq:Tl-est}, \eqref{eq:lj-def}}{\lesssim}
	2^{\epsilon j/2}
	\, 2^{-j\left(1-\eps(D+1)/2\right)}
	\, 2^{\sigma j(1-p/2)}
	\norm{h}^{p/2}_{L^p(K_0\times L)}
	\norm{v}_{L^2(\R^D)}
	\notag \\ & 
	\lesssim 
	2^{-j\left(1-\eps\frac{D+1}{2}
	-\frac{\epsilon}{2}-\sigma\left(1-\frac{p}{2}\right)\right)}
	\norm{h}^{p/2}_{L^p(K_0\times L)}
	\norm{v}_{L^2(\R^D)}.
	\label{comm0term}
\end{align}

Combining \eqref{new-eq-31} 
and \eqref{comm0term}, we obtain an estimate 
for $\abs{\innb{R_1,\bar{v}}}
\le \abs{R_{1,1}}+\abs{R_{1,2}}$:
\begin{equation*}
	\norm{R_1}_{L^2(\R^{D})} \lesssim 2^{-j Z} 
	\norm{h}_{L^p(K_0\times L)}^{p/2},
\end{equation*}
where %
$Z = Z_{\epsilon, \zeta, \sigma, \eps} > 0$ 
can be read off from \eqref{new-eq-31} 
and \eqref{comm0term}:
$$
\begin{aligned}
	Z=\min\Bigl\{&
	\epsilon\frac{\alpha}{2}-\zeta\frac{\alpha}{2}
	-\sigma\left(1-\frac{p}{2}\right),
	2\zeta-\epsilon-\sigma\left(1-\frac{p}{2}\right),
	\\ &1-\eps\frac{D+1}{2}-\frac{\epsilon}{2}
	-\sigma\left(1-\frac{p}{2}\right)\Bigr\}.
\end{aligned}
$$ 

Let
$W=\{\mx_0:\operatorname{dist}(\mx_0,K_0)<1\}$.
For the admissible exponents $1<r<2$, H\"older's inequality gives
\begin{equation*}
	\tag{\ref*{comm0term}a}\label{eq:R1-local}
	\norm{R_1}_{L^r(W)}
	\leq\abs{W}^{1/r-1/2}\norm{R_1}_{L^2(\R^D)}
	\lesssim_W 2^{-jZ}\norm{h}_{L^p(K_0\times L)}^{p/2}.
\end{equation*}
The tail estimate in Remark~\ref{rem:est_compact} will be applied
to the full averaged Fourier block
in Subsection~\ref{subsec:I}.

\subsection{Estimate of $R_2$}

We will leverage the fact that 
the $L^q \to L^q$ bound of the multiplier 
operator $\cA_\psi$ with symbol 
$\psi=\widetilde{\cL}_{j,\epsilon}^{(1)}$, 
as defined in \eqref{eq:tL1-def}, is independent 
of $\mff$, $\epsilon$, and $j$, for any $1<q <\infty$ 
(see \cite[Lemma A.10]{Erceg:2023ab}).
First, by H\"older's inequality 
(with respect to $\mx_0$),
\begin{align*}
	\abs{\innb{R_2,\bar{v}}} & \le 
	\int_{K\times \R^m} 
	R_2^{(1)}(\my,\mlambda)
	\, \abs{\rho(\mlambda)}
	\, \abs{\varphi(\my)}
	\, d\mlambda \, d\my 
	\norm{v}_{L^{r'}(\R^D )},
\end{align*}
where 
\begin{align*}
	R_2^{(1)}(\my,\mlambda)
	&=\norm{\cA_{\phi_j
	\widetilde{\cL}_{j,\epsilon}^{(1)}}
	\Bigl(\chi^{j,\eps}_{\my}(\cdot)
	T^{l_j}(h(\cdot,\mlambda))\Bigr)}_{L^r(\R^D)}
	\\ & = 
	\norm{\cA_{\widetilde{\cL}_{j,\epsilon}^{(1)}}
	\circ \cA_{\phi_j}
	\Bigl(\chi^{j,\eps}_{\my}(\cdot)
	T^{l_j}(h(\cdot,\mlambda))\Bigr)}_{L^r(\R^D)}
	\\ &
	\lesssim_r
	\norm{\cA_{\phi_j}
	\Bigl(\chi^{j,\eps}_{\my}(\cdot)
	T^{l_j}(h(\cdot,\mlambda))\Bigr)}_{L^r(\R^D)}
	\quad \text{(\cite[Lemma A.10]{Erceg:2023ab})}
	\\ & 
	\lesssim_r 
	\norm{\chi^{j,\eps}_{\my}(\cdot)
	T^{l_j}(h(\cdot,\mlambda))}_{L^r(\R^D)},
\end{align*}
where the last inequality follows from the fact
that $\cA_{\phi_j}$ is bounded on $L^r(\R^D)$ uniformly 
in $j$ (see Section \ref{sec:prelim}).
Applying H\"older's inequality (with respect to $\mlambda,\my$), 
and using \eqref{eq:chi-jeps-y-prop1} with $k=0$, $\ell=r$, 
we proceed as follows:
\begin{align*}
	&\int_{K\times \R^m}
	\norm{\chi^{j,\eps}_{\my}(\cdot)
	T^{l_j}(h(\cdot,\mlambda))}_{L^r(\R^D)}
	\, \abs{\rho(\mlambda)}\, \abs{\varphi(\my)}
	\, d\mlambda \, d\my
	\\ & \quad 
	\overset{\eqref{eq:K0-def}}{=} 
	\int_{K\times L} 
	\left( \int_{K_0} \abs{\chi^{j,\eps}_{\my}(\mx_0)
	T^{l_j}(h(\mx_0,\mlambda))}^r\, d\mx_0\right)^{1/r}
	\, \abs{\rho(\mlambda)}
	\, \abs{\varphi(\my)}
	\, d\mlambda \, d\my
	\\ & \quad 
	\lesssim_{\rho,\varphi}
	\left(\int_{K_0\times L} \left(
	\int_K \left(\chi^{j,\eps}_{\my}(\mx_0)\right)^r
	d\my\right)
	\abs{T^{l_j}(h(\mx_0,\mlambda))}^r
	\, d\mx_0\, d\mlambda\right)^{1/r}
	\\ & \quad
	\overset{\eqref{eq:chi-jeps-y-prop1}}{\lesssim}
	2^{\eps j \left(D-1\right)/r'}
	\norm{T^{l_j}(h)}_{L^r(K_0\times L)}
	\overset{\eqref{eq:Tl-est}}{\lesssim} 
	2^{\eps j (D-1)/r'}l_j^{1-p/r}
	\norm{h}_{L^p(K_0\times L)}^{p/r}
	\\ & \quad
	\overset{\eqref{eq:lj-def}}{\lesssim} 
	2^{\eps j \left(D-1\right)/r'}
	2^{\sigma j \left(1-p/r\right)}
	\norm{h}_{L^p(K_0\times L)}^{p/r},
\end{align*}
bearing in mind that $r<p$ (see \eqref{eq:r-p-ass}). 
In summary, 
\begin{align}
	\|R_2\|_{L^r(\R^D)} & \lesssim 
	2^{-j\left(\sigma\left(\frac{p}{r}-1\right) 
	-\eps\frac{D-1}{r'}\right)}
	\norm{h}_{L^p(K_0\times L)}^{p/r}.
	\label{new-eq-32}
\end{align}

\subsection{Estimate of $R_3$}\label{subsec:R3}

Before proceeding, we emphasize that the analysis of 
the term $R_3$ is the only part of the proof 
where additional information on the 
spatial regularity of the drift $\mff$ 
is required.

Let us first focus on the operator $\cA_{\psi}$ 
with the multiplier $\psi=\phi_j\cL_{j,\epsilon}^{(2)}$, 
where $\cL_{j,\epsilon}^{(2)}$ is defined 
in \eqref{eq:tL-tS-def}. 
Indeed, for $\mxi_0\in \R^D$, we have
\begin{equation}\label{eq:3rd_term_symbol}
	\begin{split}
		\psi&=\phi_j(\abs{\mxi_0}) \,
		\cL_{j,\epsilon}^{(2)}
		\\ & 
		\overset{\eqref{eq:mxi-prime}}{=}
		\phi_j(\abs{\mxi_0}) \, 
		\frac{\bigl(\xi_0+\mff_{\my}(\mlambda)\cdot \mxi\bigr)\mxi}
		{\abs{\xi_0+\mff_{\my}(\mlambda)\cdot \mxi}^2
		+2^{-\epsilon j}\abs{\mxi}^2}
		\\ & 
		=\phi_j(\abs{\mxi_0}) 
		\frac{\xi_0+\mff_{\my}(\mlambda)\cdot \mxi}
		{\sqrt{\abs{\xi_0+\mff_{\my}(\mlambda) \cdot \mxi}^2
		+2^{-\epsilon j}\abs{\mxi}^2}}
		\\ & \qquad\qquad\qquad\qquad \times
		\frac{\mxi}
		{\sqrt{\abs{\xi_0+\mff_{\my}(\mlambda) \cdot \mxi}^2
		+2^{-\epsilon j}\abs{\mxi}^2}}
		\\ & =: \phi_j(\abs{\mxi_0})\psi^{(1)}(\mxi_0)\psi^{(2)}(\mxi_0).
	\end{split}	
\end{equation}
As previously noted, $\cA_{\phi_j}$ is uniformly bounded on $L^r(\R^D)$
with respect to $j$ (see Section \ref{sec:prelim}).
We therefore turn our attention to the remaining two symbols.
Both $\psi^{(1)}$ and $\psi^{(2)}$ possess homogeneity 
of degree zero, qualifying them as $L^q$-multipliers for 
all $q \in (1, \infty)$ (see \cite[Corollary 6.2.5 
and Example 6.2.6]{Grafakos:2014aa}). 
It is crucial to carefully examine how the bounds of 
the Fourier multipliers are influenced by the parameters 
$\my$, $\mlambda$, $j$, and $\epsilon$. Importantly, 
the norm of a Fourier multiplier operator remains invariant 
under a linear change of coordinates in the 
symbol (see \cite[Proposition 5.3.8]{Hytonen:2016aa} 
or \cite[Lemma 12]{Erceg:2023aa}). 
By applying the linear transformation:
\begin{equation}\label{eq:lin-change-var}
	\eta_0=\xi_0 
	+\mff_{\my}(\mlambda) \cdot \mxi,
	\quad
	\meta=2^{-\epsilon j/2} \mxi,	
\end{equation}
$\psi^{(1)}$ and $\psi^{(2)}$ 
are transformed into:
$$
\meta_0 \mapsto 
\frac{\eta_0}{\abs{\meta_0}},
\quad
\meta_0 \mapsto 
2^{\epsilon j/2}
\frac{\meta}{\abs{\meta_0}},
$$
respectively. As a result, the norm of 
the multiplier operator associated 
with $\psi^{(1)} \psi^{(2)}$ is of order $2^{\epsilon j/2}$. 
This follows from the fact that the norm of a multiplier 
operator corresponding to a product of two functions can 
be estimated separately in terms of the norms of 
the individual multiplier operators and then 
combined multiplicatively 
(see \cite[Proposition 2.5.13]{Grafakos:2014aa}). 

Consequently, the multiplier operator $\cA_{\psi}$, 
associated with the symbol given in \eqref{eq:3rd_term_symbol}, 
is bounded from $L^q$ to itself for any $q \in (1, \infty)$. 
The operator norm is of order 
$2^{j\epsilon/2}$:
$$
\norm{\cA_{\psi}}_{L^q\to L^q}
\lesssim 2^{j \epsilon/2}.
$$  
Therefore, for the $R_3$ term in \eqref{eq:init-relation}, 
we can first apply H\"older's inequality and 
then use this bound to obtain the following result:
\begin{equation*}
	\begin{split}
		\abs{\innb{R_3,\bar{v}}} & \le
		\int\limits_{K\times \R^m}
		\norm{\cA_{\phi_j
		\cL_{j,\epsilon}^{(2)}}
		\Bigl(\bigl(\mff(\cdot,\mlambda)
		-\mff_{\my}(\mlambda)\bigr) \chi^{j,\eps}_{\my}(\cdot)
		h(\cdot,\mlambda) \Bigr)}_{L^r(\R^D)}
		\\ & \qquad\qquad\qquad \qquad
		\qquad \qquad \qquad  \times
		\abs{\rho(\mlambda)}\, \abs{\varphi(\my)}  
		\, d\mlambda\, d\my \, \norm{v}_{L^{r'}(\R^D)}
		\\ & \lesssim_{r} 
		2^{j\frac{\epsilon}{2}}
		\int\limits_{K\times \R^m}
		\norm{\bigl(\mff(\cdot,\mlambda)
		-\mff_{\my}(\mlambda)\bigr) \chi^{j,\eps}_{\my}(\cdot)
		h(\cdot,\mlambda)}_{L^r(\R^D)}
		\\ & \qquad \qquad \qquad 
		\qquad\qquad \qquad \qquad  \times
		\abs{\rho(\mlambda)}\, \abs{\varphi(\my)}  
		\, d\mlambda\, d\my \, \norm{v}_{L^{r'}(\R^D)},
	\end{split}
\end{equation*} where, by H\"older's inequality,
\begin{align*}
	R_3^{(1)} & :=\int_{K\times \R^m}
	\norm{\bigl(\mff(\cdot,\mlambda)
	-\mff_{\my}(\mlambda)\bigr) \chi^{j,\eps}_{\my}(\cdot)
	h(\cdot,\mlambda)}_{L^r(\R^D)}
	\, \abs{\rho(\mlambda)}
	\, \abs{\varphi(\my)}
	\, d\mlambda \, d\my  
	\\ & \lesssim_{\rho,\varphi}
	\left(\int_{K_0\times L} \left(
	\int_K  \left(\Delta^\mff_{\mx,\mlambda}(\my)
	\chi^{j,\eps}_{\my}(\mx_0)\right)^r
	d\my\right)
	\abs{h(\mx_0,\mlambda)}^r
	\, d\mx_0\, d\mlambda\right)^{1/r},
\end{align*}
with $\Delta^\mff_{\mx,\mlambda}(\my)
:=\abs{\mff(\mx,\mlambda)-\mff(\my,\mlambda)}$, 
recalling \eqref{eq:fy-def}. 
We multiply and divide the expression under 
the integral sign by the factor 
$\abs{\mx-\my}^{(s+\frac{D-1}{q})r} e^{x_0^2 r}$, 
where $s\in (0,\bar{s})$, and
$q\in (p',\bar p)$
is chosen such that $1/r=1/q+1/p$.
Recall that $1/p'+1/p=1$ and that 
$p$, $\bar s$ and $\bar p$ are 
fixed in ({\bf H1}) and ({\bf H2}).
The choice of $q$ is possible if 
$r\in \bigl(1,\frac{p\bar p}{p+\bar p}\bigr)$, 
a condition that will constrain the final value of $r_0$. 
Note that since $\bar p>p'$ (see ({\bf H2})), 
we automatically have 
$p\geq \frac{p\bar p}{p+\bar p}>1$.

Then, we apply the generalized H\"older 
inequality with $1/r = 1/q + 1/p$ and 
with respect to variables 
$\mx_0$ and $\my$, to obtain
\begin{align*}
	(R_3^{(1)})^r & 
	\lesssim 
	\int_L\bigl(R_3^{(1,1)}(\mlambda)\bigr)^r
	\bigl(R_3^{(1,2)}(\mlambda)\bigr)^r 
	\,d\mlambda,
\end{align*}
where
\begin{align*}
	R_3^{(1,1)}(\mlambda) 
	&:= \norm{\frac{e^{-x_0^2}
	\Delta^\mff_{\mx,\mlambda}(\my)}
	{\abs{\mx-\my}^{(s+\frac{D-1}{q})}}}_{L^q(K_0\times K)},
	\\ 
	R_3^{(1,2)}(\mlambda) 
	&:=\norm{\abs{\mx-\my}^{s+\frac{D-1}{q}} e^{x_0^2} 
	\chi^{j,\eps}_\my(\mx_0)
	h(\mx_0,\mlambda)}_{L^p(K_0\times K)}.
\end{align*}

Using Fubini's theorem and the fact 
that the Gaussian kernel $x_0 \mapsto 
e^{-x_0^2}$ is integrable on $\R$, 
we can easily obtain the following 
result for any $\mlambda \in L$:
\begin{align*}
	R_3^{(1,1)}(\mlambda)
	\lesssim \norm{\frac{\Delta^\mff_{\mx,\mlambda}(\my)}
	{\abs{\mx-\my}^{(s+\frac{D-1}{q})}}}_{L^q(K\times K)}
	= \norm{\mff(\cdot,\mlambda)}_{W^{s,q}(K)}
	\overset{({\bf H2})}{\lesssim}_{\!\!\! K,L,s} 1, 
\end{align*}
since $q<\bar p$
(see Remark \ref{rem:Sob_Omega_bdd}). 
Note that even if $K$ is not regular (as required 
for the embedding to hold), one can always choose 
a regular set $\tilde K\supset K$ (for instance, a ball) 
and carry out the above argument with 
$\tilde K$, for which the 
assumption ({\bf H2}) still applies. 

To obtain the estimate of $R_3^{(1,2)}$, 
let us note that
\begin{align*}
	\int_K \abs{\mx-\my}^{(s+\frac{D-1}{q})p} & e^{x_0^2 p}
	(\chi^{j,\eps}_\my(\mx_0))^p \,d\my 
	\\ & 
	\overset{\eqref{eq:chi-jeps-y-def}}{=}
	C^{2p} 2^{\eps j(D-1)p} 
	\int_K \abs{\mx-\my}^{(s+\frac{D-1}{q})p}
	e^{-\abs{2^{\eps j}(\mx-\my)}^2 p} \, d\my
	\\ & 
	\mathrel{\leq} C^{2p} 2^{\eps j(D-1)(p-1)}
	2^{-\eps j(s+\frac{D-1}{q})p}
	\int_{\R^{D-1}} \abs{\mz}^{(s+\frac{D-1}{q})p}
	e^{-p \abs{\mz}^2} \,d\mz
	\\ &\overset{\eqref{eq:gaussian-q-moment}}
	{\lesssim}_{\!\! p,\bar{p},\bar{s}} 
	2^{-\eps j s p} 2^{\eps j(D-1)\frac{p}{r'}},
\end{align*}
where in the second step we
used the change of variables
$\mz=2^{\eps j}(\my-\mx)$
(see \eqref{eq:z-var}) and enlarged the
domain to $\R^{D-1}$. Thus, we have
(for a.e.~$\mlambda\in L$):
\begin{align*}
	R_3^{(1,2)}(\mlambda) 
	\lesssim 2^{-\eps j s} 2^{\eps j(D-1)\frac{1}{r'}}
	\norm{h(\cdot,\mlambda)}_{L^p(K_0)}.
\end{align*}

Returning to $R_3^{(1)}$, it follows that
\begin{align*}
	R_3^{(1)} &\lesssim  
	2^{-\eps j s} 2^{\eps j(D-1)\frac{1}{r'}}
	\left( \int_L
	\norm{h(\cdot,\mlambda)}_{L^p(K_0)}^r
	\,d\mlambda \right)^{1/r} 
	\\ & 
	\lesssim 2^{-\eps j s} 
	2^{\eps j(D-1)\frac{1}{r'}}
	\norm{h}_{L^p(K_0\times L)},
\end{align*}
where another application of H\"older's 
inequality (this time in the $\mlambda$ variable) 
is used, taking advantage of the 
compactness of $L$ and of \eqref{eq:r-p-ass}. 
Thus, our analysis of $R_3$ reveals that
\begin{equation}\label{new-eq-33}
	\norm{R_3}_{L^r(\R^D)}
	\lesssim 
	2^{-j\left(\eps\left(s-\frac{D-1}{r'}\right)
	-\frac{\epsilon}{2}\right)}
	\norm{h}_{L^{p}(K_0\times L)}.
\end{equation}
Although the implicit multiplicative constant 
in this estimate depends on $s$ (and 
blows up as $s\to\bar{s}$; 
see Remark \ref{rem:Sob_Omega_bdd}), 
it suffices to analyze the exponent 
in \eqref{new-eq-33} at $s=\bar{s}$. Indeed, for any
$\beta<\eps\left(\bar{s}-\frac{D-1}{r'}\right)
-\frac{\epsilon}{2}$, one can choose $s<\bar{s}$ such that
$\beta<\eps\left(s-\frac{D-1}{r'}\right)-\frac{\epsilon}{2}$.

\subsection{Estimate of $R_4$}\label{subsec:R4}

Similar to the previous term, 
the norm of the multiplier 
operator $\mathcal{A}_\psi$ with symbol
$\psi=\phi_j \cL_{j,\epsilon}^{(3)}$, 
where $\mathcal{L}_{j,\epsilon}^{(3)}$ 
is defined in \eqref{eq:tL-tS-def}, is of the 
order of $2^{j(\epsilon/2 - 1)}$. 
Specifically, we have:
\begin{equation}\label{eq:A-bound-R4}
	\norm{\mathcal{A}_\psi}_{L^q \to L^q} 
	\lesssim 
	2^{j(\epsilon/2-1)}.
\end{equation}
Let us discuss the reasoning behind this 
estimate in greater detail.  We have
\begin{align*}
	\psi(\mxi_0)
	& =
	\frac{\phi_j(\abs{\mxi_0})}{\abs{\mxi_0}}
	\frac{\xi_0+\mff_{\my}(\mlambda)\cdot \mxi}
	{\sqrt{\abs{\xi_0+\mff_{\my}(\mlambda) 
	\cdot \mxi}^2+2^{-\epsilon j}\abs{\mxi}^2}}
	\frac{\abs{\mxi_0}}
	{\sqrt{\abs{\xi_0+\mff_{\my}(\mlambda) 
	\cdot \mxi}^2+2^{-\epsilon j}\abs{\mxi}^2}}
	\\ & =:
	\tilde{\psi}^{(0)}(\mxi_0)\psi^{(1)}(\mxi_0)
	\psi^{(3)}(\mxi_0).
\end{align*}
The symbol $\psi^{(1)}$ also appears 
in \eqref{eq:3rd_term_symbol}, and so
the norm bound of the corresponding
multiplier operator is $\lesssim 1$. 

The function $\tilde\psi^{(0)}$ we can write as 
$\tilde\psi^{(0)}(\mxi_0) 
= 2^{-j}\frac{\phi\left(\abs{2^{-j}\mxi_0}\right)}
{\abs{2^{-j}\mxi_0}}$ 
(see \eqref{eq:phi-j-def}).
Then using the fact that the multiplier norm 
is invariant under dilations 
(see \cite[Proposition 2.5.14]{Grafakos:2014aa}), the 
norm of the associated Fourier multiplier 
operator is $\lesssim 2^{-j}$.

Using the linear change of 
variables \eqref{eq:lin-change-var}, 
the analysis of $\psi^{(3)}$ is reduced to the analysis of:
\begin{align*}
	\tilde{\psi}^{(3)}(\meta_0)
	& =\frac{1}{\abs{\meta_0}}
	\sqrt{\left(\eta_0-2^{\epsilon j/2}\,
	\mff_{\my}(\mlambda) \cdot \meta\right)^2 
	+2^{\epsilon j}\abs{\meta}^2}
	\\ & 
	= \sqrt{\frac{\eta_0^2}{\abs{\meta_0}^2}+S
	+2^{\epsilon j}\frac{\abs{\meta}^2}{\abs{\meta_0}^2}},
\end{align*}
where
\begin{align*}
	S & =
	-2\, 2^{\epsilon j/2}\frac{\eta_0}{\abs{\meta_0}}
	\sum_{k=1}^{D-1}\mff_{\my,k}(\mlambda)
	\frac{\eta_k}{\abs{\meta_0}}
	\\ & \qquad 
	+2^{\epsilon j}\sum_{k=1}^{D-1}
	\mff_{\my,k}(\mlambda)^2\frac{\eta_k^2}{\abs{\meta_0}^2}
	+ 2\, 2^{\epsilon j}
	\sum_{1 \leq k < l \leq D-1}
	\mff_{\my,k}(\mlambda)\mff_{\my,l}(\mlambda)
	\frac{\eta_k}{\abs{\meta_0}}
	\frac{\eta_l}{\abs{\meta_0}}.
\end{align*}
Here, we have expanded the square of the difference, 
and then the square of the inner product via 
the simple identity $\left(\sum_k a_k b_k\right)^2 
= \sum_k a_k^2 b_k^2+2\sum_{k < l} a_k b_k a_l b_l$.

Each summand in the second line is homogeneous of 
order zero and thus acts
as an $L^q$-multiplier
for any $q \in (1, \infty)$. 
This introduces the worst norm bound of order 
$2^{\epsilon j}$ (see \cite[Corollary 6.2.5 
and Example 6.2.6]%
{Grafakos:2014aa}).
The conclusion is also based on the facts that 
$\my \in K$, $\mlambda \in L$, and we have 
employed condition ({\bf H2}).  
Now, given that $\tilde{\psi}^{(3)}$ is a composition 
of functions, we apply the Fa\`a di Bruno (chain rule) 
formula \cite{Constantine:1996aa} for 
higher-order derivatives and the Marcinkiewicz multiplier 
theorem \cite[Corollary 6.2.5]{Grafakos:2014aa}, following the same 
approach as in the proof of \cite[Lemma 10]{Erceg:2023aa}. 
This allows us to conclude that the bound of the Fourier multiplier 
operator associated with $\tilde{\psi}^{(3)}$ is of 
order $2^{\epsilon j/2}$. By the preceding discussion, the same 
bound also holds for $\psi^{(3)}$. This concludes 
the proof of \eqref{eq:A-bound-R4}.

Hence, using \eqref{eq:A-bound-R4}, we may 
estimate $R_4$ as follows:
\begin{align*}
	\abs{\innb{R_4,\bar{v}}} &\lesssim 
	\int\limits_{K\times \R^m}
	\norm{\cA_{\phi_j
	\cL_{j,\epsilon}^{(3)}}
	\Bigl(\bigl(1,\mff(\cdot,\mlambda)\bigr)
	\cdot \nabla_{\!\mx_0}\chi^{j,\eps}_{\my}(\cdot) 
	h(\cdot,\mlambda)\Bigr)}_{L^r(\R^D)}
	\\ & \qquad\qquad\qquad \qquad
	\qquad \qquad \qquad  \times
	\abs{\rho(\mlambda)}\, \abs{\varphi(\my)}  
	\, d\mlambda\, d\my \, \norm{v}_{L^{r'}(\R^D)}
	\\ & \lesssim_r
	2^{j(\epsilon/2-1)} 
	\int\limits_{K\times \R^m}
	\norm{\bigl(1,\mff(\cdot,\mlambda)\bigr)
	\cdot \nabla_{\!\mx_0}\chi^{j,\eps}_{\my}(\cdot) 
	h(\cdot,\mlambda)}_{L^r(\R^D)}
	\\ & \qquad\qquad\qquad \qquad
	\qquad \qquad \qquad  \times
	\abs{\rho(\mlambda)}\, \abs{\varphi(\my)}  
	\, d\mlambda\, d\my \, \norm{v}_{L^{r'}(\R^D)}
	\\ & \leq 2^{j(\epsilon/2-1)}
	\left(R_4^{(1)}+R_4^{(2)}\right) 
	\norm{v}_{L^{r'}(\R^D)},
\end{align*}
where
\begin{align*}
	R_4^{(1)} &:= \int\limits_{K\times \R^m}
	\norm{\mff(\cdot,\mlambda)
	\cdot \nabla_{\!\mx}\chi^{j,\eps}_{\my}(\cdot) 
	h(\cdot,\mlambda)}_{L^r(\R^D)}
	\abs{\rho(\mlambda)}\, \abs{\varphi(\my)}  
	\, d\mlambda\, d\my,
	\\ R_4^{(2)} &:= \int\limits_{K\times \R^m}
	\norm{\pa_{x_0}\chi^{j,\eps}_{\my}(\cdot) 
	h(\cdot,\mlambda)}_{L^r(\R^D)}
	\abs{\rho(\mlambda)}\, \abs{\varphi(\my)}  
	\, d\mlambda\, d\my.
\end{align*}

Applying first H\"older's inequality and then using 
({\bf H2}) and \eqref{eq:chi-jeps-y-prop1} 
for $k=1$ and $\ell=r$, it follows that
\begin{align*}
	R_4^{(1)} & \lesssim_{\rho,\varphi}
	\left(\int_{K_0\times L} \left(
	\int_K\abs{\mff(\mx,\mlambda)\cdot
	\nabla_{\!\mx}\chi^{j,\eps}_{\my}(\mx_0)}^r
	d\my\right)\abs{h(\mx_0,\mlambda)}^r
	\, d\mx_0\, d\mlambda\right)^{1/r}
	\\ & \lesssim_\mff 
	2^{\eps j\left( \left(D-1\right)
		/r'+1\right)}\norm{h}_{L^r(K_0\times L)}.
\end{align*}

Since 
$$
\abs{\pa_{x_0}\chi^{j,\eps}_{\my}(\mx_0)}
= 2\abs{x_0}\abs{\chi^{j,\eps}_{\my}(\mx_0)}
\lesssim \abs{\chi^{j,\eps}_{\my}(\mx_0)},
$$
where we have used that $\abs{x_0}$ 
is bounded due to $\mx_0\in K_0$ (see \eqref{eq:K0-def}), 
we can apply the same argument to 
$R_4^{(2)}$ (with the change that now 
\eqref{eq:chi-jeps-y-prop1} is used with $k=0$). 
The result is $R_4^{(2)}\lesssim 
2^{\eps j\left( \left(D-1\right)/r'\right)}
\norm{h}_{L^r(K_0\times L)}$. 

Hence,
\begin{equation}\label{new-eq-34}
	\norm{R_4}_{L^r(\R^D)} \lesssim 
	2^{-j\left(1-\frac{\epsilon}{2}
	-\eps\left(\frac{D}{r'}
	+\frac{1}{r}\right)\right)}
\norm{h}_{L^r(K_0\times L)}.
\end{equation}

\subsection{Estimate of $R_5$}\label{subsec:R5}

To simplify the presentation, we will introduce additional 
notation. Define $\eps_j$, $\tilde{\phi}_j$, 
and $\cL_{j,\epsilon}$ as follows:
\begin{align*}
	& \eps_j =2^{-\eps j} 
	\,\, \text{(scaling parameter 
	from \eqref{eq:chi-jeps-y-def})},
	\qquad 
	\tilde\phi_j=\sqrt{\phi_j},
	\\ & 
	\cL_{j,\epsilon}(\my,\mxi_0,\mlambda)
	=\frac{\xi_0+\mff_{\my}(\mlambda) \cdot \mxi}
	{\abs{\xi_0+\mff_{\my}(\mlambda) \cdot \mxi}^2
	+2^{-\epsilon j}\abs{\mxi}^2}.
\end{align*}
Recalling \eqref{eq:mxi-prime}, we 
have $\cL_{j,\epsilon}(\my,\mxi_0',\mlambda)
=\widetilde{\cL}_{j,\epsilon}^{(3)}$. 

Using the new notation, let us express
the mollifier-like function $\chi^{j,\eps}_{\my}$
defined in \eqref{eq:chi-jeps-y-def} as follows:
\begin{equation}\label{eq:J-def-new}
	\chi^{j,\eps}_{\my}(\mx_0) 
	= \pi^{-1/2}e^{-x_0^2} \, \frac{1}{\eps_j^{\left(D-1\right)}}
	\mathcal{G}\left(\frac{\mx-\my}{\eps_j}\right),
	\quad \mathcal{G}(\mz)
	=\pi^{-(D-1)/2}e^{-\abs{\mz}^2}.
\end{equation}
After a change of variables \eqref{eq:z-var}, 
$\tilde R_5:=\frac{2\pi i}{(-1)^{\abs{\mkappa}}}
\pi^{1/2}R_5$ becomes 
\begin{align*}
	\innb{\tilde R_5,\bar{v}} 
	& =\int e^{-2\pi i \mxi_0\cdot \mx_0}
	\frac{\phi_j(\abs{\mxi_0})}{\abs{\mxi_0}}
	\, e^{-x_0^2}\, \mathcal{G}(\mz)\, 
	\pa_{\mlambda}^\mkappa\Bigl(\rho(\mlambda)
	\cL_{j,\epsilon}(\mx+\eps_j\mz, \mxi_0',\mlambda)\Bigr)
	\\ & \qquad \qquad\qquad\qquad\qquad\qquad \times 
	\overline{\widehat{v}(\mxi_0)}
	\, \varphi\left(\mx+\eps_j\mz\right)
	\, d\gamma(\mx_0,\mlambda) \, d\mz \, d\mxi_0.
\end{align*}

Defining  
\begin{align}
	M_{j,\epsilon,\eps}(\mx_0,\mz)
	& =\cF_{\mxi_0}^{-1}
	\Biggl(\tilde \phi_j\left(\abs{\cdot_{\mxi_0}}\right)
	\int e^{-2\pi i \cdot_{\mxi_0}\cdot \tilde{\mx}_0}
	\pa_{\mlambda}^\mkappa
	\left(
	\rho(\mlambda)\,
	\cL_{j,\epsilon}\Bigl(\tilde{\mx}+\eps_j \mz, 
	\frac{\cdot_{\mxi_0}}{\abs{\cdot_{\mxi_0}}},\mlambda\Bigr)
	\right) 
	\notag \\ & \qquad\qquad\qquad\qquad\qquad \times 
	e^{-\tilde x_0^2}\varphi(\tilde{\mx}+\eps_j {\mz}) 
	\, d\gamma\left(\tilde\mx_0,\mlambda\right)\Biggr)(\mx_0),
	\label{meas-trans}
\end{align}
we can express $\tilde R_5$ as follows:
\begin{equation}\label{eq:tildeR5-tmp}
	\innb{\tilde R_5,\bar{v}} 
	=\int \abs{\mxi_0}^{-1} 
	\tilde{\phi}_j(\abs{\mxi_0}) \, \mathcal{G}(\mz) \,
	\cF_{\mx_0}\bigl(M_{j,\epsilon,\eps}
	(\cdot,\mz)\bigr)(\mxi_0)
	\, \overline{\widehat v(\mxi_0)} \, d\mz \, d\mxi_0.
\end{equation}
To avoid a notation conflict with the variable $\mx_0$ at which 
\eqref{meas-trans} is evaluated, we have introduced 
$\tilde{\mx}_0$ as a new integration variable for 
the measure $\gamma$ inside the inverse Fourier transform. 
Additionally, the integration domains have been omitted for 
simplicity, as they can be inferred from the context (see
also \eqref{eq:K0-def} and the definition of $R_5$).

We will demonstrate that the mapping
$\mx_0\mapsto M_{j,\epsilon,\eps}(\mx_0,\mz)$ 
is an $L^r$ function. Moreover, when weighted against
$\mathcal{G}(\mz)$, 
$M_{j,\epsilon,\eps}(\cdot,\mz)$ remains bounded 
independently of $\mz$.  To achieve this, we fix 
$V \in C^\infty_c(\R^D)$ and $\mz \in \R^{D-1}$, 
and examine
\begin{align}
	\notag 
	& \abs{\innb{M_{j,\epsilon,\eps}(\cdot,\mz),V}}
	\\ & \qquad
	=\Biggl | \int e^{-2\pi i \mxi_0\cdot
	\left(\tilde{\mx}_0-\mx_0\right)}
	\, \tilde{\phi}_j\left(\abs{\mxi_0}\right)\,
	\pa_{\mlambda}^\mkappa
	\Bigl(\rho(\mlambda)\, \cL_{j,\epsilon}
	(\tilde{\mx}+\eps_j \mz,\mxi_0', \mlambda)\Bigr)
	\notag \\ & \qquad\qquad\qquad\qquad\qquad \times
	e^{-\tilde x_0^2}\varphi(\tilde{\mx}+\eps_j {\mz}) 
	\, V(\mx_0) 
	\, d\gamma\left(\tilde\mx_0,\mlambda\right) 
	\,d\mx_0 \, d\mxi_0  \Biggr |
	\notag \\ & \qquad
	\overset{\eqref{eq:K0-def}}{\leq} 
	\sup\limits_{(\tilde{\mx}_0,\mlambda)\in K_0\times L}
	\abs{\widetilde M_{j,\epsilon,\eps}(\tilde \mx_0,\mlambda;\mz)}
	\, \norm{\gamma}_{\cM(\R^D\times \R^m)},
	\label{m-1} 
\end{align} where 
\begin{align*}
	\widetilde M_{j,\epsilon,\eps}(\tilde \mx_0,\mlambda;\mz)
	& =\int e^{-2\pi i \mxi_0\cdot \left(\tilde{\mx}_0-\mx_0\right)}\,
	\tilde{\phi}_j\left(\abs{\mxi_0}\right)
	e^{-\tilde x_0^2}\varphi(\tilde{\mx}+\eps_j {\mz}) 
	\\ & \qquad\qquad\times 
	\pa_{\mlambda}^\mkappa
	\Bigl(\rho(\mlambda)\cL_{j,\epsilon}
	(\tilde{\mx}+\eps_j \mz,\mxi_0', \mlambda)\Bigr)
	\, V(\mx_0) \, d\mx_0 \, d\mxi_0.
\end{align*}
Applying the change of variables
$$
\mxi_0=2^j \meta_0,
\quad 
\mx_0=2^{-j}\my_0,
\quad 
d\mx_0\, d\mxi_0
=d\my_0\,d\meta_0,
$$
we obtain
\begin{align*}
	\widetilde M_{j,\epsilon,\eps}(\tilde \mx_0,\mlambda;\mz)
	& =\int e^{-2\pi i 2^j\meta_0
	\cdot \left(\tilde{\mx}_0-2^{-j}\my_0\right)}
	\, \tilde{\phi}(\abs{\meta_0})
	e^{-\tilde x_0^2}\varphi(\tilde{\mx}+\eps_j {\mz}) 
	\\ & \qquad\qquad\times
	\pa_{\mlambda}^\mkappa
	\Bigl(\rho(\mlambda)\cL_{j,\epsilon}
	(\tilde{\mx}+\eps_j \mz,\meta_0',\mlambda)\Bigr)
	\, V(2^{-j}\my_0) \,d\my_0\, d\meta_0,
\end{align*}
and so
\begin{equation}\label{eq:tildeM-est-tmp}
	\begin{split}
		&\abs{\widetilde M_{j,\epsilon,\eps}
		(\tilde \mx_0,\mlambda;\mz)}
		\le \norm{\varphi}_{L^\infty(\R^{D-1})}
		2^{j D/r'} \norm{V}_{L^{r'}(\R^D)}
		\biggl(\int_{\R^D} \abs{J}^r\, d\my_0\biggr)^{1/r}, 
		\\ & 
		J:=\int_{\R^D} e^{-2\pi i \meta_0\cdot 
		\left(2^j \tilde{\mx}_0-\my_0\right)}\, 
		\tilde{\phi}(\abs{\meta_0})\,
		\pa_{\mlambda}^\mkappa
		\Bigl(\rho(\mlambda)\cL_{j,\epsilon}
		(\tilde{\mx}+\eps_j \mz,\meta_0',\mlambda)\Bigr) 
		\, d\meta_0,
	\end{split}
\end{equation}
recalling that
$\mw:=\tilde{\mx} + \eps_j \mz$
is confined within the compact set $K$ 
(since $\tilde\mx+\eps_j\mz\in\supp\varphi$).
Here, $\meta_0'$ is defined as in \eqref{eq:mxi-prime}: $\meta_0' 
= \frac{\meta_0}{\abs{\meta_0}}$, $\meta_0'=(\eta_0',\meta')$. 
Further, recalling \eqref{eq:phi-j-def}, we have 
$\phi_j(\abs{\mxi_0}) := \phi\left(2^{-j}\abs{\mxi_0}\right)
=\phi(\abs{\meta_0})$, and $\tilde{\phi}$ is 
defined by $\tilde\phi = \sqrt{\phi}$.

Let us estimate $\abs{J}$.
Note that
\begin{align*}
	J & =\int_{\R^D}\!
	\frac{\partial_{\meta_0}^{2\malpha}
	e^{2\pi i \meta_0\cdot \left(\my_0-2^j \tilde{\mx}_0\right)}}
	{\bigl(2\pi i \left(\my_0-2^j \tilde{\mx}_0\right)
	\bigr)^{2\malpha}}
	\, \tilde{\phi}(\abs{\meta_0})\,
	\pa_{\mlambda}^\mkappa
	\Bigl(\rho(\mlambda)\cL_{j,\epsilon}
	(\tilde{\mx}+\eps_j \mz,\meta_0',\mlambda)\Bigr) 
	\, d\meta_0
	\\ & = \frac{(-1)^{\abs{2\malpha}}}{(2\pi i)^{\abs{2\malpha}}}
	\! \int_{\R^D}\!\!
	\frac{e^{2\pi i \meta_0\cdot\left(\my_0-2^j \tilde{\mx}_0\right)}}
	{\left(\my_0-2^j \tilde{\mx}_0\right)^{2\malpha}}
	\partial_{\meta_0}^{2\malpha}\!
	\biggl(\tilde{\phi}(\abs{\meta_0})
	\pa_{\mlambda}^\mkappa
	\Bigl(\rho(\mlambda)\cL_{j,\epsilon}
	(\tilde{\mx}+\eps_j \mz,\meta_0',\mlambda)\Bigr)
	\biggr)\, d\meta_0,
\end{align*}
for any multi-index $\malpha=(\alpha_0,\alpha_1,\ldots,\alpha_{D-1}) 
\in \seq{0,1}^D$.  When we want to emphasize the 
role of $\malpha$ on the right-hand 
side, we will denote it by $J_{\malpha}$. Writing
\begin{align*}
	&\partial_{\meta_0}^{2\malpha}\!
	\left(\tilde{\phi}(\abs{\meta_0})
	\pa_{\mlambda}^\mkappa
	\Bigl(\rho(\mlambda)\cL_{j,\epsilon}
	(\tilde{\mx}+\eps_j \mz,\meta_0',\mlambda)\Bigr)
	\right)
	\\ & \qquad 
	=\sum_{\mbeta\leq 2 \malpha}
	\begin{pmatrix}
		2\malpha\\
		\mbeta
	\end{pmatrix}
	\partial_{\meta_0}^{\mbeta}
	\tilde{\phi}(\abs{\meta_0})
	\, \partial_{\meta_0}^{2\malpha-\mbeta}
	\pa_{\mlambda}^\mkappa
	\Bigl(\rho(\mlambda)\cL_{j,\epsilon}
	(\tilde{\mx}+\eps_j \mz,\meta_0',\mlambda)\Bigr),
\end{align*}
we arrive at
\begin{align*}
	\abs{J_\malpha} & \lesssim
	\frac{1}{\left(\my_0-2^j \tilde{\mx}_0\right)^{2\malpha}}
	\\ & \qquad \times 
	\int_{\R^D}\sum_{\mbeta\leq 2 \malpha}
	\abs{\partial_{\meta_0}^{\mbeta}
	\tilde{\phi}(\abs{\meta_0})}
	\, \abs{\partial_{\meta_0}^{2\malpha-\mbeta}
	\pa_{\mlambda}^\mkappa
	\Bigl(\rho(\mlambda)\cL_{j,\epsilon}
	(\tilde{\mx}+\eps_j \mz,\meta_0',\mlambda)\Bigr)}
	\, d\meta_0.
\end{align*}
Note that the order of $2\malpha-\mbeta$ is bounded 
by $2D$. Given that $\alpha_i \in \seq{0,1}$, it follows 
that $2\alpha_i \in \seq{0,2}$. 
Consequently, $2\alpha_i-\beta_i \in \seq{0,1,2}$. 
Define
\begin{equation}\label{eq:C-kappa-def}
	C_{j,\epsilon,\eps}^{\mkappa}:= 
	\sup_{\substack{\mw \in K, 
	\, \abs{\meta_0} \in \supp(\tilde{\phi}), 
	\, \mlambda\in \R^m \\
	\tilde \mbeta \in \seq{0,1,2}^D,
	\, \abs{\tilde \mbeta}\le 2D}}
	\abs{\partial_{\meta_0}^{\tilde\mbeta}
	\pa_{\mlambda}^\mkappa
	\Bigl(\rho(\mlambda)\cL_{j,\epsilon}
	(\mw,\meta_0',\mlambda)\Bigr)}.
\end{equation}
Then, for any $\tilde \mx_0,\my_0\in \R^D$, 
$\mlambda\in \R^m$, $\mz\in \R^{D-1}$, 
and any multi-index $\malpha\in \seq{0,1}^D$,
\begin{equation}\label{eq:tildeM-est-tmp2}
	\begin{split}
		\abs{J} & \le
		C_{j,\epsilon,\eps}^{\mkappa}\frac{1}
		{\left(\my_0-2^j \tilde{\mx}_0\right)^{2\malpha}}
		\sum_{\mbeta\leq 2 \malpha} 
		\left(\int_{\R^D}\abs{\partial_{\meta_0}^{\mbeta}
		\tilde{\phi}(\abs{\meta_0})}	
		\, d\meta_0 \right)
		\\ & \lesssim_{\phi}
		C_{j,\epsilon,\eps}^{\mkappa}\frac{1}
		{\left(\my_0-2^j \tilde{\mx}_0\right)^{2\malpha}} 
		=C_{j,\epsilon,\eps}^{\mkappa}
		\prod_{k=0}^{D-1}
		\frac{1}{\left(y_k-2^j \tilde{x}_k\right)^{2\alpha_k}}.
	\end{split}
\end{equation}

Consider now the multi-dimensional 
integral $\int_{\R^D} \abs{J}^r\, d\my_0$ 
from \eqref{eq:tildeM-est-tmp}:
$$
\int_{\R^D} \abs{J}^r\, d\my_0
=\int_{\R}\int_{\R}\cdots \int_{\R}
\abs{J}^r\, dy_0\, dy_1\cdots \, dy_{D-1}.
$$
For the $k$-th iterated integral, we split 
the domain $\R$ into $A_k\cup B_k$, where 
$$
A_k=\seq{y_k\in \R: \abs{y_{k}-2^j\tilde{x}_{k}}\leq 1},
\quad
B_k=\seq{y_k\in \R: \abs{y_{k}-2^j\tilde{x}_{k}}> 1},
$$
so that 
$$
\int_{\R^D} \abs{J}^r \, d\my_0
=\left(\int_{A_0}+\int_{B_0}\right)
\cdots \left(\int_{A_{D-1}}+\int_{B_{D-1}}\right)
\abs{J}^r\, dy_0\cdots \, dy_{D-1}.
$$ 
Let us introduce a multi-index notation 
$\malpha = (\alpha_0, \alpha_1, \ldots, \alpha_{D-1})
\in \seq{0,1}^D$, where each $\alpha_k$ takes a value 
of $0$ or $1$. This notation will help us 
keep track of all the different integrals where each 
variable $y_k$ is integrated over either $A_k$ or $B_k$. 
We define $\alpha_k$ as follows:  
If the integration with respect to the variable $y_k$ is
over the set $A_k$, we define $\alpha_k=0$; whereas if
$y_k$ is integrated over the set $B_k$, we assign
$\alpha_k=1$.

We can use the multi-index $\malpha$ to split 
$\int_{\R^D} \abs{J}^r\, d\my_0$ into a sum of $2^D$ terms, each 
corresponding to a different combination of the sets 
$A_k$ and $B_k$ for the variables $y_k$. Specifically:
$$
\int_{\R^D} \abs{J}^r \, d\my_0 
= \sum_{\malpha \in \seq{0,1}^D}
\int_{S_{\alpha_0}}\int_{S_{\alpha_1}} 
\cdots 
\int_{S_{\alpha_{D-1}}} 
\abs{J}^r \, dy_0 \, dy_1 \, \cdots \, dy_{D-1},
$$
where each $S_{\alpha_k}\subset \R$ is given by
$$
S_{\alpha_k} =
\begin{cases}
	A_k & \text{if } \alpha_k = 0, \\
	B_k & \text{if } \alpha_k = 1.
\end{cases}
$$
Given a multi-index $\malpha$ from the above sum, 
we can apply the estimate in \eqref{eq:tildeM-est-tmp2} 
for $\abs{J}^r$ corresponding to this specific 
choice of $\malpha$, leading to:
\begin{multline}
	\int_{\R^D} \abs{J}^r \, d\my_0
	\le (C_{j,\epsilon,\eps}^{\mkappa})^r
	\!\!\!\sum_{\malpha\in \seq{0,1}^D}
	\prod_{k=0}^{D-1}\left(\int_{S_{\alpha_k}}
	\frac{1}{\left(y_k-2^j \tilde{x}_k\right)^{2 r\alpha_k}}
	\, dy_k\right)
	\\ =(C_{j,\epsilon,\eps}^{\mkappa})^r
	\!\!\! \sum_{\malpha\in \seq{0,1}^D}
	\left(\, \prod_{k:\, \alpha_k=0} 
	\int_{A_k} 1 \, dy_k \right) 
	\left(\, \prod_{k:\, \alpha_k=1}\int_{B_k} 
	\frac{1}{\left(y_k-2^j \tilde{x}_k\right)^{2 r}}
	\, dy_k\right)
	\\  \lesssim_D (C_{j,\epsilon,\eps}^{\mkappa})^r
	\sum_{\malpha\in \seq{0,1}^D}
	\left(\int_{\abs{\omega} \geq 1} \frac{1}{\omega^{2 r}}\, d\omega 
	\right)^{\abs{\malpha}}
	\lesssim (C_{j,\epsilon,\eps}^{\mkappa})^r .
	\label{eq:int-Y-y-est}
\end{multline}

Let us explicitly bound $C_{j,\epsilon,\eps}^{\mkappa}$ 
in terms of $j$, $\epsilon$, $\eps$, and $\abs{\mkappa}$. 
First, starting from \eqref{eq:C-kappa-def} and 
applying the Leibniz formula in $\mlambda$, together 
with the smoothness of $\rho$, we obtain
\begin{equation*}
	C_{j,\epsilon,\eps}^{\mkappa}
	\lesssim 
	\sup_{\substack{\mw \in K, 
	\, \abs{\meta_0} \in \supp(\tilde{\phi}), 
	\, \mlambda\in L \\
	\tilde \mbeta \in \seq{0,1,2}^D,
	\, \abs{\tilde \mbeta}\le 2D \\ 
	\tilde{\mkappa}\in \N_0^m, 
	\tilde{\mkappa}\le \mkappa}}
	\abs{\partial_{\meta_0}^{\tilde\mbeta}
	\pa_{\mlambda}^{\tilde{\mkappa}}
	\cL_{j,\epsilon}(\mw,\meta_0',\mlambda)},
	\quad 
	\meta_0'=\frac{\meta_0}{\abs{\meta_0}}.
\end{equation*}

To proceed, we recall the multivariate 
Fa{\`a} di Bruno (chain rule) formula 
for higher-order derivatives 
(see, e.g., \cite[Theorem 2.1]{Constantine:1996aa}).
Let $F:\R^D\to\R$ and $G:\R^{\bar{D}}\to\R^D$ 
be smooth functions, and define 
(we intentionally use non-bold $x$ 
to distinguish it from other variables)
$$
H(x)=F(G(x)).
$$ 
For any multi-index $\malpha\in\N_0^{\bar{D}}$ 
with total order $|\malpha|=n$, we have
$$
\partial^{\malpha}H(x)
=\malpha! \sum_{1\leq \abs{\mbeta}\leq n}(\partial^\mbeta F)(G(x))
\sum_{p(\malpha,\mbeta)} \prod_{i=1}^n
\frac{\prod_{k=1}^D \bigl(\partial^{\momega_i}G_k(x)
\bigr)^{(\mdelta_i)_k}}{(\mdelta_i!)(\momega_i!)^{|\mdelta_i|}},
$$
where the convention $0^0=1$ is used, and 
$p(\malpha,\mbeta)$ denotes the set of ordered $2n$-tuples
$(\mdelta_1,\dots,\mdelta_n,\momega_1,\dots,\momega_n)\in 
(\N_0^D)^n\times (\N_0^{\bar D})^n$; we refer to 
\cite[Remark 2.2(i)]{Constantine:1996aa} 
for the precise definition. 
Let us highlight two identities relevant to our analysis: 
$\sum_{i=1}^n \mdelta_i=\mbeta$ and 
$\sum_{i=1}^n\abs{\mdelta_i} \momega_i 
= \malpha$. 

Let us apply this formula with
$\mw$, $\mlambda$, and $\tilde{\mkappa}\leq\mkappa$
fixed, setting
$$
F(\meta_0)=\pa_{\mlambda}^{\tilde{\mkappa}}
\cL_{j,\epsilon}(\mw,\meta_0,\mlambda),
\quad 
G(\meta_0)=\meta_0'=\frac{\meta_0}{\abs{\meta_0}},
$$
where $G$ is homogeneous of order zero.
One can check that for any multi-index $\momega$,
$$
\abs{\partial^{\momega}G(\meta_0)}
\lesssim_{\abs{\momega}}
1/\abs{\meta_0}^{\abs{\momega}} 
\lesssim 1,
$$
assuming that $\abs{\meta_0}$ is bounded away from zero
(which is the case since $\abs{\meta_0}\in 
\supp\tilde\phi\subseteq (1/2,2)$;
see \eqref{eq:phi-j-def}).
Then by the Fa{\`a} di Bruno formula,
$$
\abs{\partial^{\tilde{\mbeta}}(F\circ G)(\meta_0)}
\mathrel{\lesssim}
\sum_{\substack{\mnu\in\N_0^D\\
0\leq\abs{\mnu}\leq\abs{\tilde{\mbeta}}}}
\abs{(\partial^{\mnu}
F)\bigl(G(\meta_0)\bigr)}.
$$
Upon reversing the substitutions,
we retrieve
$$
\abs{\pa^{\tilde \mbeta}_{\meta_0}
\pa_{\mlambda}^{\tilde{\mkappa}}
\cL_{j,\epsilon}(\cdot,\meta'_0,\cdot)}
\lesssim 
\sum_{\substack{\mnu\in\N_0^D\\
0\leq\abs{\mnu}\leq\abs{\tilde{\mbeta}}}}
\abs{\bigl(\pa^{\mnu}_{\meta_0}
\pa_{\mlambda}^{\tilde{\mkappa}}
\cL_{j,\epsilon} \bigr)(\cdot,\meta'_0,\cdot)}.
$$
Hence,
\begin{equation*}
	C_{j,\epsilon,\eps}^{\mkappa}
	\lesssim 
	\sup_{\substack{\mw \in K, 
	\, \meta_0\in \S^{D-1}, 
	\, \mlambda\in L \\
	\tilde \mbeta \in \N_0^D,
	\, \abs{\tilde \mbeta}\le 2D \\ 
	\tilde{\mkappa}\in \N_0^m, 
	\tilde{\mkappa}\le \mkappa}}
	\abs{\bigl(\partial_{\meta_0}^{\tilde\mbeta}
	\pa_{\mlambda}^{\tilde{\mkappa}}
	\cL_{j,\epsilon}\bigr)(\mw,\meta_0,\mlambda)}.
\end{equation*}

Next, let us apply the Fa{\`a} di Bruno formula with
$$
F(\meta_0)=\frac{\eta_0}{\abs{\meta_0}^2},
\quad 
G(\meta_0,\mlambda)
=\bigl(\eta_0+\mff_{\mw}(\mlambda) \cdot \meta,\,
2^{-\epsilon j/2}\meta\bigr),
$$
so that
\begin{align*}
	&\partial_{\meta_0}^{\malpha} 
	G(\meta_0,\mlambda)=0,
	\quad \abs{\malpha}\ge 2, \\
	&\partial_{\mlambda}^{\tilde{\malpha}}
	G(\meta_0,\mlambda)=
	\bigl(\partial_{\mlambda}^{\tilde{\malpha}}
	\mff_{\mw}(\mlambda) \cdot \meta,\, 0\bigr),
	\quad \abs{\tilde{\malpha}}\geq1, \\
	& \abs{\partial_{\meta_0}^{\malpha} 
	F(\meta_0)}\lesssim 
	1/\abs{\meta_0}^{\abs{\malpha}+1}.
\end{align*}
To obtain the last
estimate, note that for any $A>0$, 
$F(A\meta_0) = A^{-1} F(\meta_0)$, so differentiating 
gives $\partial_{\meta_0}^{\malpha} F(\meta_0) 
= A^{\abs{\malpha}+1} (\partial_{\meta_0}^{\malpha} F)(A\meta_0)$. 
Setting $A = \abs{\meta_0}^{-1}$ for $\meta_0 \neq 0$ yields 
$\partial_{\meta_0}^{\malpha} F(\meta_0)=
(\partial_{\meta_0}^{\malpha} F)(\meta_0')/
\abs{\meta_0}^{\abs{\malpha}+1}$, and since 
$(\partial_{\meta_0}^{\malpha} F)(\cdot)$ is bounded 
on $\S^{D-1}$, we obtain the estimate.

By ({\bf H2}), $\meta_0\in \S^{D-1}$, $\mw\in K$,
and $\mlambda\in L$, it holds that 
$\abs{\partial_{\meta_0}^{\malpha}
\partial_{\mlambda}^{\tilde{\malpha}}
G(\meta_0,\mlambda)}\lesssim 1$. 
Therefore, by the Fa{\`a} di Bruno formula, and 
since $\abs{G(\meta_0,\mlambda)}
\gtrsim 2^{-\epsilon j/2}$ (this can be shown 
using the procedure for $R_{1,2}$ 
in Subsection \ref{subsec:R1}),
\begin{align*}
	\abs{(\partial_{\meta_0}^{\tilde\mbeta}
	\pa_{\mlambda}^{\tilde{\mkappa}}
	\cL_{j,\epsilon})(\mw,\meta_0,\mlambda)}
	\lesssim 1/\abs{G(\meta_0,\mlambda)}^{2D+\abs{\mkappa}+1}
	\lesssim 2^{\epsilon j\left(D+\tfrac{\abs{\mkappa}+1}{2}\right)}.
\end{align*}

Summarizing, we have demonstrated that
$$
C_{j,\epsilon,\eps}^{\mkappa}
\lesssim 
2^{\epsilon j\left(D+\tfrac{\abs{\mkappa}+1}{2}\right)}.
$$
Now, combining the previous estimate 
with \eqref{eq:int-Y-y-est} 
and \eqref{eq:tildeM-est-tmp} yields
$$
\abs{\widetilde M_{j,\epsilon,\eps}
(\tilde \mx_0,\mlambda;\mz)}\lesssim 
2^{j D/r'}
2^{\epsilon j\left(D+\tfrac{\abs{\mkappa}+1}{2}\right)}
\norm{V}_{L^{r'}(\R^D)},
$$
for all $\tilde \mx_0\in \R^D$, $\mlambda \in \R^m$, 
and $\mz\in \R^{D-1}$. From this and \eqref{m-1}, for any 
$V\in C^\infty_c(\R^D)$ and any $\mz\in \R^{D-1}$, 
we arrive at
$$
\abs{\innb{M_{j,\epsilon,\eps}(\cdot,\mz),V}}
\lesssim
2^{j D/r'}
2^{\epsilon j\left(D+\tfrac{\abs{\mkappa}+1}{2}\right)}
\norm{\gamma}_{\cM(\R^D\times \R^m)}
\norm{V}_{L^{r'}(\R^D)}.
$$
This implies the sought-after claim: 
$M_{j,\epsilon,\eps}(\cdot,\mz) \in L^r(\R^D)$, 
and it is accompanied by the following explicit bound:
\begin{equation}\label{eq:M-measure-bound}
	\sup_{\mz\in \R^{D-1}}
	\norm{M_{j,\epsilon,\eps}(\cdot,\mz)}_{L^r(\R^D)}
	\lesssim 
	2^{j D/r'}
	2^{\epsilon j\left(D+\tfrac{\abs{\mkappa}+1}{2}\right)}
	\norm{\gamma}_{\cM(\R^D\times \R^m)}.
\end{equation}

Based on \eqref{eq:tildeR5-tmp}, $\tilde R_5$
can be understood as follows:
\begin{equation*}
	\innb{\tilde R_5,\bar{v}}
	=\int_{\R^{D-1}}\int_{\R^D}
	\cA_{\frac{\tilde\phi_j}{\abs{\cdot}}}
	\bigl(M_{j,\epsilon,\eps}
	(\cdot,\mz)\bigr)(\mx_0)
	\, \overline{v(\mx_0)}\, d\mx_0\,
	\mathcal{G}(\mz)\, d\mz.
\end{equation*}
The symbol
$\mxi_0\mapsto
\frac{\tilde\phi_j(\abs{\mxi_0})}{\abs{\mxi_0}}$
has the same form as the function
$\tilde\psi^{(0)}$ given
at the beginning of Subsection \ref{subsec:R4}%
, with $\phi$ replaced by $\tilde\phi$.
Hence,
$\norm{\cA_{\tilde\phi_j/\abs{\cdot}}}_{L^r\to L^r}
\lesssim 2^{-j}$.

By applying this bound along with \eqref{eq:M-measure-bound}, 
we conclude that
\begin{align}
	\abs{\innb{R_5,\bar{v}}} 
	& =\abs{
	\left\langle \frac{(-1)^{\abs{\mkappa}}}{2\pi i}
	\pi^{-1/2}\tilde R_5,\bar{v}\right\rangle}
	\notag \\ &
	\lesssim\int_{\R^{D-1}}
	2^{j D/r'} 2^{\epsilon j\left(D+\tfrac{\abs{\mkappa}+1}{2}\right)}
	2^{-j}\, \mathcal{G}(\mz)\, d\mz
	\norm{\gamma}_{\cM(\R^D\times \R^m)}
	\norm{v}_{L^{r'}(\R^D)}
	\notag \\ & 
	=2^{-j\left(1-\epsilon\left(D+\tfrac{\abs{\mkappa}+1}{2} 
	\right) - \frac{D}{r'}	\right)}
	\norm{\gamma}_{\cM(\R^D\times \R^m)}\norm{v}_{L^{r'}(\R^D)} \implies 
	\notag
	\\ & \norm{R_5}_{L^r(\R^D)} 
	\lesssim 2^{-j\left(1-\epsilon\left(D
	+\tfrac{\abs{\mkappa}+1}{2}
	\right)-\frac{D}{r'}	\right)}
	\norm{\gamma}_{\cM(\R^D\times \R^m)}.
	\label{eq:R5-final}
\end{align}

\subsection{Estimate of $I$}\label{subsec:I}

Next, we direct our focus to the term $I$ on the left-hand 
side of \eqref{eq:init-relation}. 
As previously discussed, to establish 
the theorem, we must verify that \eqref{eq:intro-goal} 
holds (via duality). In this context, the relevant 
quantity is not $I$, but rather
\begin{equation}\label{eq:wtilde-I-def}
	\innb{\widetilde I,\bar{v}}
	:=\int_{\R^D} \cA_{\phi_j}
	\left(\int_{\R^m} \rho(\mlambda)
	\tilde \varphi(\cdot) 
	h(\cdot, \mlambda) 
	\, d\mlambda \right)\!(\mx_0)
	\overline{v(\mx_0)} \, d\mx_0,
\end{equation}
which does not contain the function 
$\chi^{j,\eps}_{\my}(\cdot)$ defined in \eqref{eq:chi-jeps-y-def}. 
Here,  $v = v(\mx_0) \in C^\infty_c(\R^D)$ 
and $\varphi \in C^\infty_c(K)$ are arbitrary functions, with
\begin{equation*}
	\tilde\varphi (\mx_0)
	=\tilde\varphi (x_0,\mx)
	:=\pi^{-1/2}e^{-x_0^2}\varphi(\mx).
\end{equation*}
We will relate $I$ and $\widetilde{I}$ by estimating the 
difference between them. 

By applying the Plancherel theorem, 
\begin{align*}
	\innb{I,\bar{v}}
	& = \int e^{-2\pi i \mxi_0\cdot (x_0, \mx)}
	\, \phi_j(\abs{\mxi_0}) 
	\rho(\mlambda) \, \chi^{j,\eps}_{\my}(x_0,\mx)
	\\ & \quad  \qquad\qquad\qquad \times
	h(x_0,\mx,\mlambda)
	\overline{\widehat{v}(\mxi_0)} \varphi(\my)
	\, dx_0\, d\mx\, d\my\, d\mxi_0\, d\mlambda
	\\ & \overset{\eqref{eq:J-def-new}}{=} 
	\int
	e^{-2\pi i \mxi_0\cdot 
	\bigl(x_0,\my+2^{-\eps j}\mz\bigr)}
	\, \phi_j(\abs{\mxi_0})\rho(\mlambda) 
	\, \pi^{-1/2}e^{-x_0^2}\mathcal{G}(\mz)
	\\ & \quad \qquad\qquad\qquad \times
	h\bigl(x_0,\my+2^{-\eps j}\mz,\mlambda\bigr)
	\overline{\widehat{v}(\mxi_0)}
	\varphi(\my)
	\, dx_0\, d\mz\, d\my\, d\mxi_0\, d\mlambda
	\\ & = 
	\int
	e^{-2\pi i \mxi_0\cdot \tilde \my_0}
	\, \phi_j(\abs{\mxi_0})\rho(\mlambda) 
	\, \pi^{-1/2}e^{-\tilde y_0^2}\mathcal{G}(\mz) 
	\\ & \quad \qquad\qquad\qquad \times
	h(\tilde \my_0,\mlambda)
	\overline{\widehat{v}(\mxi_0)}
	\varphi\bigl(\tilde \my-2^{-\eps j}\mz\bigr)
	\, d\tilde \my_0\, d\mz\, d\mxi_0\, d\mlambda,
\end{align*}
where, in the second line, we applied 
a change of variables:
$$
\mz=\frac{\mx-\my}{2^{-\eps j}} 
\implies 
d\mz=2^{\eps j(D-1)}\, d\mx,
$$
and in the final step, we made 
another change of variables:
$$
\tilde{\my}_0
=(\tilde{y}_0,\tilde{\my})
=(x_0,\my +2^{-\eps j}\mz) 
\implies dx_0 \, d\my=d\tilde \my_0.
$$
We have suppressed the integration domain. 
For example, in the first line we have
$\mlambda \in L$, 
$\mxi_0\in \R^D$, 
$\my\in K$,  
$\mx_0=(x_0,\mx)\in K_0$.
However, we may freely assume that all integrals are taken
over the entire space, since the compact support of the
relevant functions ultimately localizes the integration
to a compact set where this is justified.

By writing $\tilde \varphi\bigl(\tilde y_0,\tilde \my-2^{-\eps j}\mz\bigr)
=\tilde \varphi(\tilde y_0,\tilde \my)+
\left(\tilde \varphi\bigl(\tilde y_0,\tilde \my-2^{-\eps j}\mz\bigr)
-\tilde \varphi(\tilde y_0,\tilde \my)\right)$ 
and again applying the Plancherel theorem, we obtain
\begin{align*}
	\innb{I,\bar{v}}= \int_{\R^D} \cA_{\phi_j}
	\left(\int_{\R^m} \rho(\mlambda)
	\tilde\varphi(\cdot) 
	h(\cdot, \mlambda) 
	\, d\mlambda \right)\!(\mx_0) 
	\overline{v(\mx_0)}\, d\mx_0+E
	=\innb{\widetilde{I},\bar{v}}+E,
\end{align*}
where we used that $\int \mathcal{G}(\mz)\, d\mz = 1$ 
and $\widetilde{I}$ is defined in \eqref{eq:wtilde-I-def}. 

The error term $E$ takes the form
\begin{align*}
	&E =\int_{\R^D} \cA_{\phi_j}
	\left(\int_{\R^m} \rho(\mlambda)
	\, \Delta^{\tilde \varphi}_{j,\eps}(\cdot) 
	\, h(\cdot, \mlambda) 
	\, d\mlambda \right)\!(\mx_0) 
	\overline{v(\mx_0)}\, d \mx_0,
	\\ & \Delta^{\tilde \varphi}_{j,\eps}(\mx_0)
	:=\int_{\R^{D-1}} \Bigl(
	\tilde \varphi\bigl( x_0, \mx-2^{-\eps j}\mz\bigr)
	-\tilde \varphi( x_0, \mx)\Bigr)
	\mathcal{G}(\mz)\, d\mz.
\end{align*}
By applying H\"older's inequality and 
using the uniform $L^r$-continuity of $\cA_{\phi_j}$ (see 
Section \ref{sec:prelim}), we obtain
\begin{align*}
	\abs{E} & \lesssim_{\rho} 
	\sup_{\mx_0\in K_0}
	\abs{\Delta^{\tilde \varphi}_{j,\eps}(\mx_0)}
	\norm{h}_{L^r(K_0\times L)} 
	\norm{v}_{L^{r'}(\R^D)} 
	\\ & \leq \sup_{\mx_0\in K_0} \int_{\R^{D-1}} 
	\abs{\tilde\varphi(x_0,\mx-2^{-\eps j}\mz)
	- \tilde\varphi(x_0,\mx)}
	\mathcal{G}(\mz) \, d\mz \ 
	\norm{h}_{L^r(K_0\times L)} 
	\norm{v}_{L^{r'}(\R^D)} 
	\\ & \lesssim_\varphi 2^{-\eps j}
	\int_{\R^{D-1}} \abs{\mz} \mathcal{G}(\mz) \,d\mz \
	\norm{h}_{L^r(K_0\times L)} 
	\norm{v}_{L^{r'}(\R^D)}
	\\ &
	\overset{\substack{\eqref{eq:J-def-new}
	\\ \eqref{eq:gaussian-q-moment}}}{\lesssim} 
	2^{-\eps j} 
	\norm{h}_{L^r(K_0\times L)} 
	\norm{v}_{L^{r'}(\R^D)}.
\end{align*}

In summary,
\begin{equation}\label{eq:wtilde-est-tmp}
	\abs{\innb{\widetilde{I},\bar{v}}} 
	\lesssim \abs{\innb{I,\bar{v}}}
	+2^{-\eps j}
	\norm{h}_{L^r(K_0\times L)}
	\norm{v}_{L^{r'}(\R^D)},
\end{equation}
where $\widetilde{I}$, as defined 
in \eqref{eq:wtilde-I-def}, is the 
object we aim to estimate. 
The term $I$ is given by \eqref{eq:init-relation} 
and includes the terms $R_1, \ldots, R_5$. 
Set
$u(\mx_0)=\pi^{-1/2}e^{-x_0^2}\varphi(\mx)
\innb{h(\mx_0,\cdot),\rho}$, so that
$\widetilde I=\cA_{\phi_j}u$ and $\supp u\subseteq K_0$.
Test \eqref{eq:wtilde-est-tmp} with functions supported in the
fixed neighborhood $W$ from Subsection~\ref{subsec:R1}, and use
its local $L^r$ estimate for $R_1$.
On $W^c$, Remark~\ref{rem:est_compact} gives, for every $N>0$,
\begin{equation*}
	\tag{\ref*{eq:wtilde-est-tmp}a}\label{eq:wtilde-tail}
	\norm{\widetilde I}_{L^r(W^c)}
	\lesssim_N 2^{-jN}\norm{u}_{L^r(\R^D)}
	\lesssim_{N,\rho,\varphi}2^{-jN}\norm{h}_{L^r(K_0\times L)}.
\end{equation*}
Combining \eqref{eq:wtilde-est-tmp} with the local estimate
\eqref{eq:R1-local} and the tail estimate \eqref{eq:wtilde-tail}
yields
\begin{equation}\label{eq:wtilde-est-new}
	\begin{split}
		\norm{\widetilde I}_{L^r(\R^D)}
		&\mathrel{\lesssim_W}
		\norm{R_1}_{L^2(\R^D)}
		+\sum_{i=2}^5\norm{R_i}_{L^r(\R^D)}
		\\ &\quad
		+(2^{-\eps j}+2^{-jN})
		\norm{h}_{L^r(K_0\times L)},
	\end{split}
\end{equation}
where the terms $R_1,\ldots,R_5$ have
previously been estimated by sums of
terms of the form $2^{-j(\cdots)}$.

\subsection{Concluding the proof}\label{subsec:final}

Fix
$0<\beta<\theta<\beta_0(r)$ and choose the proof parameters,
including $s<\bar s$ in \eqref{new-eq-33}, so that all decay
rates exceed $\theta$; their admissibility is established below.
Take $N>\theta$ in \eqref{eq:wtilde-est-new}.
For the compactly supported problem, put
$H=\norm{h}_{L^p(K_0\times L)}$ and
$G=\norm{\gamma}_{\cM(K_0\times L)}$.
For a sum of localized sources, use the sum of their variation
norms as $G$, as explained in Subsection~\ref{subsec:prep}.
Since $r<p$, the $L^r$ norm of $h$ is bounded by $CH$.
For $1<p<2$, summing the dyadic estimates in
\eqref{eq:wtilde-est-new} with weights $2^{j\beta}$ gives
$$
\norm{u}_{B^\beta_{r,1}(\R^D)}
\leq C\left(H^{p/2}+H^{p/r}+H+G\right).
$$
Here the low-frequency block is bounded by $CH$, and
$\sum_{j\geq1}2^{-j(\theta-\beta)}<\infty$.
All cutoffs are fixed independently of $h$ and $\gamma$.
If $A=H+G>0$, apply the argument to $(h/A,\gamma/A)$,
with the truncations applied afresh to the normalized solution.
The normalized norms are at most one, so rescaling and
\eqref{eq:besov_emb} give
$$
\norm{u}_{W^{\beta,r}(\R^D)}
\lesssim\norm{u}_{B^\beta_{r,1}(\R^D)}\leq C(H+G).
$$
The case $A=0$ is immediate.
For $p\geq2$, use $\norm{h}_{L^2(K_0\times L)}\leq CH$
and omit the truncations and $R_2$; the remaining estimates
are already linear in $H$ and $G$, with the original $p$
retained in the estimate of $R_3$.

Returning to the original solution, choose $\varphi=1$ near
the spatial projection of $\overline{U_0}$ and use the fixed
cutoff from Subsection~\ref{subsec:prep}.
On $U_0$, the weighted average is
$u=\pi^{-1/2}e^{-x_0^2}\innb{h,\rho}$.
Multiplication by the fixed smooth factor $\pi^{1/2}e^{x_0^2}$
on this bounded set, together with the localized source bound,
gives the quantitative estimate in Theorem~\ref{thm-1} and hence
the asserted local regularity.

\medskip

It remains to ensure that all exponents 
in terms of the form $2^{-j(\cdots)}$ are negative 
and to specify the admissible range of 
the integrability exponent $r>1$ and the upper 
bound $\beta_0$ of the regularity parameter. 
We do not compute an explicit value of $\beta_0$ 
in all cases: $\beta_0$ will be 
determined explicitly for $p\ge 2$, 
whereas for $p<2$ only an explicit 
lower bound for $\beta_0$ is provided.

Recall that these terms originate 
from \eqref{new-eq-31} and \eqref{comm0term} ($R_1$), 
\eqref{new-eq-32} ($R_2$), \eqref{new-eq-33} ($R_3$), 
\eqref{new-eq-34} ($R_4$), \eqref{eq:R5-final} ($R_5$), 
and \eqref{eq:wtilde-est-tmp}. 
The exponents depend on four strictly 
positive adjustable parameters: $\epsilon$, $\eps$, 
$\sigma$, and $\zeta$, and an additional one 
dictating the integrability 
$r\in \bigl(1,\frac{p\bar{p}}{p+\bar{p}}\bigr)$ 
(recall \eqref{eq:r-p-ass} and the discussion 
given in Subsection \ref{subsec:R3}). 
To ensure that the exponents are negative, 
these parameters must be selected such that 
each of the following expressions 
is strictly positive. The maximal 
regularity is then obtained as the maximum of 
the minimums of the eight functions below,
yielding a  max-min problem: 
\begin{equation}\label{eq:param-ineq}
  \begin{cases}
    \text{\eqref{new-eq-31}:}\quad
    \epsilon\frac{\alpha}{2}
      -\zeta\frac{\alpha}{2}
      -\sigma\Bigl(1-\frac{p}{2}\Bigr),
    \\[0.15cm]
    \text{\eqref{new-eq-31}:}\quad
    2\zeta-\epsilon-\sigma\Bigl(1-\frac{p}{2}\Bigr),
    \\[0.15cm]
    \text{\eqref{comm0term}:}\quad
    1-\varepsilon\frac{D+1}{2}
      -\frac{\epsilon}{2}
      -\sigma\Bigl(1-\frac{p}{2}\Bigr),
    \\[0.15cm]
    \text{\eqref{new-eq-32}:}\quad
    \sigma\Bigl(\frac{p}{r}-1\Bigr)
    -\varepsilon\,\frac{D-1}{r'},
    \\[0.15cm]
    \text{\eqref{new-eq-33}:}\quad
    \varepsilon\Bigl(\bar{s}-\frac{D-1}{r'}\Bigr)
    -\frac{\epsilon}{2},
    \\[0.15cm]
    \text{\eqref{new-eq-34}:}\quad
    1-\frac{\epsilon}{2}-\varepsilon\Bigl(\frac{D}{r'}
    +\frac{1}{r}\Bigr),
    \\[0.15cm]
    \text{\eqref{eq:R5-final}:}\quad
	1-\epsilon\Bigl(D+
	\tfrac{\abs{\mkappa}+1}{2}\Bigr)-\frac{D}{r'},
    \\[0.15cm]
    \text{\eqref{eq:wtilde-est-tmp}:}\quad
    \varepsilon,
  \end{cases}
\end{equation}
where we recall that $\frac{1}{r}+\frac{1}{r'} = 1$.
Here, $D \geq 2$ represents the total dimension 
(combining time and space). 
The numbers $p>1$, $\bar{s}\in (0,1]$, and $\alpha > 0$ 
are fixed parameters, as well as the multi-index $\mkappa$. 
These parameters are defined in ({\bf H1}), ({\bf H2}),
\eqref{eq-1}, and \eqref{non-deg}.
As a reminder, we are currently working under the 
assumption that $p < 2$ (see \eqref{eq:r-p-ass_2}). 
Furthermore, once the parameters are chosen, the (regularity) 
exponent $\beta_0$ is equal to the smallest 
value among all expressions 
appearing in \eqref{eq:param-ineq}. 
Consequently, the objective 
is to maximize this smallest value.

For ease of reference, we denote the eight lines 
in \eqref{eq:param-ineq} as $\eqref{eq:param-ineq}_1$ 
through $\eqref{eq:param-ineq}_8$.
The max-min optimization problem \eqref{eq:param-ineq} 
can then be expressed as follows: 
given $r > 1$, solve
\begin{equation}\label{eq:param-maxmin}
	\max_{\epsilon,\zeta,\sigma,\eps>0}\,
	\min_{i\in\{1,\dots,8\}}
	\bigl\{\eqref{eq:param-ineq}_i\bigr\} .
\end{equation}

Before proceeding, let us briefly note the 
simplifications that occur when $p\ge 2$. In this case, 
$h\in L^2_{\loc}(\R^{D+m})$, so the truncation functions 
$T_{l_j}$ and $T^{l_j}$ (see \eqref{eq:Tl-def}) are no 
longer needed. Consequently, in $R_1$ we simply have $h$ 
in place of $T_{l_j}(h)$, while $R_2$ vanishes entirely. 
As a result, the fourth line in \eqref{eq:param-ineq} 
is removed, and we set $\sigma=0$, with all other 
terms unchanged. Thus, for $p\ge 2$, 
\eqref{eq:param-maxmin} reduces to
\begin{equation}\label{eq:param-maxmin-L2}
	\max_{\substack{\epsilon,\zeta,\eps>0\\ \sigma=0}}
	\, \min_{i\in\{1,\dots,8\}\setminus\{4\}}
	\bigl\{\eqref{eq:param-ineq}_i\bigr\}.
\end{equation}

We collect the properties of \eqref{eq:param-maxmin} 
and \eqref{eq:param-maxmin-L2} in the following lemma. 

\begin{lemma}\label{lem:param}
Let $p>1$, $\bar{s}\in (0,1]$, 
$\alpha>0$, $\mkappa\in\N_0^m$,
and $D\geq 2$ be given. 
Define the number $r_1$ by
\begin{equation}\label{eq:r1-def}
	r_1:=\min\left\{\frac{D}{D-1},
	\frac{D-1}{D-1-\bar{s}},p\right\} .
\end{equation}

\noindent \textbf{Case $p\geq 2$.}  For any $r\in (1,r_1)$ 
a solution to \eqref{eq:param-maxmin-L2} exists,
denoted by $\beta_1(r)$. The function $\beta_1$ 
is strictly positive and decreasing on $(1,r_1)$, 
and it is given by:
\begin{equation*}
	\beta_1(r) = \frac{\alpha}{4+\alpha} \min\left\{
	\frac{\frac{D}{r}+1-D}{\frac{\alpha}{4+\alpha} 
	+ D + \frac{\abs{\mkappa}+1}{2}},
	\frac{2(4+\alpha)g(r)}{(4+3\alpha)(g(r)
	+\frac{D+1}{2})}\right\} ,
\end{equation*}
where $g(r)=\frac{D-1}{r}+\bar{s}-(D-1)$.

\medskip

\noindent \textbf{Case $p< 2$.} For any $r\in (1,r_1)$ 
a solution to \eqref{eq:param-maxmin} exists,
denoted by $\beta_2(r)$. The function 
$\beta_2$ is decreasing on $(1,r_1)$. 
Furthermore, $r_2$ given by
\begin{equation*}
	r_2 := \sup\left\{ r\in (1,r_1) :
	\frac{4\alpha(\bar{s}
	-\frac{D-1}{r'})(\frac{p}{r}-1)}
	{(4+\alpha)(2-p)} > \frac{D-1}{r'}
	\right\}
\end{equation*}
is well defined and satisfies 
$1<r_2\leq r_1$, where $1/r+1/r'=1$.
The function $\tilde\beta_2$, given by 
$$
\tilde\beta_2(r) 
=\frac{(1-\frac{p}{2})(1-\frac{D}{r'})
\left(\frac{2}{2-p}A_1(r)
-\frac{D-1}{r'}\right)}
{\bigl(A_1(r) + A_2(r) + A_3(r)\bigr)C(r)},
$$
where 
\begin{align*}
	& A_1(r)= \frac{2\alpha}{4+\alpha}
	\left(\bar{s}-\frac{D-1}{r'}\right)
	\left(\frac{p}{r}-1\right),
	\\ &
	A_2(r)=\left(1-\frac{p}{2}\right)
	\frac{D-1}{r'}(2D+\abs{\mkappa}),
	\\ &
	A_3(r)= 2\left(\bar{s}-\frac{D-1}{r'}\right)
	\left(D+\frac{\abs{\mkappa}+1}{2}\right)
	\left(\frac{p}{r}-\frac{p}{2}\right) ,
	\\ & 
	C(r)=\max\left\{1,
	\,\frac{3+\frac{3}{2(\bar{s}-\frac{D-1}{r'})}
	\left(D+1+\frac{(2-p)(D-1)(r-1)}{p-r}\right)
	+\frac{r(2-p)}{p-r}}{2D+\abs{\mkappa}+1}\right\},
\end{align*}
then obeys the relation
\begin{equation}\label{eq:beta2>tildebeta2}
	\beta_2(r) \geq \tilde{\beta}_2(r)>0, 
	\quad r\in (1,r_2).
\end{equation}
\end{lemma}

Before turning to the proof of Lemma~\ref{lem:param}, 
let us first complete the proof of Theorem \ref{thm-1}, 
that is, determine the integrability parameter $r_0$ 
and the regularity mapping
$\beta_0:(1,r_0)\to\R^+$.
In the preceding discussion, we reduced the 
existence of $\beta_0(r)$ to solving the 
max-min problems \eqref{eq:param-maxmin} 
and \eqref{eq:param-maxmin-L2}. Therefore, it remains only 
to state the explicit values of $r_0$ and $\beta_0$ 
provided by the lemma. We define 
$r_0$ and $\beta_0(r)$ as follows:
$$
r_0=\min\left\{r_i,
\frac{p\bar{p}}{p+\bar{p}}\right\},
\qquad
\beta_0(r)=\beta_i(r), 
\quad r\in (1,r_0),
$$
where $i=1$ if $p\geq 2$ and $i=2$ otherwise.
The values $r_i$ and $\beta_i$ are provided 
in Lemma \ref{lem:param}.

\begin{proof}[Proof of Lemma \ref{lem:param}]
Since all expressions in \eqref{eq:param-ineq} 
are continuous (indeed, affine) in the 
parameters $\epsilon$, $\zeta$, $\eps$, 
and $\sigma$, their minimum is also continuous. 
Moreover, without loss of generality, the maxima 
in \eqref{eq:param-maxmin} and \eqref{eq:param-maxmin-L2} 
can be taken over a compact set. 
In fact, the condition $\eqref{eq:param-ineq}_3\ge 0$ 
provides upper bounds for $\eps$, $\epsilon$, 
and $\sigma$, which, together with 
$\eqref{eq:param-ineq}_1\ge 0$, yield the same for $\zeta$. 
Hence, the solutions $\beta_1(r)$ and $\beta_2(r)$ 
to \eqref{eq:param-maxmin-L2} 
and \eqref{eq:param-maxmin}, 
respectively, exist.

The variable $r$ appears in the expressions 
$\eqref{eq:param-ineq}_i$ for $i=4,5,6,7$.
Since all four terms are clearly 
decreasing in $r$, we conclude that 
both $\beta_1$ and $\beta_2$ 
are also decreasing in $r$. 

It is useful to note that a 
necessary condition for $\eqref{eq:param-ineq}_4>0$, 
$\eqref{eq:param-ineq}_5>0$ and 
$\eqref{eq:param-ineq}_7>0$ is 
$r'>\max\{\frac{D-1}{\bar{s}},D\}$ and $r<p$. 
In terms of $r$ this reads as $r<r_1$, 
where $r_1$ is given in \eqref{eq:r1-def}. 
Note that since $D\geq 2$, we have $r_1\leq 2$.

Before splitting the analysis into 
two parts based on the value of $p$, 
let us simplify \eqref{eq:param-maxmin} 
and \eqref{eq:param-maxmin-L2}.
Since $\eqref{eq:param-ineq}_8$ is greater than
$\eqref{eq:param-ineq}_5$ (recall 
that $\bar{s}\leq 1$ and $D\geq 2$),
and $\eqref{eq:param-ineq}_6$ is greater 
than $\eqref{eq:param-ineq}_3$ (due to $r<r_1\leq 2$), 
we can remove indices $i=6$ and $i=8$ 
from the minimum in both \eqref{eq:param-maxmin} 
and \eqref{eq:param-maxmin-L2}.
Moreover, $\zeta$ appears only 
in $\eqref{eq:param-ineq}_1$ and
$\eqref{eq:param-ineq}_2$, and the 
two terms involving $\zeta$ have 
opposite signs. Thus, the optimal choice is to 
select its value such that the 
two expressions are equilibrated, i.e., 
$\eqref{eq:param-ineq}_1=\eqref{eq:param-ineq}_2$.
Then we get
\begin{equation}\label{eq:zeta}
	\zeta=\frac{\alpha+2}{4+\alpha}\epsilon.
\end{equation}
This satisfies $\epsilon/2<\zeta<\epsilon$.
Therefore, the index $i=2$ 
can also be eliminated from the
minimum in \eqref{eq:param-maxmin} 
and \eqref{eq:param-maxmin-L2}, while in the maximum 
we prescribe \eqref{eq:zeta}. 
This results in 
\begin{equation}\label{eq:param-ineq-l1}
	\text{$\eqref{eq:param-ineq}_1$}
	=\frac{\alpha}{4+\alpha}\epsilon
	-\sigma\Bigl(1-\frac{p}{2}\Bigr).
\end{equation}

\smallskip \noindent \textbf{Case $p\geq 2$.} Based on the 
previous observations, the max-min problem
\eqref{eq:param-maxmin-L2} is reduced to 
\begin{equation*}
	\max_{\substack{\epsilon,\eps>0\\ 
	\sigma=0, \eqref{eq:zeta}}}
	\,\min \Bigl\{\eqref{eq:param-ineq}_1,
	\eqref{eq:param-ineq}_3, \eqref{eq:param-ineq}_5, 
	\eqref{eq:param-ineq}_7\Bigr\}.
\end{equation*}
The parameter $\eps$ appears only in 
$\eqref{eq:param-ineq}_3$ and 
$\eqref{eq:param-ineq}_5$,
and the two terms involving $\eps$ enter with opposite signs.
Thus, the optimal value is chosen by equating 
the two expressions, yielding
\begin{equation}\label{eq:eps-L2}
	\eps = \frac{1}{\bar{s}+\frac{D+1}{2}
	-\frac{D-1}{r'}}.
\end{equation}
Since $r\in (1,r_1)$, we have 
$\frac{2}{2\bar{s}+D+1}<\eps<\frac{2}{D+1}$, and 
\begin{equation}\label{eq:param-ineq-L2-l5}
	\eqref{eq:param-ineq}_5
	=\frac{\bar{s}-\frac{D-1}{r'}}{\bar{s}
	-\frac{D-1}{r'}+ \frac{D+1}{2}}
	-\frac{\epsilon}{2}.
\end{equation}

We are left with the expressions 
$\eqref{eq:param-ineq}_i$ for $i=1,5,7$ (with $\sigma=0$), 
together with \eqref{eq:zeta} and \eqref{eq:eps-L2}, 
and only one unknown parameter $\epsilon$. 
The coefficient of $\epsilon$ is positive in 
$\eqref{eq:param-ineq}_1$ (see \eqref{eq:param-ineq-l1}), 
but negative in both $\eqref{eq:param-ineq}_5$
(see \eqref{eq:param-ineq-L2-l5}) 
and $\eqref{eq:param-ineq}_7$. 
Hence, the optimal value of
$\epsilon$ is attained at the
minimum of the solutions to the equations 
$\eqref{eq:param-ineq}_1=\eqref{eq:param-ineq}_5$ 
and $\eqref{eq:param-ineq}_1=\eqref{eq:param-ineq}_7$. 
Substituting this value into $\eqref{eq:param-ineq}_1$, 
we obtain finally
\begin{align*}
	\beta_1(r) &= \frac{\alpha}{4+\alpha} \min\left\{
	\frac{1-\frac{D}{r'}}{\frac{\alpha}{4+\alpha}
	+D+\frac{\abs{\mkappa}+1}{2}},
	\frac{2(4+\alpha)}{4+3\alpha}
	\frac{\bar{s}-\frac{D-1}{r'}}{\bar{s}-\frac{D-1}{r'}
	+ \frac{D+1}{2}}\right\} 
	\\ &
	= \frac{\alpha}{4+\alpha}
	\min\left\{\frac{\frac{D}{r}+1-D}{\frac{\alpha}{4+\alpha}
	+D + \frac{\abs{\mkappa}+1}{2}},
	\frac{2(4+\alpha)}{4+3\alpha}
	\frac{\frac{D-1}{r}+\bar{s}-(D-1)}{\frac{D-1}{r}
	+\bar{s}-(D-1)+ \frac{D+1}{2}}\right\}.
\end{align*}

\smallskip\noindent \textbf{Case $p<2$.} Let us recall that 
\begin{equation*}
	\beta_2(r)
	=\max_{\substack{\epsilon,\eps,\sigma>0
	\\ \eqref{eq:zeta}}}\,
	\min \Bigl\{\eqref{eq:param-ineq}_1, 
	\eqref{eq:param-ineq}_3, 
	\eqref{eq:param-ineq}_4, 
	\eqref{eq:param-ineq}_5, 
	\eqref{eq:param-ineq}_7\Bigr\}.
\end{equation*}
Deriving an explicit formula for $\beta_2(r)$ 
appears rather involved. We will therefore 
construct a lower bound $\tilde\beta_2(r)$ 
and analyze its sign instead. 
The idea is to balance the simplest expressions 
in \eqref{eq:param-ineq}, namely 
$\eqref{eq:param-ineq}_{4,5,7}$, and equate them 
with $\eqref{eq:param-ineq}_{1}$, while treating 
the remaining term separately. 

Let us take $r\in (1,r_1)$ and $\beta\in\R$,
and consider the following equalities:
\begin{equation*}
\text{\eqref{eq:param-ineq}$_4$}=\beta \;, \quad
	\text{\eqref{eq:param-ineq}$_5$}=\beta \;, \quad
	\text{\eqref{eq:param-ineq}$_7$}=\beta .
\end{equation*} 
By solving these equations we get
\begin{equation}\label{eq:epsilon_eps_sigma}
	\epsilon=\frac{1-\frac{D}{r'}
	-\beta}{D+\frac{\abs{\mkappa}+1}{2}}, 
	\quad
	\eps = \frac{\beta +\frac{\epsilon}{2}}
	{\bar{s}-\frac{D-1}{r'}},
	 \quad
	\sigma=\frac{\beta+\varepsilon
	\,\frac{D-1}{r'}}{\frac{p}{r}-1}.
\end{equation}
To ensure that all parameters, $\epsilon$, $\eps$ 
and $\sigma$, are positive, it suffices to have 
\begin{equation}\label{eq:beta_constr}
	0<\beta < 1 -\frac{D}{r'},
\end{equation}
recalling that $r$ is taken from the 
interval $(1,r_1)$.

We now impose the additional condition 
$\eqref{eq:param-ineq}_1=\beta$. 
Combined with \eqref{eq:epsilon_eps_sigma}, 
this yields an explicit formula for
$\beta$:
\begin{equation}\label{eq:param-beta}
	\beta=\frac{(1-\frac{p}{2})
	(1-\frac{D}{r'}) w(r)}
	{A_1 + A_2 + A_3} ,
\end{equation}
where 
\begin{align*}
	&w(r)= \frac{4\alpha}{(2-p)(4+\alpha)}
	\left(\bar{s}-\frac{D-1}{r'}\right)
	\left(\frac{p}{r}-1\right)-\frac{D-1}{r'},
	\\ & 
	A_1= \frac{2\alpha}{4+\alpha}\left(\bar{s}
	-\frac{D-1}{r'}\right)\left(\frac{p}{r}-1\right),
	\\ &
	A_2=\left(1-\frac{p}{2}\right)
	\frac{D-1}{r'}(2D+\abs{\mkappa}),
	\\ & A_3=2\left(\bar{s}-\frac{D-1}{r'}\right)
	\left(D+\frac{\abs{\mkappa}+1}{2}\right)
	\left(\frac{p}{r}-\frac{p}{2}\right).
\end{align*}

It remains to verify that the 
conditions \eqref{eq:beta_constr} are satisfied. 
Since $A_1,A_2, A_3>0$, the quantity 
$\beta$ is positive whenever $w(r)$ is positive. 
Observe that $w$ is continuous 
and strictly decreasing on $[1,r_1)$, with
$w(1)=\frac{4\alpha\bar{s}(p-1)}{(2-p)(4+\alpha)}>0$. 
Hence,
\begin{equation}\label{eq:r2}
	r_2:=\sup \bigl\{r\in (1,r_1):w(r)>0\bigr\}
\end{equation}
is well-defined, and for $r\in(1,r_2)$ 
we have $\beta>0$. 
For the upper bound
in \eqref{eq:beta_constr},
using
\eqref{eq:param-beta} and the definitions of $w(r)$ and $A_1$,
we obtain
\begin{equation*}
	\beta<\frac{(1-\frac{p}{2})
	\left(1-\frac{D}{r'}\right)
	\frac{2}{2-p} A_1}{A_1} 
	= 1-\frac{D}{r'}.
\end{equation*}

Therefore, for any $r\in(1,r_2)$, with $r_2$ 
defined in \eqref{eq:r2}, and for $\epsilon$, $\eps$, 
and $\sigma$ given by \eqref{eq:epsilon_eps_sigma}, 
where $\beta$ is defined in \eqref{eq:param-beta}, 
we have $\epsilon,\eps,\sigma>0$ and
$$
\eqref{eq:param-ineq}_1
=\eqref{eq:param-ineq}_4
=\eqref{eq:param-ineq}_5
=\eqref{eq:param-ineq}_7
=\beta>0.
$$

It remains to examine $\eqref{eq:param-ineq}_3$. 
In general, however, the inequality 
$\eqref{eq:param-ineq}_3\ge \beta$ does not hold, 
and it may even occur that 
$\eqref{eq:param-ineq}_3<0$. 
To avoid such a scenario, we 
consider the rescaled parameters:
$$
(\tilde\epsilon,
\tilde\varepsilon,
\tilde\sigma)
:=\frac{1}{C} (\epsilon,\varepsilon,\sigma),
$$
where $C\geq 1$ is to be determined. 
Note first that for these new parameters (which, of course, 
remain positive and thus admissible), we have
$$
\eqref{eq:param-ineq}_1
=\eqref{eq:param-ineq}_4
=\eqref{eq:param-ineq}_5
=\frac{\beta}{C}
\qquad
\hbox{and} 
\qquad 
\eqref{eq:param-ineq}_7
\geq \frac{\beta}{C}.
$$
Then, it is left to choose $C \geq 1$ 
such that $\eqref{eq:param-ineq}_3 
\geq \frac{\beta}{C}$, which is achieved if 
\begin{equation}\label{eq:C-ineq}
	C\geq \beta+\eps\frac{D+1}{2}
	+\frac{\epsilon}{2}
	+\sigma\left(1-\frac{p}{2}\right).
\end{equation}

To demonstrate the possibility of 
such a choice \eqref{eq:C-ineq}, 
we will determine upper 
bounds for $\beta$, $\epsilon$, 
$\eps$, and $\sigma$.  For $\beta$ we have
$$
\beta \leq \frac{(1-\frac{p}{2})\left(1-\frac{D}{r'}\right)
\frac{2}{2-p} A_1}{A_1 + A_3}
\leq \frac{1}{1+\frac{A_3}{A_1}}
<\frac{1}{1+\frac{4+\alpha}
{\alpha}\bigl(D+\frac{\abs{\mkappa}+1}{2}\bigr)}
<\frac{1}{D+\frac{\abs{\mkappa}+1}{2}} ,
$$
since $\frac{A_3}{A_1}
>\frac{4+\alpha}{\alpha}\bigl(D
+\frac{\abs{\mkappa}+1}{2}\bigr)$.
Using this bound and the fact that $\beta > 0$, we 
can easily derive the 
following upper bounds:
\begin{align*}
	& \frac{\epsilon}{2}
	\leq \frac{1}{2D+\abs{\mkappa}+1},
	\\ & 
	\eps\frac{D+1}{2}
	\leq \frac{3(D+1)}{(4D+2\abs{\mkappa}+2)
	\left(\bar{s}-\frac{D-1}{r'}\right)} , 
	\\ & 
	\sigma\left(1-\frac{p}{2}\right)
	\leq \frac{r(2-p)}{(p-r)(2D+\abs{\mkappa}+1)}
	+\frac{3(2-p)(r-1)(D-1)}{2(p-r)
	\left(\bar{s}-\frac{D-1}{r'}\right)
	(2D+\abs{\mkappa}+1)}.
\end{align*}
Combining these bounds, we deduce that 
\begin{align}\label{eq:C-final}
	C:=\max\left\{1, 
	\,\frac{3+\frac{3}{2(\bar{s}-\frac{D-1}{r'})}
	\left(D+1+\frac{(2-p)(D-1)(r-1)}{p-r}\right)
	+\frac{r(2-p)}{p-r}}{2D+\abs{\mkappa}+1}\right\}
\end{align}
satisfies \eqref{eq:C-ineq}. 

Therefore, as promised, we conclude that for 
every $r \in (1, r_2)$, the following 
inequality holds:
\begin{equation*}
	\beta_2(r)\geq \frac{\beta}{C} >0,
\end{equation*}
where $\beta$ and $C$ are given by 
\eqref{eq:param-beta} and \eqref{eq:C-final}. 
This proves \eqref{eq:beta2>tildebeta2}.
\end{proof}

\begin{remark}\label{rem:reg-exp}
We do not claim that the resulting regularity 
$\beta\in(0,\beta_0)$ is optimal---indeed, it is clearly 
suboptimal compared to the homogeneous case 
(see, e.g., \cite{DiPerna-etal:91,Lions:1994qy,Tadmor:2006vn}). 
Nevertheless, to the best of our knowledge, this is the 
first result treating the heterogeneous setting at this 
level of generality. Our exponent is dimension-dependent 
and substantially lower than those obtained in 
the cited homogeneous works. A similar observation 
appears when comparing with \cite{Gerard1990,Gerard:1992aa}; 
we refer to Section \ref{sec:intro} 
for a discussion of the differences in 
assumptions.

For instance, in the simpler scenario with parameters 
$D=2$, $\alpha=1/2$, $\abs{\mkappa}=1$, $p=2$, $\bar{s}=1$, 
and $\bar{p}=\infty$
(for the admissibility of these values of $\bar{s}$
and $\bar{p}$, see Remark \ref{rem:H2}),  we 
have $r_0=\tfrac{D}{D-1}=2$.
Then $r=1.8<2$ is admissible 
and %
$\beta_0(1.8)=\beta_1(1.8)
=\tfrac{1}{252}\approx0.003968$,
implying 
$\innb{h,\rho}\in W^{\beta, r}_{\loc}$ for any
$\beta<\tfrac{1}{252}$.
In contrast, \cite[Theorem A]{Lions:1994qy} provides, 
under analogous conditions, $r_0=1.8$ and 
$\beta_0\approx 0.1$.  
If we change $\bar{s}$ to $\bar{s}=0.1$, then 
$r_0=\tfrac{D-1}{D-1-\bar{s}}= \tfrac{10}{9}$.
For $r=\tfrac{15}{14}<r_0$ ($r'=15$) we get
$\beta_0(\tfrac{15}{14})
=\tfrac{1}{253}\approx0.003953$.

Since our primary goal was to establish some regularity in 
the open case involving heterogeneous coefficients, 
we have not explored whether the approach could yield better 
exponents. Investigating improved exponents, as 
well as extensions to more general equations including 
second-order operators and stochastic perturbations, 
is left for future research.
\end{remark}

\begin{remark}\label{rem:t-dep}
For simplicity of presentation, we 
have considered equations with $t$-independent 
drift $\mff=\mff(\mx,\mlambda)$, although 
this assumption is not essential. 
Relaxing this assumption opens 
the possibility of using the kinetic 
formulation of isentropic gas dynamics 
\cite{Lions:1994aa} to extract regularity 
from velocity averages in certain 
situations (see Subsection \ref{subsec:isentropic}).
\end{remark}

%%%%%%%%%%%%%%%%%%%%%%%%%%%
%%%%%%%%%%%%%%%%%%%%%%%%%%%
\section{Applications}\label{sec:appl}

In this section, we present two applications 
of Theorem \ref{thm-1} that highlight its impact. 
The first provides a new regularization 
result for scalar conservation laws with discontinuous 
flux in several spatial dimensions. This result also covers 
the heterogeneous case with smooth $x$-dependence, 
in the setting studied in \cite{Dalibard:2006ab}, 
where no such result was previously available. 
The second application establishes a partial 
quantitative regularization effect for the system 
of isentropic gas dynamics. To the best of our knowledge, 
this is also new.
 
\subsection{Heterogeneous conservation laws}

We will apply Theorem \ref{thm-1} to quantify 
the Sobolev regularity of solutions 
to scalar conservation laws with 
a nonlinear discontinuous flux $f(\mx,u)$ with bounded 
variation in $\mx$ and general 
initial data $u_0\in L^\infty$.  
A substantial body of literature addresses 
the regularizing effect in homogeneous conservation 
laws; an exhaustive review is beyond our scope. 
In addition to the references cited in the introduction 
(notably \cite{Lions:1994qy}), we refer the reader to 
\cite{Gess:2019ab,Golse:2013aa,Jabin:2010aa,
	Jabin:2002aa,Lellis:2003aa,Tadmor:2006vn}, 
the books \cite{Dafermos:2005mz,Perthame:2002qy}, and 
discussions therein. However, these works all 
consider flux functions independent 
of the spatial variable $\mx$.

In recent years, many researchers have dedicated 
significant effort to analyzing conservation laws with 
$\mx$-dependent flux functions. This effort has been partly 
motivated by numerous applications, often involving discontinuous 
$\mx$-dependence (see, e.g., \cite{Andreianov:2010fk} 
and references therein). 
Much of this work concentrates on well-posedness 
questions and on the convergence analysis 
of numerical schemes.
Regularity results for
one-dimensional conservation laws with $L^\infty$ initial data
have been obtained in \cite{Ghoshal:2024a,Ghoshal:2024b}
for fluxes that are independent of $x$ on either side of a single
fixed spatial discontinuity, under suitable assumptions
on the two fluxes.
The work \cite{Karlsen:2023ae}
addressed a specific one-dimensional conservation law with 
nonlinear flux $f(k(x),u)$, where $k$ is a spatial
$BV$ coefficient, without restriction
to a single interface,
and $u_0\in L^\infty$. The authors constructed solutions 
that belong to a Besov space.

Our aim is to establish a regularity result for
the so-called quasi-solutions of multidimensional 
scalar conservation laws with discontinuous flux.
We introduce the following definition.

\begin{definition}\label{def-qs}
We say that a function $u\in L^\infty(\R^+\times\R^d)$ 
is a quasi-solution to the scalar conservation law
\begin{equation}\label{scl-heter}
	\pa_t u(t,\mx) + \Div_\mx f(\mx,u(t,\mx)) = 0,
\end{equation}
if there exists a locally bounded Radon measure 
$\gamma\in \cM_{\loc}(\R^+\times\R^{d+1})$ 
such that, for almost every $\lambda\in\R$, 
the equation
\begin{equation}\label{quasi}
	\pa_t \abs{u-\lambda}
	+\Div_\mx \Bigl(\sign(u-\lambda)
	\bigl(f(\mx,u)-f(\mx,\lambda)\bigr)\Bigr)
	=\gamma
\end{equation}
holds in $\cDp(\R^+\times\R^d)$.
\end{definition}

We consider flux functions 
$f:\R^d\times\R\to\R^d$ 
in \eqref{scl-heter} satisfying
\begin{equation}\label{assum:mff}
	\begin{aligned}
		&\text{(for a.e.~$\mx$)}\quad 
		\lambda\mapsto f(\mx,\lambda)\in C^2(\R;\R^d),
		\\ & 
		\text{($\forall\lambda\in\R$)}\quad 
		\mx\mapsto f(\mx,\lambda)\ \text{and}\
		\mx\mapsto \pa_\lambda f(\mx,\lambda) 
		\ \text{belong to } BV(\R^d),
		\\ &\pa_\lambda f,\ \pa_{\lambda\lambda}f \in
		L^\infty_{\loc}(\R^d\times\R),
		\\ & 
		\text{($\forall$ compact $K\subset\R^d$, 
		$L\subset\R$)}\quad 
		\sup_{\lambda\in L}
		\abs{D\pa_\lambda f(\cdot,\lambda)}(K)<\infty.
	\end{aligned}
\end{equation}
Here $\abs{Dg}(K)$ denotes the total 
variation of $g$ on $K$.

We recall that any Cauchy problem 
associated with \eqref{scl-heter}, with $f$ 
satisfying \eqref{assum:mff}, whose combined 
vanishing viscosity and smooth flux approximation  
yields a bounded family of solutions that 
is strongly precompact in $L^1_{\loc}(\R^+\times\R^d)$, 
indeed admits a quasi-solution 
(see, e.g., \cite{Panov:2010aa}). Before establishing 
the regularity of such quasi-solutions 
(under a non-degeneracy condition), 
we need a technical lemma.

Assumption ({\bf H2}) ensures that drift 
vectors $\mff(\mx,\mlambda)$ with $BV$-regularity 
in $\mx$ are admissible, so in particular 
discontinuities in the spatial variable are 
allowed. This follows from standard properties of 
$BV$ functions (cf.~\cite[Theorem 14.9]{Leoni:2009aa}). 
For completeness, we provide 
the following elementary lemma.

\begin{lemma}\label{lem:BV_H2}
Let $K\subset\mathbb{R}^d$ be bounded 
and $0<s<1$, $q\ge 1$ with $sq<1$.  
If $V\in BV(K)\cap L^\infty(K)$, 
then $V\in W^{s,q}(K)$ and
\begin{equation}\label{eq:V-fractional}
	[V]_{W^{s,q}(K)}^{q}
	\le \frac{C_{d,K}
	\bigl(2\norm{V}_{L^\infty(K)}
	\bigr)^{q-1}}{1-sq}\abs{DV}(K_1)
	+C_{d,K}\,\bigl(2\norm{V}_{L^\infty(K)}
	\bigr)^{q},
\end{equation}
where
$$
[V]_{W^{s,q}(K)}^{q}
:=\iint_{K\times K}
\frac{\abs{V(\mx)-V(\my)}^{q}}{\abs{\mx-\my}^{d+sq}}
\,d\mx\,d\my,
$$
$\abs{DV}(K_1)$ denotes the total variation of $V$ on 
$K_1:=\seq{\mx:\mathrm{dist}(\mx,K)<1}$, 
and the constant $C_{d,K}$ 
depends only on $d$ and $K$.
\end{lemma}

\begin{proof}
Let $M=\norm{V}_{L^\infty(K)}$ and set $\sigma=sq\in(0,1)$.  
For $a,b\in[-M,M]$ we have 
$\abs{a-b}^q\le(2M)^{\,q-1}\abs{a-b}$. Hence
$$
[V]_{W^{s,q}(K)}^{q}
\le (2M)^{\,q-1}\iint_{K\times K}
\frac{\abs{V(\mx)-V(\my)}}{\abs{\mx-\my}^{d+\sigma}}
\,d\mx\,d\my.
$$
Split the integral into $\seq{\abs{\mx-\my}\le 1}$ 
and $\seq{\abs{\mx-\my}>1}$. For $\abs{\mx-\my}>1$, 
use that the integrand is bounded by 
$2M$ and $K$ is a bounded set. For 
$\abs{\mx-\my}\le 1$, write $\mh=\my-\mx$ and 
use the $BV$ translation estimate 
\cite[Theorem 13.48]{Leoni:2009aa}
$$
\int_K \abs{V(\mx+\mh)-V(\mx)}\,d\mx 
\le \abs{\mh}\abs{DV}(K_{\abs{\mh}})
\le \abs{\mh}\abs{DV}(K_1),
$$
which yields
$$
\iint_{\abs{\mx-\my}\leq 1}
\frac{\abs{V(\mx)-V(\my)}}
{\abs{\mx-\my}^{d+\sigma}}\,d\mx\,d\my
\le \abs{DV}(K_1)\int_{B(0,1)}
\frac{1}{|\mh|^{d+\sigma-1}}\,d\mh,
$$
where the last integral equals
$\abs{\bS^{d-1}}/(1-\sigma)=:C_d/(1-\sigma)$. 
Putting everything together gives 
\eqref{eq:V-fractional}.
\end{proof}

We have the following theorem.

\begin{theorem}\label{thm:LPT}
Let $u$ be a quasi-solution of
\eqref{scl-heter} in the sense of Definition~\ref{def-qs}.
Assume that the flux $f$ 
satisfies \eqref{assum:mff}. If the associated 
drift vector
$\mff:=\pa_\lambda f$
also satisfies the non-degeneracy 
condition \eqref{non-deg}, then there 
exist $\beta>0$ and $r>1$ such that 
$u\in W^{\beta,r}_{\loc}(\R^+\times \R^d)$.
\end{theorem}  

\begin{proof}
By differentiating \eqref{quasi} with 
respect to $\lambda$,  we conclude in a 
standard manner (see, e.g., \cite{Perthame:2002qy})  
that $\chi(t,\mx,\lambda):=\sign(u(t,\mx)-\lambda)$ 
satisfies the following kinetic equation  
in the sense of distributions:
\begin{equation*}
	\pa_t \chi
	+\Div_{\mx}\bigl(\mff(\mx,\lambda)\chi\bigr)
	=-\pa_\lambda \gamma,
\end{equation*}
where $\mff=\pa_\lambda f$ is the kinetic drift.

Assumption \eqref{assum:mff} guarantees that 
({\bf H2}) holds,  ({\bf H1}) is satisfied 
with $p=\infty$ (since $\abs{\chi}\le 1$),  
and ({\bf H3}) is trivially fulfilled.  
Since $\mff$ satisfies the non-degeneracy condition,  
Theorem \ref{thm-1} together 
with Lemma \ref{lem:BV_H2} yields that,  
for any $\rho\in C_c^1(\R)$,
\begin{equation}\label{reg-scl}
	(t,\mx)\mapsto
	\int_{\R}\chi(t,\mx,\lambda)\rho(\lambda)\, d\lambda 
	\in W_{\loc}^{\beta,r}((0,\infty)\times \R^d)
\end{equation}
for some $\beta>0$ and $r>1$.

If the
quasi-solution $u$ is
bounded by a constant $M$,  
and an even $\rho\in C_{c}^1(\R)$ is chosen
so that $\rho=1$ on $[-M,M]$,  
then \eqref{reg-scl} gives
$$
(t,\mx)\mapsto
\int_{\R}\chi(t,\mx,\lambda)
\rho(\lambda)\, d\lambda
= 2 u(t,\mx)
\in W_{\loc}^{\beta,r}(\R^+\times \R^d),
$$
which establishes the fractional 
Sobolev regularity of the  quasi-solution $u$ 
to the heterogeneous conservation law \eqref{scl-heter}.
\end{proof}

\subsection{Isentropic gas dynamics}\label{subsec:isentropic}

We will now apply our findings to the isentropic 
gas dynamics equations (see, for example, 
\cite[Chapter 17.8]{Dafermos:2016aa}). 
First, let us recall the following theorem:

\begin{theorem}[{\cite[Theorem 3]{Lions:1994aa}}]
\label{thm:gas-kinetic}
Define 
$$
p(\varrho)=\kappa \varrho^\gamma, 
\qquad 
\kappa=\frac{(\gamma-1)^2}{4\gamma},
$$
and suppose $\gamma > 1$. 
Let $(\varrho,\varrho u)\in
L^\infty(\R^+;L^1(\R;\R^2))$ 
have finite energy, i.e., $\varrho u^2+\varrho^\gamma
\in L^\infty(\R^+;L^1(\R))$%
, and $\varrho \ge 0$.
Then $(\varrho, u)$ is an entropy 
solution of the isentropic gas dynamics system
\begin{equation}\label{eq:isentropic}
	\begin{cases}
		\partial_t \varrho+\partial_x (\varrho u)=0, 
		\\
		\partial_t (\varrho u)
		+\partial_x \bigl( \varrho u^2
		+p(\varrho) \bigr)=0, \quad 
		t\ge 0,\,\, x\in \R,
	\end{cases}
\end{equation}
if and only if there exists a non-positive 
bounded measure $m$ on $\R^+ \times \R^2$ 
such that the function (weak 
entropy kernel)
$$
\tilde{\chi}(t,x,\lambda)
=\chi(\varrho;\lambda-u)
:=\bigl(\varrho^{\gamma-1}
-(\lambda-u)^2\bigr)_+^{\mu}, 
\qquad \mu
=\frac{3-\gamma}{2(\gamma-1)},
$$
satisfies (in the sense of distributions) 
the kinetic equation
\begin{equation}\label{eq:kinetic}
	\partial_t \tilde{\chi}
	+\partial_x \Bigl(
	\bigl(\theta \lambda
	+(1-\theta)u \bigr)\tilde{\chi}\Bigr)
	=\partial^2_\lambda m(t,x,\lambda),
\end{equation}
where $\theta=\frac{\gamma-1}{2}$ is a constant 
and
$(w)_+^{\mu}:=\max\bigl\{0,\abs{w}^{\mu-1}w\bigr\}$%
, with $(0)_+^\mu=0$.
Moreover, if $\varrho,\varrho u$ are $C^1$ in 
an open subset $\cO \subset \R^+ \times \R$, 
then $m=0$ for $\cO\times\R$.
\end{theorem}

We note that \eqref{eq:kinetic} 
generalizes the kinetic equation 
for quasi-solutions of scalar conservation laws. The key 
new feature is that the drift 
$\theta\lambda + (1-\theta)u$ depends 
on the unknown gas velocity $u=u(t,x)$, which 
may be highly irregular in $(t,x)$; this 
heterogeneity prevents the use of 
standard kinetic regularity techniques. 
Indeed, in \cite{Lions:1994aa}, the authors 
only establish strong compactness for bounded entropy 
solutions of \eqref{eq:isentropic}, 
but their result provides no quantitative 
regularity (except in the special case $\gamma=3$, where 
$\theta=1$ and the heterogeneity disappears). 
Using Theorem \ref{thm-1}, we prove that imposing 
a regularity condition on the velocity field $u$ (for 
instance, $u\in BV$) leads, via the kinetic 
formulation \eqref{eq:kinetic}, to 
quantitative regularity for the 
density $\varrho$ (in the case $\gamma\neq 3$).

\begin{theorem}\label{th:isentropic-regularity}
Assume that the pair $(u,\varrho)$ represents 
a locally bounded entropy solution to 
the isentropic gas dynamics system 
\eqref{eq:isentropic} with $\gamma\neq 3$ 
(see Theorem
\hyperref[thm:gas-kinetic]{%
	\ref*{thm:gas-kinetic}}),
and that $u\in W^{\beta,r}_{\loc}(\R^+\times \R)$
for some $\beta>0$ and $r>1$.
Then %
$\varrho,\ \varrho u\in
W_{\loc}^{\bar{\beta},\bar{r}}(\R^+\times \R)$ 
for some exponents $\bar{\beta}>0$ and $\bar{r}>1$. 
\end{theorem}

\begin{proof}
We rescale $\chi$, $\tilde\chi$, and $m$ by the same
positive constant $c_\gamma$, chosen below, and retain the notation.

The aim is to apply Theorem \ref{thm-1} 
to \eqref{eq:kinetic}, given
the regularity assumption on $u$.	
First, the drift
$v:=\theta\lambda+(1-\theta)u$ satisfies
\eqref{non-deg} with $\alpha=1$.
Indeed, fix $K\subset\subset\R^+\times\R$ and
$L\subset\subset\R$, and set
$F:=\norm{v}_{L^\infty(K\times L)}<\infty$ and
$\delta:=[2(1+F)]^{-1}$.
If $(\xi_0,\xi)\in\bS^1$ and $\abs{\xi}\leq\delta$, then
$\abs{\xi_0+\xi v}\geq\sqrt{1-\delta^2}-\delta F>1/4$,
so the sublevel set satisfies
$\meas\{\lambda\in L:\abs{\xi_0+\xi v}<\nu\}=0$
for $0<\nu\leq1/4$.
If $\abs{\xi}>\delta$, its measure is at most
$2\nu/(\theta\delta)$, since $\theta>0$ and the
slope in $\lambda$ is $\theta\xi$.
For $\nu\geq1/4$, use
$\meas(L)\leq4\meas(L)\nu$.
Thus, for every $\nu>0$,
$$
\esssup_{(t,x)\in K}
\sup_{(\xi_0,\xi)\in\bS^1}
\meas\{\lambda\in L:\abs{\xi_0+\xi v}<\nu\}
\leq C\nu,
$$
where $C=\max\{2/(\theta\delta),4\meas(L)\}$.

Since both $\varrho$ and $u$
are locally bounded,
the kinetic kernel
is locally bounded for $1<\gamma<3$, so ({\bf H1}) holds for any
$1<p<\infty$. For $\gamma>3$, take
$1<p<\theta/(\theta-1)$. The substitution
$\lambda-u=\varrho^\theta y$ gives, for $\varrho>0$,
$$
\int_\R\abs{\tilde\chi(t,x,\lambda)}^p\,d\lambda
=C_{\gamma,p}\varrho^{\theta+(1-\theta)p}.
$$
Here $C_{\gamma,p}<\infty$ since $\mu p>-1$, and
$\theta+(1-\theta)p>0$. At vacuum the kernel is zero a.e. in
$\lambda$. Thus ({\bf H1}) holds in this case as well.

Since $u$ is locally bounded,
fractional Sobolev interpolation gives
$u\in W^{s,q}_{\loc}(\R^+\times\R)$ for any finite $q>r$
and sufficiently small $s\in(0,1)$. Choose $q>\max\{r,p'\}$.
The affine drift and its velocity derivatives then satisfy ({\bf H2})
with $\bar s=s$ and $\bar p=q>p'$, while ({\bf H3}) follows
from the local finiteness of $m$.

By Theorem \ref{thm-1} (see also 
Remark \ref{rem:t-dep}), we obtain 
for any $\rho\in
C^{2}_c(\R)$ that
\begin{equation}\label{eq:reg-isent-1}
	(t,x)\mapsto
	 \int_{\R}\rho(\lambda)
	 \tilde{\chi}(t,x,\lambda)\,d\lambda
	 \in  W_{\loc}^{\bar{\beta},\bar{r}}
	 (\R^+\times \R).
\end{equation} 
Let us now show how this implies 
the assertion of the theorem.

At vacuum the kinetic
function vanishes, so it suffices to consider $\varrho>0$
in the moment calculation below.
Set $a:=\varrho^\theta$. 
The kinetic function 
can be expressed in the form
$$
\tilde{\chi}(t,x,\lambda)
=\chi(\varrho;\lambda-u)
= c_\gamma\Bigl(a^{2}
-(\lambda-u)^2\Bigr)_{+}^{\mu},
$$ 
where the normalization constant $c_\gamma$ 
is chosen so that the zeroth $\lambda$-moment 
(i.e., taking $\rho(\lambda)\equiv 1$ 
in \eqref{eq:reg-isent-1}) recovers 
the density $\varrho$ (see below).
Furthermore, since we assume that $\varrho$ and $u$ 
are locally bounded, for any 
$K \subset \subset \R^+
\times \R$, it must be that 
$$
\bigcup_{(t,x)\in K}
\supp\bigl(\tilde{\chi}(t,x,\cdot_\lambda)
\bigr)\subset (-R(K),R(K))
$$
for a positive constant $R$ depending on the set $K$. 
More precisely, since
$$
\tilde{\chi}(t,x,\lambda)\neq 0
\quad
\text{for $\lambda \in (u-a,u+a)$},
$$
the change of variables
$$
y=\frac{\lambda-u}{a}, 
\quad 
y\in[-1,1],
\quad 
d\lambda=a\,dy,
$$
yields, for any %
$\rho\in C^{2}_c(\R)$
with $\rho\equiv 1$ on $(-R(K),R(K))$ 
and for a.e.~$(t,x)\in K$, the zeroth moment
\begin{align*}
	\int_{\R} \rho(\lambda)
	\chi(\varrho;\lambda-u)\, d\lambda 
	& =\int_{\R} \chi(\varrho;\lambda-u)\, d\lambda
	=c_\gamma \, a \int_{-1}^1 
	\bigl(a^2(1-y^2)\bigr)^{\mu}\,dy 
	\\ & =c_\gamma a^{2\mu+1}
	\int_{-1}^1(1-y^2)^{\mu}\,dy.
\end{align*}
Since $2\mu+1 = 1/\theta$, 
we have $a^{2\mu+1}
=(\varrho^\theta)^{1/\theta} = \varrho$.
Denote
$$
B_{\mu}
:=\int_{-1}^1(1-y^2)^{\mu}\,dy
=\frac{\Gamma\left(\tfrac{1}{2}\right)
\Gamma(\mu+1)}
{\Gamma\left(\mu
+\tfrac{3}{2}\right)},
$$
where $\Gamma$ denotes the Euler gamma function, 
and $\Gamma\left(\tfrac{1}{2}\right)=\sqrt{\pi}$. 
Specifying $c_\gamma=\frac{1}{B_{\mu}}$, 
we find that for almost 
every $(t,x)\in K$,
$$
\int_{\R}\rho(\lambda)
\chi(\varrho; \lambda-u)
\,d\lambda = \varrho.
$$

By symmetry of $\tilde\chi$
about $\lambda=u$, for a.e.~$(t,x)\in K$ we also have
$$
\int_\R\lambda\rho(\lambda)
\tilde\chi(t,x,\lambda)\,d\lambda=\varrho u.
$$
Since $\lambda\mapsto\lambda\rho(\lambda)$ also belongs to $C_c^2(\R)$,
\eqref{eq:reg-isent-1} gives the asserted regularity of both
$\varrho$ and $\varrho u$ with the same exponents.
This completes the proof.
\end{proof}

%%%%%%%%%%%%%%%%%%%%%%%%%%%
%%%%%%%%%%%%%%%%%%%%%%%%%%%
\section*{Acknowledgement}

The research of ME is supported by
the Croatian Science Foundation (UIP-2025-02-1337)
and by the European Union -- NextGenerationEU
through the National Recovery and Resilience Plan 2021--2026
(IK IA 1.1.3. Impact4Math, institutional grant of University
of Zagreb, Faculty of Science). 
The research of KHK is supported by the Research Council 
of Norway (351123/NASTRAN).

%%%%%%%%%%%%%%
%REFERENCES
%%%%%%%%%%%%%%

\section*{Use of AI tools}
No AI tools were used in preparing the originally submitted
manuscript. During the revision of the manuscript, ChatGPT-6
and SIRI AI (macOS 27) were used 
for language improvements and text refinement.

\end{document}